\documentclass[aap]{imsart}

\makeatletter
\def\journal@name{The Annals of Applied Probability}
\makeatother

\RequirePackage{amsthm,amsmath,amsfonts,amssymb}
\usepackage{mathtools}
\usepackage{amsmath}
\usepackage{subcaption}
\usepackage{float}
\usepackage{algorithm}
\usepackage{algpseudocode}
\usepackage{enumitem}
\usepackage{mathrsfs}
\usepackage{longtable}
\RequirePackage[authoryear]{natbib}
\RequirePackage[colorlinks,citecolor=blue,urlcolor=blue]{hyperref}
\RequirePackage{graphicx}

\startlocaldefs
\numberwithin{equation}{section}
\theoremstyle{plain}
\newtheorem{theorem}{Theorem}[section]
\newtheorem{corollary}[theorem]{Corollary}
\newtheorem{proposition}[theorem]{Proposition}
\newtheorem{lemma}[theorem]{Lemma}
\theoremstyle{definition}
\newtheorem{definition}[theorem]{Definition}

\newtheorem{assumption}[theorem]{Assumption}
\newtheorem{remark}[theorem]{Remark}

\newcommand{\norm}[1]{\lVert#1\rVert}
\newcommand{\abs}[1]{\left\lvert#1\right\rvert}
\newcommand{\argmin}{\mathop{\mathrm{argmin}}\limits}
\newcommand{\scalar}[2]{\left\langle #1, #2 \right\rangle}
\renewcommand{\appendixname}{Supplementary Material}
\endlocaldefs

\newcommand{\imsreprintdoi}{10.1214/25-AAP2289}
\newcommand{\imsreprintyear}{2026}
\newcommand{\imsreprintvolume}{36}
\newcommand{\imsreprintissue}{3}
\newcommand{\imsreprintfirstpage}{2594}
\newcommand{\imsreprintlastpage}{2625}
\newcommand{\imsreprintdata}{%
  \imsreprintyear, Vol.~\imsreprintvolume, No.~\imsreprintissue,
  \imsreprintfirstpage--\imsreprintlastpage%
}

\doi{\imsreprintdoi}

\makeatletter
\providecommand{\owner@url}{https://www.imstat.org}

\setpkgattr{copyright}{size}{\fontsize{6}{7}\selectfont\raggedright}
\setpkgattr{copyright}{text}{%
  \url@fmt{}{\itshape}{The Annals of Applied Probability}{\journal@url}\break
  \imsreprintdata\break
  \url@fmt{DOI: }{}{\imsreprintdoi}{\doi@base\imsreprintdoi}\break
  $\copyright$~\url@fmt{}{}{Institute of Mathematical Statistics}{\owner@url},
  \imsreprintyear
}

\setpkgattr{arxiv}{text}{This is an electronic reprint of the original article published by the
                                  \url@fmt{}{}{Institute of Mathematical Statistics}{\owner@url}\ in 
                                  \url@fmt{}{\itshape}{The Annals of Applied Probability}{\journal@url},
                                  \url@fmt{}{}{\imsreprintdata}{\doi@base\imsreprintdoi}.
                                  This reprint differs from the original in pagination, typographic detail,
                                  and the inclusion of supplementary material.}
\setpkgattr{arxiv}{style}{\small\raggedright}

\def\arxivnotice@fmt{%
  {%
    \protected@xdef\@thefnmark{}%
    \safe@footnotetext{\vskip-3\p@\noindent\fbox{\parbox{\textwidth}{\arxiv@style\arxiv@text}}}%
  }%
}

\def\put@fmt@data{%
  \copyright@fmt
  \history@fmt
  \@thanks
  \keyword@fmt
  \arxivnotice@fmt
  \abstract@fmt}
\makeatother

\begin{document}

\begin{frontmatter}
\title{Deep Operator BSDE: a Numerical Scheme to Approximate Solution Operators}
\runtitle{A Numerical Scheme to Approximate Solution Operators}

\begin{aug}
\author[A]{\fnms{Pere}~\snm{Diaz Lozano}\ead[label=e1]{peredl@math.uio.no}}
\and
\author[A]{\fnms{Giulia}~\snm{Di Nunno}\ead[label=e2]{giulian@math.uio.no}}

\address[A]{Department of Mathematics,
University of Oslo, Oslo, Norway\printead[presep={,\ }]{e1,e2}}
\end{aug}

\begin{abstract}
Motivated by dynamic risk measures and conditional $g$-expectations, in this work we propose a numerical method to approximate the solution operator given by a Backward Stochastic Differential Equation (BSDE). The main ingredients for this are the Wiener chaos decomposition and the classical Euler scheme for BSDEs. We show convergence of this scheme under very mild assumptions, and provide a rate of convergence in more restrictive cases. We then implement it using neural networks, and we present several numerical examples where we can check the accuracy of the method.
\end{abstract}



\end{frontmatter}


\section{Introduction}

We are interested in the numerical approximation of the operator given by the solution of a backward stochastic differential equation (BSDE for short) of the following form:
\begin{flalign}\label{BSDE}
    Y_{t} = \xi + \int_{t}^{T} g(s, Y_{s}, Z_{s}) ds - \int_{t}^{T} Z_{s} \cdot dB_{s}, \quad t \in [0,T],
\end{flalign}
where $B$ is a $d$-dimensional standard Brownian motion, $\mathbb{F} = (\mathcal{F}_{t})_{t \in [0,T]}$ is its natural and augmented filtration, $\xi$ is an $\mathbb{R}$-valued $\mathcal{F}_{T}$-measurable random variable, and the generator $g \colon [0,T] \times \mathbb{R} \times \mathbb{R}^{d} \to \mathbb{R}$ is a measurable map. A solution to this equation is essentially defined to be a pair of adapted processes $(Y,Z) = (Y_{t}, Z_{t})_{t \in [0,T]}$ satisfying equation (\ref{BSDE}). BSDEs were introduced by \cite{BismutJean_Michel1973Ccfi} in the case of a linear generator, and were later generalized by \cite{PARDOUX199055} to the nonlinear case. Several applications of BSDEs in finance can be found for example in \cite{ElKarouiN.1997BSDE}, where they are used in pricing and hedging of financial products. We also mention the connection between BSDEs with generator $g$ and $g$-conditional expectations with respect to $\mathbb{F}$, which are defined as the collection of operators 
\begin{gather}\label{g expectation}
    \mathcal{E}^{g} = (\mathcal{E}_{t}^{g})_{t \in [0,T]}, \quad \mathcal{E}_{t}^{g} \colon \xi \mapsto Y_{t}^{\xi},
\end{gather}
where $Y_{t}^{\xi}$ denotes the solution to (\ref{BSDE}) with terminal condition $\xi$. See  \cite{Peng1997} for further details on $g$-expectations and \cite{RosazzaGianinEmanuela2006Rmvg} for the connection with dynamic risk measures.    

There is already an extensive research on the numerical approximation of (\ref{BSDE}), with most methods relying on an Euler-type discretization scheme on a time grid $\pi \coloneqq \{t_{i}\}_{0 \leq i \leq n} \subset [0,T]$, leading to the recursive formulas 
\begin{gather*}
    Z_{t_{i}}^{\pi} \coloneqq \frac{1}{\Delta t_{i}} \mathbb{E} \big[ Y_{t_{i+1}}^{\pi} \Delta B_{t_{i}}  \vert \mathcal{F}_{t_{i}} \big], \quad Y_{t_{i}}^{\pi} \coloneqq \mathbb{E} \big[ Y_{t_{i+1}}^{\pi} \vert \mathcal{F}_{t_{i}} \big] + \Delta t_{i} g(t_{i}, Y_{t_{i}}^{\pi}, Z_{t_{i}}^{\pi}) ,
\end{gather*}
which is called the backward Euler scheme for BSDEs. The main references in this topic are due to \cite{BouchardBruno2004DaaM} and \cite{zhang_bsde_method}, although the basic ideas date back to \cite{Chevance_1997}, who considered generators independent of $Z$, and \cite{Bally1997}, who considered a random time mesh.

In \cite{BouchardBruno2004DaaM}, the convergence of the backward Euler scheme was proved to be of order $1/2$, under the hypothesis that the terminal condition is a Lipschitz function of the final value of a forward SDE. In \cite{zhang_bsde_method}, the same order was obtained also for path-dependent functionals.

The critical challenge with the implementation of the scheme is the computation of conditional expectations at each time-step. Some approaches include methods based on least-squares regression (\cite{GobetEmmanuel2005ARMC}, \cite{LemorJean-Philippe2006RoCo}), Malliavin calculus (\cite{BouchardBruno2004OtMa}, \cite{GOBETEMMANUEL2016Aobs}), quantization techniques (\cite{Bally2003Aqaf}), tree-based methods (\cite{BriandPhilippe2021DtfB}) and cubature methods (\cite{CrisanD.2012Sbsd}). More recently, there has been a surge in the use of deep learning based models to estimate these conditional expectations, which outperform more classical approaches in high-dimensional settings, see \cite{HureCome2020Dbsf}. These can also be interpreted as nonlinear least-squares Monte Carlo. For a comprehensive survey on the topic of numerical methods for BSDEs, see \cite{ChessariJared2023Nmfb}. 

We now drive the attention to the fact that the methods available in the literature approximate (\ref{BSDE}) for a \textit{fixed terminal condition}, so that whenever we want to evaluate the solution at a different one, the proposed algorithm needs to be re-executed from scratch. Contrary to that, and motivated by (\ref{g expectation}), we propose a numerical scheme that approximates the operator given by the solution to the BSDE (\ref{BSDE}) with generator $g$, 
\begin{gather}\label{solution operator}
    \xi \mapsto (Y_{t}^{\xi}, Z_{t}^{\xi})_{t \in [0,T]},
\end{gather}
which will be called the \textit{solution operator.} 

\subsection*{Related work}

There is a growing literature on operator learning for semi-linear PDEs, which—via the nonlinear Feynman–Kac formula—are equivalent to Markovian Forward–Backward SDEs (FBSDE). These works seek to approximate the map that sends the drift $b$, diffusion $\sigma$, and terminal condition $\varphi$ to the solution field $u$, where $u\colon [0,T] \times \mathbb{R}^n \to \mathbb{R}$ satisfies 
\begin{flalign*}
        & \partial_t u + \frac{1}{2} \text{Tr} \big( \sigma \sigma^T D_x^2 u \big) + \scalar{b}{\nabla_x u} + g\big(t, x, u, \sigma^T \nabla_x u \big) = 0, \\
        & u(T,x) = \varphi(x).
\end{flalign*}
A common strategy is to identify each of $(b, \sigma, \varphi)$ by  its values on a fixed grid and feed that finite vector into the network. Representative methods include DeepONets \citep{deepOnet}, Integral‐Kernel Neural Operators \citep{anandkumar2019neural}, Fourier Neural Operators \citep{neuraloperators1}, and Physics‐Informed Neural Operators \citep{li2021fourier}. An alternative, the structure-informed approach of \cite{benth2024structureinformedoperatorlearningparabolic}, projects inputs onto a truncated functional basis and builds the operator network on those basis coefficients. For a comprehensive survey of operator learning for PDEs, see Section 4 in \cite{gonon2024overviewmachinelearningmethods}.

On the other hand, many works consider PDEs whose coefficients depend on a finite parameter vector $\mu$. They train a single network
\begin{gather*}
    u_\theta \colon (t,x,\mu) \mapsto  u(t,x;\mu),
\end{gather*}
where $\theta$ collects all model parameters. Some methods enforce the FBSDE representation to solve a statistical learning problem (e.g. \cite{NEURIPS2020_c1714160}, \cite{Sabate_Vidales04072021}, \cite{LRV}), while others minimize the PDE residual over $(t,x,\mu)$ directly (e.g. \cite{GLAU2022127355}, \cite{HUANG2024106354}). Once trained, these solvers instantly deliver both solution values and sensitivities for any $\mu$ in the training range, with no retraining per choice of parameters required.

In this context, our method is closest in spirit to the structure-informed framework of \cite{benth2024structureinformedoperatorlearningparabolic}, since we also approximate infinite-dimensional operators by truncating a suitable basis. However, after fixing a finite truncation, our method aligns with the approach in \cite{NEURIPS2020_c1714160}.

\subsection*{Contributions and paper organization}

In contrast to these PDE-based methods, our framework is purely probabilistic and works directly with the BSDE formulation. This lets us handle a broader class of terminal conditions, going beyond the Markovian FBSDE setting. To our knowledge, this is the first numerical scheme that approximates the solution operator of a BSDE. In this work we will present the framework, propose the algorithm and study its convergence properties.
 
 We now describe briefly the fundamental ingredients of our approach. The main tool is the Wiener chaos decomposition of $\xi \in L^{2}(\mathcal{F}_{T})$, 
 \begin{gather*}
       \xi  = \sum_{k \geq 0} \sum_{|a|=k} d_{a} \prod_{i \geq 1} H_{a_{i}}\Big( \int_{0}^{T} h_{i}(s) \cdot dB_{s} \Big), 
\end{gather*}
where $H_{n}$ denotes the Hermite polynomial of degree $n$, $(h_{i})_{i \geq 1}$ is an orthonormal basis of $L^{2}([0,T]; \allowbreak\mathbb{R}^{d})$, and $|a| = \sum_{i \geq 1} a_{i}$ for any sequence of non-negative integers $a = (a_{i})_{i \geq 1}$. 

After a brief review of background concepts in Section \ref{section 2}, the first part of the paper (Section \ref{section 3}) consists of two main results. We start by showing that the elements of the basis of the chaos decomposition satisfy an infinite-dimensional linear SDE, which is later used to write the solution to (\ref{BSDE}) for $\xi$ above as the solution to a Forward-Backward SDE (FBSDE) system
\begin{flalign*}
    X_{t} &= x_{0} + \int_{0}^{t} b(s,X_{s}) ds + \int_{0}^{t} \sigma(s, X_{s}) dB_{s} \\
    Y_{t} &= \sum_{k \geq 0} \sum_{|a|=k} d_{a} X_{T}^{a} + \int_{t}^{T} g(s, Y_{s}, Z_{s} ) ds - \int_{t}^{T} Z_{s} \cdot dB_{s},
\end{flalign*}
where the forward process $X$ is infinite-dimensional. We then show that there exist two measurable maps that provide a Markovian representation of the BSDE solution, that is such that $Y_{t} = u(t, X_{t})$ and $Z_{t} = v(t, X_{t})$. We later incorporate $\xi$ as an additional variable of these maps and prove joint measurability with respect to the product $\sigma$-algebra.

The second part of the paper (Section \ref{section 4}) deals with the numerical scheme. Motivated by classical backward Euler-type methods for BSDEs, we present an operator Euler scheme to approximate recursively the solution operator (\ref{solution operator}).

We then proceed to prove its convergence. Furthermore, we can obtain a rate of convergence under additional assumptions on the terminal conditions and the generator, exploiting some regularity results of \cite{zhang_bsde_method}, \cite{HuYaozhong2011MCFB}.  

We conclude the section by proving that these recursive operators can be approximated arbitrarily well by finite-dimensional maps, using this to show how to implement the scheme. To address the high dimensionality of the problem we make use of neural networks, resulting in an algorithm (Section \ref{section 5}) which we call Deep Operator BSDE. We conclude the paper with some numerical examples (Section \ref{section 6}), where we assess the performance of the method.

\clearpage
\section{BSDEs and their numerical approximation}\label{section 2}

\subsection{Definitions and notation}

Let us fix a complete probability space $(\Omega, \mathcal{F}, \mathbb{P})$ carrying an $\mathbb{R}^{d}$-valued Brownian motion $B = (B_{t})_{t \in [0,T]}$. We will make use of the following notation:

\begin{itemize}
    \item $\mathbb{F} = (\mathcal{F}_{t})_{ t \in [0,T]}$ the filtration generated by the Brownian motion $B$ and augmented with the $\mathbb{P}$-null sets;

    \item For any $p \geq 1$, $L^{p}(\mathcal{F}_{t}, \mathbb{R}^{n})$, $t \in [0,T]$, the space of all $\mathcal{F}_{t}$-measurable random variables $X\colon \Omega \to \mathbb{R}^{n}$ satisfying $\norm{X}_{L^{p}(\mathcal{F}_{t}, \mathbb{R}^{n})}^{p} \coloneqq \mathbb{E} \abs{X}^{p}  < \infty$;

    \item For any $X \in L^{1}(\mathcal{F}_{T})$ and $t \in [0,T]$, we denote $\mathbb{E}[X \vert \mathcal{F}_{t}]$ by $\mathbb{E}_{t}[X]$;

    \item $\mathbb{H}_{T}^{2}(\mathbb{R}^{n})$ the space of all $\mathbb{F}$-predictable processes $\varphi \colon \Omega \times [0,T] \to \mathbb{R}^{n}$ such that $\norm{\varphi}_{\mathbb{H}_{T}^{2}(\mathbb{R}^{n})}^{2} \coloneqq  \mathbb{E} \Big[ \int_{0}^{T} \vert \varphi_{t} \vert^{2} dt  \Big] < \infty$;

    \item $\mathbb{S}_{T}^{2}(\mathbb{R}^{n})$ the space of all $\mathbb{F}$-predictable and continuous processes $\varphi \colon \Omega \times [0,T] \to \mathbb{R}^{n}$ such that $\norm{\varphi}_{\mathbb{S}_{T}^{2}(\mathbb{R}^{n})}^{2} \coloneqq  \mathbb{E} \Big[ \sup_{t \in [0,T]} \vert \varphi_{t} \vert^{2}   \Big] < \infty$.
\end{itemize}

Whenever the dimension is clear, to ease the notation we omit the dependence on $\mathbb{R}^{n}$.

\subsection{Backward stochastic differential equations}

The precise meaning of a solution to a BSDE is clarified in the following definition.

\begin{definition}
    Let $g \colon [0,T] \times \mathbb{R} \times \mathbb{R}^{d} \to \mathbb{R}$ be measurable, and let $\xi: \Omega \to \mathbb{R}$ be a $\mathcal{F}_{T}$-measurable random variable. We say that a pair of stochastic processes $(Y,Z)$ is a solution of the BSDE with data $(\xi, g)$ if the following conditions are met:
\begin{enumerate}[label=(\roman*)]
    \item equation (\ref{BSDE}) is satisfied $\mathbb{P}$-a.s.;
    \item $Y$ is continuous $\mathbb{P}$-a.s.;
    \item $(Y,Z) \in \mathbb{H}_{T}^{2}(\mathbb{R}) \times \mathbb{H}_{T}^{2}(\mathbb{R}^{d})$.
\end{enumerate}
\end{definition}

The most standard set of assumptions on the generator and the terminal condition under which (\ref{BSDE}) has a unique solution is the following. 

\begin{assumption}\label{assumption 1}

\begin{enumerate}[label=(\roman*)]
    \item  $g: [0,T] \times \mathbb{R} \times \mathbb{R}^{d} \to \mathbb{R}$ is measurable and uniformly Lipschitz in $(y,z)$, i.e. there exists a constant $[g]_{L} > 0$ such that $dt$-a.e.
    \begin{gather*}
        \vert g(t, y_{1}, z_{1}) - g(t, y_{2}, z_{2}) \vert \leq [g]_{L} (\vert y_{1}-y_{2} \vert + \vert z_{1} - z_{2} \vert ) \quad \forall (y_{1}, z_{1}), (y_{2}, z_{2});
    \end{gather*}

    \item $\int_{0}^{T} \vert g(t, 0, 0) \vert^2 dt  < \infty$;

    \item $\xi \in L^{2}(\mathcal{F}_{T})$.
\end{enumerate}
\end{assumption}

Let us now state the well-known result on the well-posedness of (\ref{BSDE}), plus some a priori estimates of the solution. See e.g. Theorem 4.2.1 and Theorem 4.3.1 in \cite{ZhangJianfeng2017BSDE} for the proof.

\begin{theorem}\label{theorem 2}
    Under Assumption \ref{assumption 1}, the BSDE (\ref{BSDE}) has a unique solution $(Y,Z)$ in $\mathbb{H}_{T}^{2}(\mathbb{R}) \times \mathbb{H}_{T}^{2}(\mathbb{R}^{d})$, which actually belongs to $\mathbb{S}_{T}^{2}(\mathbb{R}) \times \mathbb{H}_{T}^{2}(\mathbb{R}^{d})$. Moreover, we have the following estimate:
    \begin{gather*}
        \norm{Y}_{S_{T}^{2}(\mathbb{R})} + \norm{Z}_{H_{T}^{2}(\mathbb{R}^{d})} \leq C \Big\{ \norm{\xi}_{L^{2}(\mathcal{F}_{T})}  + \int_{0}^{T} \vert g(t, 0, 0) \vert dt  \Big\},
    \end{gather*}
    where $C$ is a constant depending only on $T$, $[g]_{L}$ and $d$.
\end{theorem}

An immediate consequence is the following result, which will be crucial in the sequel. See e.g. \cite[Theorem 4.2.2]{ZhangJianfeng2017BSDE}.

\begin{corollary}\label{corollary 1 a priori bounds}
Let $g$ satisfy Assumption \ref{assumption 1} (i)-(ii), and for $i=1,2$, let $\xi_{i}$ satisfy Assumption \ref{assumption 1} (iii). Let $(Y^{i}, Z^{i}) \in \mathbb{S}_{T}^{2}(\mathbb{R}) \times \mathbb{H}_{T}^{2}(\mathbb{R}^{d})$ be the solution to the BSDE with data $(\xi^{i}, g)$. Then there exists a constant $C > 0$, depending only on $T$, $[g]_{L}$ and $d$, such that
    \begin{gather*}
        \norm{Y^{1} - Y^{2}}_{S_{T}^{2}(\mathbb{R})} + \norm{Z^{1} - Z^{2}}_{H_{T}^{2}(\mathbb{R}^{d})} \leq C \norm{\xi^{1}-\xi^{2}}_{L^{2}(\mathcal{F}_{T})}  .
    \end{gather*}
\end{corollary}

Thanks to Theorem \ref{theorem 2}, we have that for a fixed generator $g$, the solution operator (\ref{solution operator}) is well-defined.

\begin{definition}\label{solution operators}
    Let $g$ satisfy Assumption \ref{assumption 1} (i)-(ii). We define the solution operators of the associated BSDE as 
    \begin{gather*}
    \begin{array}{rcccc}
    \mathcal{Y} \colon & L^{2}(\mathcal{F}_{T}) & \to & \mathbb{S}_{T}^{2}(\mathbb{R})  \\
    & \xi & \mapsto & Y^{\xi}
\end{array} , \quad \begin{array}{rcccc}
    \mathcal{Z} \colon & L^{2}(\mathcal{F}_{T}) & \to &\mathbb{H}_{T}^{2}(\mathbb{R}^{d}) \\
    & \xi & \mapsto & Z^{\xi},
\end{array} 
\end{gather*}
where $(Y^{\xi}, Z^{\xi})$ is the solution to the BSDE with data  $(\xi, g)$. For any $t \in [0,T]$, we also use the notation $(\mathcal{Y}_{t}, \mathcal{Z}_{t}): \xi \mapsto  (Y_{t}^{\xi}, Z_{t}^{\xi})$. 
\end{definition}

We stress the relationship between the $g$-conditional expectations defined by (\ref{g expectation}) and the solution operator $\mathcal{Y}$. Indeed, for all $t$ we have that $\mathcal{E}_{t}^{g}(\cdot) = \mathcal{Y}_{t}(\cdot)$.

The a priori estimates given in Corollary \ref{corollary 1 a priori bounds} tell us that the solution operators are globally Lipschitz.

\begin{corollary}\label{corollary lipschitz}
    Let $g$ satisfy Assumption \ref{assumption 1} (i)-(ii). Then the corresponding solution operators are globally Lipschitz.
\end{corollary}

\subsection{Markovian Forward-Backward SDEs}

A particular case which is often studied in the literature is when the final condition of the BSDE is given by some sufficiently integrable function of the final value of the solution of a Forward SDE,  
\begin{flalign}
    X_{t} &= x + \int_{0}^{t} b(s, X_{s}) ds + \int_{0}^{t} \sigma(s, X_{s}) d B_{s}, \quad x \in \mathbb{R}^{m} \label{FSDE system} \\
    Y_{t} &= f(X_{T}) + \int_{t}^{T} g(s, Y_{s}, Z_{s}) ds - \int_{t}^{T} Z_{s} \cdot dB_{s} \label{BSDE system}.
\end{flalign}
These kind of systems are called decoupled Markovian Forward-Backward Stochastic Differential Equations (FBSDEs). Comparing it with (\ref{BSDE}), here we have that $\xi = f(X_{T})$.

One of the reasons why this case is interesting is because of the following result. See e.g. \cite[Theorem 3.4]{1997Bsde} or \cite[Theorem 5.1.3]{ZhangJianfeng2017BSDE} for the proof. 

\begin{theorem}\label{theorem markovianity}
    Let $b, \sigma \colon [0,T] \times \mathbb{R}^{m} \to \mathbb{R}^{m} \times \mathbb{R}^{m \times d}$ be deterministic maps, uniformly Lipschitz continuous w.r.t. $x$ and such that $b(\cdot, 0)$ $\sigma(\cdot, 0)$ are bounded. 
    
    Let $g$ satisfy Assumption \ref{assumption 1} (i)-(ii), and $f\colon\mathbb{R}^{m} \to \mathbb{R}$ be of polynomial growth. 

    Then there exists a unique solution $(X,Y,Z) \in \mathbb{S}_{T}^{2}(\mathbb{R}^{m}) \times \mathbb{S}_{T}^{2}(\mathbb{R}) \times \mathbb{H}_{T}^{2}(\mathbb{R}^{d})$. Moreover, there exist two measurable maps $u\colon[0,T] \times \mathbb{R}^{m} \to \mathbb{R}$ and $v\colon[0,T] \times \mathbb{R}^{m} \to \mathbb{R}^{d}$ such that 
    \begin{gather*}
        Y_{t} = u(t, X_{t}), \quad Z_{t} = v(t, X_{t}) \quad \text{$dt \otimes \mathbb{P}$-a.e.}
    \end{gather*}
\end{theorem}

\begin{remark}
    Recall that $Y \in \mathbb{H}_{T}^{2}(\mathbb{R})$, which together with Theorem \ref{theorem markovianity} implies  
    \begin{gather}\label{Y_U_H2}
       \int_{[0,T] \times \Omega} |u(t, X_{t}(\omega))|^{2} dt \otimes \mathbb{P}(dw) < \infty.
    \end{gather}
    Let us consider the measure space $([0,T] \times \Omega, \mathcal{B}([0,T]) \otimes \mathcal{F}, dt \otimes \mathbb{P})$, and let $\nu$ be the pushforward measure on $([0,T] \times \mathbb{R}^{m}, \mathcal{B}([0,T]) \otimes \mathcal{B}( \mathbb{R}^{m}))$ given by the measurable map 
    \begin{gather*}
        \begin{array}{rcccc}
       \widehat{X} \colon & [0,T] \times \Omega & \to & [0,T] \times \mathbb{R}^{m} & \\
        & (t, \omega) & \mapsto & (t, X_{t}(\omega)).
    \end{array} 
    \end{gather*}
    By the image measure theorem applied to (\ref{Y_U_H2}), one gets 
    \begin{gather*}
        \int_{[0,T] \times \mathbb{R}^{m}} |u(t, x)|^{2} \nu(dt, dx) < \infty.
    \end{gather*}
    We therefore have that $u \in L^{2}([0,T] \times \mathbb{R}^{m}, \mathbb{R}; \nu)$. Similarly, $v \in L^{2}([0,T] \times \mathbb{R}^{m}, \mathbb{R}^{d}; \nu)$.
\end{remark}

\subsection{The backward Euler scheme for BSDEs}\label{subsection b euler}

We now briefly present the backward Euler scheme for BSDEs, which is the most standard method used to numerically approximate the solution to (\ref{BSDE}) for a fixed $\xi$. This can be seen as an analog of the forward Euler scheme. We emphasize that this is always presented for a fixed terminal condition. We follow \cite[Section 5.3.2]{ZhangJianfeng2017BSDE}. 

Consider a time partition of $[0,T]$, $\pi \coloneqq \{0 = t_{0} < \dots < t_{n} = T\}$. Let us write $\Delta t_{i} \coloneqq t_{i+1}-t_{i}$ and $\Delta B_{i} \coloneqq B_{t_{i+1}} - B_{t_{i}}$. Then, in order to motivate the scheme, notice that
\begin{flalign}
    Y_{t_{i}} &= Y_{t_{i+1}} + \int_{t_{i}}^{t_{i+1}} g(t, Y_{t}, Z_{t}) dt - \int_{t_{i}}^{t_{i+1}} Z_{t} \cdot dB_{t} \notag \\
    & \approx Y_{t_{i+1}} + \Delta t_{i} g(t_{i}, Y_{t_{i}}, Z_{t_{i}}) - Z_{t_{i}} \cdot \Delta B_{i}. \label{bsde euler}
\end{flalign}
By taking conditional expectations with respect to $\mathcal{F}_{t_{i}}$ we get 
\begin{gather*}
    Y_{t_{i}} \approx \mathbb{E}_{t_{i}} \big[ Y_{t_{i+1}} \big] + \Delta t_{i} g(t_{i}, Y_{t_{i}}, Z_{t_{i}}).
\end{gather*}
Multiplying by $\Delta B_{t_{i}}$, taking conditional expectations with respect to $\mathcal{F}_{t_{i}}$ and using Itô's isometry, yields 
\begin{gather*}
    Z_{t_{i}} \approx \frac{1}{\Delta t_{i}} \mathbb{E}_{t_{i}} \Big[ Y_{t_{i+1}} \Delta B_{i} \Big].
\end{gather*}
This motivates the numerical scheme given by $Y_{n}^{\pi} \coloneqq \xi$ and, for $i=n-1, \dots, 0$,
\begin{gather}\label{backward sde euler}
    Z_{i}^{\pi} \coloneqq \frac{1}{\Delta t_{i}} \mathbb{E}_{t_{i}} \Big[ Y_{i+1}^{\pi} \Delta B_{i} \Big], \quad Y_{i}^{\pi} \coloneqq \mathbb{E}_{t_{i}} \big[ Y_{i+1}^{\pi} \big] + \Delta t_{i} g(t_{i}, Y_{i}^{\pi}, Z_{i}^{\pi}).
\end{gather}
In case $\xi$ cannot be perfectly simulated, one replaces it with an approximation.
\begin{remark}
    The scheme given by (\ref{backward sde euler}) is called \textit{implicit scheme}, because at each time step, $Y_{i}^{\pi}$ is defined implicitly through an equation. Since the mapping $y \mapsto \Delta t_{i} g(t,y,z)$ has Lipschitz constant $\Delta t_{i} [g]_{L}$, we have that for a sufficiently fine grid, the second equation in (\ref{backward sde euler}) has a unique solution, and therefore $Y_{i}^{\pi}$ is well defined.

    An alternative to the implicit scheme is the explicit scheme, defined by replacing the second equation with
    \begin{gather*}
    Y_{i}^{\pi} \coloneqq \mathbb{E}_{t_{i}} \Big[ Y_{i+1}^{\pi} + \Delta t_{i} g(t_{i}, Y_{i+1}^{\pi}, Z_{i}^{\pi}) \Big].
\end{gather*}
\end{remark}
In the sequel we always use the implicit scheme, although the results that we present can be proved in the same way for the explicit one.

\subsubsection{Implementation of the backward Euler scheme for BSDEs}

The keystone in the implementation of this scheme is the computation of the conditional expectations appearing in each time step. In the case of a Markovian FBSDE system, there is an analog result to Theorem \ref{theorem markovianity} for $(Y_{i}^{\pi}, Z_{i}^{\pi})_{0 \leq i \leq n}$, which shows that for each $i \in \{0, \dots, n\}$, there exist measurable transformations $u_{i}, v_{i}\colon \mathbb{R}^{m} \to \mathbb{R} \times \mathbb{R}^{d}$ such that 
\begin{gather*}
    Y_{i}^{\pi} = u_{i}(X_{i}^{\pi}), \quad Z_{i}^{\pi} = v_{i}(X_{i}^{\pi}),
\end{gather*}
where $X^{\pi}$ denotes some approximation of the forward diffusion. Moreover, we have that $(u_{i}, v_{i})\in L^{2}(\mathbb{R}^{m}, \mathbb{R} ; \nu_i) \times L^{2}(\mathbb{R}^{m}, \mathbb{R}^d ; \nu_i)$, where $\nu_i$ is the law of $X_{i}^{\pi}$. 

This result is easily proved by backward induction. We emphasize that the Markov property of the forward process and the fact that the final condition of the BSDE only depends on the final value of the SDE are crucial elements here.

By the $L^{2}$-projection rule of conditional expectations, we can then write 
\begin{flalign*}
     u_{i} &\coloneqq \argmin_{\Tilde{u} \in L^{2}(\nu_i)} \mathbb{E}  \Big| \Tilde{u}(X_{i}^{\pi}) - Y_{i+1}^{\pi} - \Delta t_{i} g(t_{i}, \Tilde{u}(X_{i}^{\pi}), Z_{i}^{\pi}) \Big|^{2}, \\
    v_{i} &\coloneqq \argmin_{\Tilde{v} \in L^{2}(\nu_i)} \mathbb{E}  \Big| \Tilde{v}(X_{i}^{\pi}) - \frac{1}{\Delta t_{i}} Y_{i+1}^{\pi} \Delta B_{i} \Big|^{2}.
\end{flalign*}
In practice, one approximates the solution to such infinite-dimensional optimization problems by solving it over a finite-dimensional subspace of $L^{2}(\nu_i)$. The most classical approach consists in considering the linear span given by a truncated orthonormal basis of $L^{2}(\nu_i)$, see e.g. \cite{BouchardBruno2004DaaM}. More recently, it has been proposed to use neural networks, which are particularly suitable when the dimension of the forward diffusion is high, see e.g. \cite{HureCome2020Dbsf}.

\section{Analysis of the solution operators: a Wiener chaos decomposition approach}\label{section 3}

Through the rest of the paper we fix a generator $g$ satisfying Assumption \ref{assumption 1} (i)-(ii). With the computation of conditional $g$-expectations in mind, we are interested in studying the solution operators
$\mathcal{Y}\colon \xi \mapsto Y^{\xi}$ and $\mathcal{Z}\colon \xi \mapsto Z^{\xi}$. 

The main goal of this section is to show that the solution operators evaluated at any time $t \in [0,T]$, which recall are defined as $\mathcal{Y}_{t}(\xi) \coloneqq Y_{t}^{\xi}$ and $\mathcal{Z}_{t}(\xi) \coloneqq Z_{t}^{\xi}$, can be written as a measurable transformation of $(t, X_{t}, \xi)$, where $X = (X_{t})_{t \in [0,T]}$ satisfies an infinite-dimensional SDE that relies on the Wiener chaos decomposition. In order to do so, we will frame the problem as the solution of an infinite-dimensional Markovian Forward-Backward SDE system with parametric final condition.

\subsection{The Wiener chaos decomposition}

We briefly recall the definition of the Wiener chaos decomposition of $L^{2}(\mathcal{F}_{T})$. We refer to \cite[Section 1.1]{NualartDavidTMCa} or \cite[Chapter 1]{DiNunnoGiulia2009MCfL} for more details on this topic. Let $H_{n}(x)$ denote the $n$th Hermite polynomial, which is defined by 
\begin{gather*}
    H_{n}(x) = \frac{(-1)^{n}}{n!} e^{\frac{x^{2}}{2}} \frac{d^{n}}{dx^{n}}\big( e^{\frac{-x^{2}}{2}} \big), \quad n \geq 1,
\end{gather*}
and $H_{0}(x) = 1$. These polynomials are the coefficients of the decomposition in powers of $t$ of the function $F(x,t) = \exp\big(tx - \frac{t^{2}}{2} \big)$. In fact, we have
\begin{flalign*}
    F(x,t) = \exp \Big( \frac{x^{2}}{2} - \frac{1}{2}(x-t)^{2} \Big)
    &= e^{\frac{x^{2}}{2}} \sum_{n=0}^{\infty} \frac{t^{n}}{n!} \big( \frac{d^{n}}{dt^{n}}e^{-\frac{(x-t)^{2}}{2}} \big)\big\vert_{t=0}
    = \sum_{n=0}^{\infty} t^{n} H_{n}(x).
\end{flalign*}
Furthermore, since $\partial_{x} F(x,t) = t \exp \Big( \frac{x^{2}}{2} - \frac{1}{2}(x-t)^{2} \Big)$, one has that 
\begin{flalign}
    H_{n}'(x) &= H_{n-1}(x), \quad n \geq 0,\label{derivative hermite}
\end{flalign}
with the convention that $H_{-1}(x) \equiv 0$. 

Let $\mathcal{A}$ denote the set of all sequences $a = (a_{1}, a_{2}, \dots)$, $a_{i} \in \mathbb{N} \cup \{0\}$, such that all the terms, except a finite number of them, vanish. For $a \in \mathcal{A}$, we set $a! = \prod_{i \geq 1} a_{i}!$ and $|a| = \sum_{i \geq 1} a_{i}$, and define 
\begin{gather*}
    \Phi_{a} = \sqrt{a!} \prod_{i=1}^{\infty} H_{a_{i}} \Big(\int_{0}^{T} h_{i}(s) \cdot dB_{s} \Big),
\end{gather*}
where $(h_{i})_{i \geq 1}$ is an orthonormal basis of $L^{2}([0,T]; \mathbb{R}^{d})$. Notice that the above product is well defined because $H_{0}(x) = 1$ and $a_{i} \neq 0$ only for a finite number of indices. The set $\{ \Phi_{a}, |a| = k\}$ is called the $k$th level of the Wiener chaos.

We then have the following result, called the Wiener chaos decomposition of $L^{2}(\mathcal{F}_{T})$. See \cite[Proposition 1.1.1]{NualartDavidTMCa} for the proof.
\begin{proposition}
    The collection $\{ \Phi_{a}, a \in \mathcal{A} \}$ is a complete orthonormal system in $L^{2}(\mathcal{F}_{T})$. That is, for any $\xi \in L^{2}(\mathcal{F}_{T})$ we have the decomposition 
    \begin{gather*}
        \xi  = \sum_{k \geq 0} \sum_{|a|=k} d_{a} \prod_{i \geq 1} H_{a_{i}}\Big( \int_{0}^{T} h_{i}(s) \cdot dB_{s} \Big), \quad d_{a} \coloneqq  a! \mathbb{E} \Big[ \xi \times \prod_{i \geq 1} H_{a_{i}} \Big( \int_{0}^{T} h_{i}(s) \cdot dB_{s} \Big) \Big].
    \end{gather*}
\end{proposition}

\subsection{An infinite-dimensional Forward-Backward SDE}

In our approach, we use the chaos decomposition to write the solution of a BSDE with fixed data $(\xi, g)$ as the solution of a FBSDE system with infinite-dimensional forward SDE. 

Consider the process $X = (X_{t})_{t \in [0,T]}$, taking values in $\mathbb{R}^{\mathcal{A}}$, which is defined coordinate-wise by 
\begin{gather}\label{infinite stochastic process hermite}
    X_{t}^{(a_{1}, a_{2}, \dots)} \coloneqq \prod_{i \geq 1} H_{a_{i}}  \Big( \int_{0}^{t} h_{i}(s) \cdot dB_{s} \Big),
\end{gather}
for any $a=(a_{1}, a_{2}, \dots) \in \mathcal{A}$. 
\begin{proposition}\label{proposition 2}
    The process $X = (X_{t})_{t \in [0,T]}$ satisfies the infinite-dimensional linear SDE
 \begin{gather*}
     dX_{t} = b(t,X_{t})dt + \sigma(t, X_{t}) dB_{t},
 \end{gather*}
 where $b\colon[0,T] \times \mathbb{R}^{\mathcal{A}} \to \mathbb{R}^{\mathcal{A}}$ and $\sigma\colon[0,T] \times \mathbb{R}^{\mathcal{A}} \to \mathbb{R}^{\mathcal{A} \times d}$ are given by 
 \begin{flalign*}
     b(t,x)^{(a_{1}, a_{2}, \dots )} &= \frac{1}{2} \sum_{i \geq 1} x^{(a_{1}, \dots, a_{i}-2, \dots)} |h_{i}(t)|_{\mathbb{R}^{d}}^{2}  \\ & \qquad + \sum_{1 \leq j < i} x^{(a_{1}, \dots, a_{j}-1, \dots, a_{i}-1, \dots)} \scalar{h_{j}(t)}{h_{i}(t)}_{\mathbb{R}^{d}} , \\
     \sigma(t,x)^{(a_{1}, a_{2}, \dots )} &= \sum_{i \geq 1} x^{(a_{1}, \dots, a_{i}-1, \dots)} h_{i}(t),
 \end{flalign*}
 where we use the convention $x^{a} \equiv 0$ whenever $a_{i} < 0$ for some $i \in \mathbb{N}$. The initial condition is determined by the fact that $H_{n}(0) = 0$ whenever $n$ is odd and $H_{2k} (0) = \frac{(-1)^{k}}{2^{k} k!}$ for all $k \geq 1$. 
\end{proposition}

\begin{proof}
Using the integration by parts formula we have that 
\begin{flalign*}
    & dX_{t}^{(a_{1}, a_{2}, \dots)} = \sum_{i \geq 1}  \prod_{\substack{j \geq 1 \\ j \neq i}} H_{a_{j}} \Big( \int_{0}^{t} h_{j}(s) \cdot dB_{s} \Big) \times d \Big( H_{a_{i}}  \Big( \int_{0}^{t} h_{i}(s) \cdot dB_{s} \Big) \Big)  \\
    & + \sum_{1 \leq j < i} \prod_{l \neq i,j} H_{a_{l}}\Big( \int_{0}^{t} h_{l}(s) \cdot dB_{s} \Big) \times d \Big\langle H_{a_{j}}  \Big( \int_{0}^{\cdot} h_{j}(s) \cdot dB_{s} \Big), H_{a_{i}}  \Big( \int_{0}^{\cdot} h_{i}(s) \cdot dB_{s} \Big) \Big\rangle_{t} .
\end{flalign*}
Let us now find the Itô representation of $H_{a_{i}}  \Big( \int_{0}^{t} h_{i}(s) \cdot dB_{s}\Big)$. Applying Itô's formula together with property (\ref{derivative hermite}), we arrive at 
\begin{flalign*}
    d \Big( H_{a_{i}}  \Big( \int_{0}^{t} h_{i}(s) \cdot dB_{s}\Big) \Big) = & H_{a_{i}-1} \Big(  \int_{0}^{t} h_{i}(s) \cdot dB_{s}\Big) \times h_{i}(t) \cdot dB_{t} \\ & + \frac{1}{2} H_{a_{i}-2}\Big(  \int_{0}^{t} h_{i}(s) \cdot dB_{s}\Big) \times |h_{i}(t)|_{\mathbb{R}^{d}}^{2} dt,
\end{flalign*}
where $|\cdot|_{\mathbb{R}^{d}}$ denotes the Euclidean norm in $\mathbb{R}^{d}$, and we use the convention $H_{-2} \equiv 0$. This means that the covariation part can be written as 
\begin{flalign*}
    d \Big\langle H_{a_{j}} &  \Big( \int_{0}^{\cdot} h_{j}(s) \cdot dB_{s} \Big), H_{a_{i}}  \Big( \int_{0}^{\cdot} h_{i}(s) \cdot dB_{s} \Big) \Big\rangle_{t} \\ &= H_{a_{j}-1} \Big(  \int_{0}^{t} h_{j}(s) \cdot dB_{s}\Big)  H_{a_{i}-1} \Big(  \int_{0}^{t} h_{i}(s) \cdot dB_{s}\Big) \scalar{h_{j}(t)}{h_{i}(t)}_{\mathbb{R}^{d}} dt.
\end{flalign*}
Putting everything together, we get
\begin{flalign*}
    & dX_{t}^{(a_{1}, a_{2}, \dots)} = \sum_{i \geq 1}  \Bigg\{ \prod_{\substack{j \geq 1 \\ j \neq i}} H_{a_{j}} \Big( \int_{0}^{t} h_{j}(s) \cdot dB_{s} \Big) \bigg[ H_{a_{i}-1} \Big(  \int_{0}^{t} h_{i}(s) \cdot dB_{s}\Big) \times h_{i}(t) \cdot dB_{t} \\ & + \frac{1}{2} H_{a_{i}-2}\Big(  \int_{0}^{t} h_{i}(s) \cdot dB_{s}\Big) \ |h_{i}(t)|_{\mathbb{R}^{d}}^{2} dt \bigg]\Bigg\} + \sum_{1 \leq j < i} \Bigg\{ \prod_{l \neq i,j} H_{a_{l}}\Big( \int_{0}^{t} h_{l}(s) \cdot dB_{s} \Big) \ \\
    & H_{a_{j}-1} \Big(  \int_{0}^{t} h_{j}(s) \cdot dB_{s}\Big) \ H_{a_{i}-1} \Big(  \int_{0}^{t} h_{i}(s) \cdot dB_{s}\Big) \scalar{h_{j}(t)}{h_{i}(t)}_{\mathbb{R}^{d}} dt \Bigg\} \\
    &= \sum_{i \geq 1} \Big( X_{t}^{(a_{1}, \dots, a_{i}-1, \dots)} h_{i}(t) \cdot dB_{t} + \frac{1}{2} X_{t}^{(a_{1}, \dots, a_{i}-2, \dots)} |h_{i}(t)|_{\mathbb{R}^{d}}^{2} dt \Big) \\ & + \sum_{1 \leq j < i} X_{t}^{(a_{1}, \dots, a_{i}-1, \dots, a_{j}-1, \dots)} \scalar{h_{j}(t)}{h_{i}(t)}_{\mathbb{R}^{d}} dt.
\end{flalign*}
\end{proof}

\begin{remark}\label{remark fd sde}
Consider the stochastic process $X^{\leq p}$ obtained by considering only the first $p$ levels of $(X^{a})_{a \in \mathcal{A}}$, that is $X^{\leq p} \coloneqq \{ X^{a}, |a| \leq p\}$. Then it is easy to see that $X^{\leq p}$ satisfies an SDE system with deterministic coefficients. This is because the $k$th level of the vector fields $b$ and $\sigma$ only depends on the lower levels (in particular, it depends on the $(k-1)$th and $(k-2)$th). 

This is similar to the infinite-dimensional SDE that the signature of a process satisfies, given by $dS(X)_{t} = S(X)_{t} \otimes dX_{t}$. See \cite{LyonsTerryJ2007DEDb} for its precise definition.

If we further consider only the dimensions of $X^{\leq p}$ such that $a_{m} = 0$ for all $m > M$, for some $M \in \mathbb{N}$, that is $X^{\leq p, M} \coloneqq \{ X^{a}, |a| \leq p, a_{k} = 0 \ \forall k>M\}$, then the process satisfies a classical finite-dimensional SDE with deterministic coefficients.
\end{remark}

We can then connect the solution of the BSDE  
\begin{gather*}
    Y_{t} = \xi + \int_{t}^{T} g(s, Y_{s}, Z_{s} ) ds - \int_{t}^{T} Z_{s} \cdot dB_{s}
\end{gather*}
to the solution of the FBSDE system
\begin{flalign}
    X_{t} &= x_{0} + \int_{0}^{t} b(s,X_{s}) ds + \int_{0}^{t} \sigma(s, X_{s}) dB_{s} \label{forward sde infinite dimensional}\\
    Y_{t} &= \sum_{k \geq 0} \sum_{|a|=k} d_{a} X_{T}^{a} + \int_{t}^{T} g(s, Y_{s}, Z_{s} ) ds - \int_{t}^{T} Z_{s} \cdot dB_{s} \label{bsde infinite dimensional}
\end{flalign}
via the chaos decomposition $\xi = \sum_{k \geq 0} \sum_{|a|=k} d_{a} X_{T}^{a} \in L^{2}(\mathcal{F}_{T})$ and Proposition \ref{proposition 2}. That is, a FBSDE system where the forward SDE is infinite-dimensional and where the terminal condition of the BSDE only depends on the final value of the forward SDE. 

We conclude this subsection by recalling the most natural way of using the chaos decomposition to project $L^{2}(\mathcal{F}_{T})$ into a finite-dimensional subspace. This method will also be used in the implementation of the numerical method, see Section \ref{section finite-dimensional approximation}.

Let $p, M\in \mathbb{N}$. For $\xi \in L^{2}(\mathcal{F}_{T})$, we define the projection 
\begin{gather}\label{truncation of the chaos decomposition}
    \Pi_{p, M}(\xi) \coloneqq \sum_{k=0}^{p} \sum_{|a|=k} d_{a} X_{T}^{a}, \quad a = (a_{1}, \dots, a_{M}).
\end{gather}
Recall that each $a = (a_{1}, \dots, a_{M})$ with $|a| = k$ means that $\sum_{j=1}^{M} a_{j} = k$. The dimension of the finite-dimensional subspace $\Pi_{p,M}(L^{2}(\mathcal{F}_{T}))$ can be easily obtained by using known combinatorial results, see (\ref{combinatorics cardinality}). Here we have that $p$ denotes the maximum order of the chaos decomposition, and $M \in \mathbb{N}$ the number of elements in a truncated basis of $L^{2}([0,T]; \mathbb{R}^{d})$. 

\begin{remark}
    By Remark \ref{remark fd sde}, when we restrict the terminal condition $\xi$ to belong to $\Pi_{p,M}(L^{2}(\mathcal{F}_{T}))$, the infinite-dimensional FBSDE (\ref{forward sde infinite dimensional})-(\ref{bsde infinite dimensional}) is clearly equivalent to a finite-dimensional one, since the terminal condition only depends on a finite number of dimensions of the forward SDE. 
\end{remark}

\subsection{Markovian property of the FBSDE system}\label{A proof of the Markovianity of the previous FBSDE system}

We equip $\mathbb{R}^{\mathcal{A}}$ with the standard product $\sigma$-algebra $\mathcal{B}(\mathbb{R}^{\mathcal{A}}) \coloneqq \bigotimes_{a \in \mathcal{A}} \mathcal{B}(\mathbb{R})$. For any $p, M \in \mathbb{N}$, we denote by $\pi_{p,M}$ the projections
\begin{gather*}
    \begin{array}{rcccc}
   \pi_{p,M} \colon & \mathbb{R}^{\mathcal{A}} & \to & \mathbb{R}^{\mathcal{A}_{p,M}} & \\
    & x & \mapsto & (x^{a})_{a \in \mathcal{A}_{p,M}},
    \end{array} 
\end{gather*}
where $\mathcal{A}_{p,M} \coloneqq \{ a \in \mathcal{A} \ : \ |a| \leq p, a_{k} = 0 \ \forall k > M \}$. Notice that this is a set of finite cardinality, and therefore $\mathbb{R}^{\mathcal{A}_{p,M}}$ has finite dimension.

For every $s, t \in [0,T]$ with $s \leq t$, let $W_{s,t} = C([s,t]; \mathbb{R}^{d})$ be the space of all continuous functions defined on $[s,t]$ with values in $\mathbb{R}^{d}$. We equip $W_{s,t}$ with the uniform norm over $[s,t]$. We then denote by $\mathcal{B}(W_{s,t})$ the corresponding Borelian $\sigma$-algebra.

The following lemma proves that $X$ has the Markov property. This will be useful in the sequel. 

\begin{lemma}\label{markov property X infinite dimensional}
    For every $s, t \in [0,T]$ with $s \leq t$, there exists a $\mathcal{B}(\mathbb{R}^{\mathcal{A}}) \otimes \mathcal{B}(W_{s,t}) / \mathcal{B}(\mathbb{R}^{\mathcal{A}})$-measurable map $\varphi \colon \mathbb{R}^{\mathcal{A}} \times W_{s,t} \to \mathbb{R}^{\mathcal{A}}$ such that 
    \begin{gather*}
        X_{t} = \varphi( X_{s}, \{B_{r}-B_{s}\}_{r \in [s,t]}) \quad \text{$\mathbb{P}$-a.s.}
    \end{gather*}
    Consequently, the process $X = (X_{t})_{t \in [0,T]}$ has the Markov property.
\end{lemma}

\begin{proof}
As noted in Remark \ref{remark fd sde}, for each $p, M \in \mathbb{N}$, the process $X^{\leq p, M}$ satisfies a finite-dimensional SDE. Therefore there exists a $\mathcal{B}(\mathbb{R}^{\mathcal{A}_{p,M}}) \otimes \mathcal{B}(W_{s,t})/\mathcal{B}(\mathbb{R}^{\mathcal{A}_{p,M}})$-measurable map $\Tilde{\varphi}_{p, M} \colon \mathbb{R}^{\mathcal{A}_{p,M}} \times W_{s,t} \to \mathbb{R}^{\mathcal{A}_{p,M}}$ such that 
\begin{gather*}
    X_{t}^{\leq p, M} = \Tilde{\varphi}_{p, M}( X_{s}^{\leq p, M}, \{B_{r}-B_{s}\}_{r \in [s,t]}) \quad \text{$\mathbb{P}$-a.s.,}
\end{gather*}
see e.g. \cite[Chapter IV]{ikeda2014stochastic}. Since the projection $\pi_{p,M}$ is $\mathcal{B}(\mathbb{R}^{\mathcal{A}})\allowbreak/\mathcal{B}(\mathbb{R}^{\mathcal{A}_{p,M}})$-measurable, we have that the map $\varphi_{p, M} \colon \mathbb{R}^{\mathcal{A}} \times W_{s,t} \to \mathbb{R}^{\mathcal{A}_{p,M}}$ defined by
\begin{gather*}
   \varphi_{p, M}(x, B) \coloneqq \Tilde{\varphi}_{p, M}( \pi_{p,M}(x), B), \quad x \in \mathbb{R}^{\mathcal{A}}, B \in W_{s,t}
\end{gather*}
is $\mathcal{B}(\mathbb{R}^{\mathcal{A}}) \otimes \mathcal{B}(W_{s,t})/\mathcal{B}(\mathbb{R}^{\mathcal{A}_{p,M}})$-measurable. We can then construct $\varphi \colon \mathbb{R}^{\mathcal{A}} \times W_{s,t} \to \mathbb{R}^{\mathcal{A}}$ by
\begin{gather*}
\varphi(\pi_{p,M}(x),B) = \varphi_{p,M}(x, B), \quad \text{for any $p, M \in \mathbb{N}$.}     
\end{gather*}
We remark that the collection of maps $\{ \varphi_{p,M}\}_{p,M \in \mathbb{N}}$ completely determines $\varphi$. Then, since the projections $\varphi_{p,M}$ are measurable, we have that $\varphi$ is measurable, see e.g. \cite[Lemma 1.9]{KallenbergOlav2021FoMP}.

Finally, to prove that the process $X = (X_{t})_{t \in [0,T]}$ has the Markov property, it is enough to show that for any bounded Borel measurable $\phi\colon \mathbb{R}^{\mathcal{A}} \to \mathbb{R}$, there exists another Borel measurable $\psi\colon \mathbb{R}^{\mathcal{A}} \to \mathbb{R}$ such that
\begin{gather*}
    \mathbb{E} [\phi(X_{t}) \vert \mathcal{F}_{s}] = \psi(X_{s}) \quad \text{$\mathbb{P}$-a.s.}
\end{gather*} 
Using the first part of the lemma yields
\begin{flalign*}
   \mathbb{E} [\phi(X_{t}) \vert \mathcal{F}_{s}] &= \mathbb{E} [\phi \circ \varphi( X_{s}, \{B_{r}-B_{s}\}_{r \in [s,t]}) \vert \mathcal{F}_{s}].
\end{flalign*} 
Since $X_{s}$ is $\mathcal{F}_{s}$-measurable and the family of random variables $\{B_{r}-B_{s}\}_{r \in [s,t]}$ is independent of $\mathcal{F}_{s}$, we can use the freezing lemma \cite[Lemma 4.1]{BaldiPaolo2017SCAI} to obtain
\begin{flalign*}
   \mathbb{E} [\phi \circ \varphi( X_{s}, \{B_{r}-B_{s}\}_{r \in [s,t]}) \vert \mathcal{F}_{s}] = \mathbb{E} [\phi \circ \varphi( x, \{B_{r}-B_{s}\}_{r \in [s,t]})] \big{|}_{x=X_{s}} \eqqcolon \psi(X_{s}).
\end{flalign*}  
\end{proof}

Here below we discuss the Markovian property, in the same sense as Theorem \ref{theorem markovianity}, of the infinite-dimensional FBSDE system (\ref{forward sde infinite dimensional})-(\ref{bsde infinite dimensional}).

\begin{theorem}\label{theorem markovian infinite}
    Equip $[0,T] \times \mathbb{R}^{\mathcal{A}}$ with the product $\sigma$-algebra. Then, for any $\xi \in L^{2}(\mathcal{F}_{T})$, there exist two measurable maps $u\colon[0,T] \times \mathbb{R}^{\mathcal{A}} \to \mathbb{R}$ and $v\colon[0,T] \times \mathbb{R}^{\mathcal{A}} \to \mathbb{R}^{d}$ such that $Y_{t} = u(t, X_{t})$ and $Z_{t} = v(t, X_{t})$ $dt \otimes \mathbb{P}$-a.e. Moreover, we have that $u$ and $v$ belong to $L^{2}(\nu)$ for some finite measure $\nu$ on $[0,T] \times \mathbb{R}^{\mathcal{A}}$.
\end{theorem}

\begin{proof}[Proof of Theorem \ref{theorem markovian infinite}]
    We proceed to prove the result for the $Y$ component, with a similar approach applicable to the $Z$ component.

    For any $p, M \in \mathbb{N}$ consider the projection $\xi^{\leq p,M} \coloneqq \Pi_{p,M}(\xi)$. Let $Y^{\leq p,M}$ denote the $Y$ component of the solution to the BSDE with terminal condition $\xi^{\leq p,M}$.
    
    By Theorem \ref{theorem markovianity}, we know that there exist a measurable map $\Tilde{u}_{p,M}\colon[0,T] \times \mathbb{R}^{\mathcal{A}_{p,M}} \to \mathbb{R}$ such that $Y_{t}^{\leq p,M} = \Tilde{u}_{p,M}(t, X_{t}^{\leq p,M})$ $dt \otimes \mathbb{P}$-a.e.

    Since the projections $\pi_{p,M}$ are $\mathcal{B}(\mathbb{R}^{\mathcal{A}})/\mathcal{B}(\mathbb{R}^{\mathcal{A}_{p,M}})$-measurable, we have that the map $u_{p,M}$ defined by 
    \begin{gather*}
    \begin{array}{rcccc}
   u_{p,M} \colon & [0,T] \times \mathbb{R}^{\mathcal{A}} & \to & \mathbb{R} & \\
    & (t, x) & \mapsto & \Tilde{u}_{p,M}(t, \Pi_{p,M}(x))
    \end{array}  
    \end{gather*}
    is measurable. Let us consider the measure space $([0,T] \times \Omega, \mathcal{B}([0,T]) \otimes \mathcal{F}_{T}, dt \otimes \mathbb{P})$, and let $\nu$ be the pushforward (finite) measure on $([0,T] \times \mathbb{R}^{\mathcal{A}}, \mathcal{B}([0,T]) \otimes \mathcal{B}( \mathbb{R}^{\mathcal{A}}))$ given by the transformation
\begin{gather*}
    \begin{array}{rcccc}
   \widehat{X} \colon & [0,T] \times \Omega & \to & [0,T] \times \mathbb{R}^{\mathcal{A}} & \\
    & (t, \omega) & \mapsto & (t, X_{t}(\omega)).
\end{array} 
\end{gather*}
We now argue that the limit of $u_{p,M}$ exists in $L^{2}(\nu)$. 

By the a priori bounds given in Corollary \ref{corollary 1 a priori bounds}, there exists a constant $C>0$, depending only on $T$, $[g]_{L}$ and $d$, such that for any $p_{i}, M_{i} \in \mathbb{N}$, $i=1,2$ one has that
\begin{gather*}
         \int_{[0,T] \times \Omega} \vert u_{p_{1}, M_{1}}(t, X_{t}(\omega)) - u_{p_{2}, M_{2}}(t, X_{t}(\omega)) \vert^{2} dt \otimes \mathbb{P}(d\omega) \leq C \mathbb{E} \vert \xi^{\leq p_{1}, M_{1}} - \xi^{\leq p_{2}, M_{2}} \vert^{2} .
\end{gather*}
By the image measure theorem, one gets 
\begin{gather*}
         \int_{[0,T] \times \mathbb{R}^{\mathcal{A}}} \vert u_{p_{1}, M_{1}}(t, x) - u_{p_{2}, M_{2}}(t, x) \vert^{2} \nu(dt, dx) \leq C \mathbb{E} \vert \xi^{\leq p_{1}, M_{1}} - \xi^{\leq p_{2}, M_{2}}\vert^{2} .
\end{gather*}
From this it is easy to construct a Cauchy sequence $\{u_{N}\}_{N \geq 1}$ in $L^{2}(\nu)$. Indeed, set $p$ and $M$ to be equal to $N$. Since $L^{2}(\mathcal{F}_{T})$ is a Hilbert space, we have that for any $\epsilon > 0$, we can find $N_{\epsilon} \in \mathbb{N}$ such that $\mathbb{E} \vert \xi^{\leq N_{1}, N_{1}} - \xi^{\leq N_{2}, N_{2}} \vert^{2} \leq \epsilon$ for all $N_{1}, N_{2} \geq N_{\epsilon}$, which means that $\{u_{N}\}_{N \geq 1}$ is Cauchy. Hence the limit exists in $L^{2}(\nu)$ and $\nu$-a.s. for a subsequence. We then define the (measurable) limit map of this subsequence everywhere by setting it to $0$ on the $\nu$ null set where it is not defined. We denote this map by $u$. 

In order to show that $Y_{t} = u(t, X_{t})$ $dt \otimes \mathbb{P}$-a.e., we can use again Corollary \ref{corollary 1 a priori bounds} to arrive at 
\begin{gather*}
    \mathbb{E} \Big[ \int_{0}^{T} \vert Y_{t} - u_{p, M}(t, X_{t}) \vert^{2} dt \Big]  \leq C \mathbb{E} \vert \xi - \xi^{\leq p, M}\vert^{2},
\end{gather*}
from which we easily deduce the result.

\end{proof}

\subsection{Incorporating the terminal condition as an additional variable in the Markovian maps}

    With the results of the previous section, we can then see that the solution operators given in Definition \ref{solution operators} can be equivalently written as 
    \begin{gather*}
    \begin{array}{rcccc}
    \mathcal{Y} \colon & L^{2}(\mathcal{F}_{T}) & \to & L^{2}([0,T] \times \mathbb{R}^{\mathcal{A}}, \nu)  \\
    & \xi & \mapsto & u^{\xi}
\end{array} , \quad \begin{array}{rcccc}
    \mathcal{Z} \colon & L^{2}(\mathcal{F}_{T}) & \to &L^{2}([0,T] \times \mathbb{R}^{\mathcal{A}}, \nu) \\
    & \xi & \mapsto & v^{\xi}
\end{array} ,
\end{gather*}
where $u^{\xi}$ and $v^{\xi}$ are the maps in Theorem \ref{theorem markovian infinite} for the terminal condition $\xi$. 
Alternatively, we can regard these operators as functionals by evaluating $u,v$ at a particular point,
\begin{gather*}
    \begin{array}{rcccc}
    \mathscr{Y} \colon & [0,T] \times \mathbb{R}^{\mathcal{A}} \times L^{2}(\mathcal{F}_{T})  & \to & \mathbb{R}  \\
    & (t,x,\xi) & \mapsto & \mathcal{Y}(\xi)(t,x)
\end{array} , \quad \begin{array}{rcccc}
    \mathscr{Z} \colon & [0,T] \times \mathbb{R}^{\mathcal{A}} \times L^{2}(\mathcal{F}_{T})  & \to & \mathbb{R}^{d}  \\
    & (t,x,\xi) & \mapsto & \mathcal{Z}(\xi)(t,x)
\end{array}.
\end{gather*}

For each fixed $\xi$, both $\mathscr{Y}(\cdot, \cdot, \xi)$ and $\mathscr{Z}(\cdot, \cdot, \xi)$ belong to $L^{2}(\nu)$. For a fixed $(t,x) \in [0,T] \times \mathbb{R}^{\mathcal{A}}$ we have the following continuity result for $\mathscr{Y}(t, x, \cdot)$ and $\mathscr{Z}(t, x, \cdot)$.

\begin{proposition}\label{continuity of xi}
    $\mathscr{Y}(t, x, \cdot)\colon L^{2}(\mathcal{F}_{T}) \to \mathbb{R}$ and $\mathscr{Z}(t, x, \cdot)\colon L^{2}(\mathcal{F}_{T}) \to \mathbb{R}^{d}$ are uniformly continuous for any $(t,x) \in [0,T] \times \mathbb{R}^{\mathcal{A}}$ $\nu$-a.e.
\end{proposition}

\begin{proof}
We prove it for $\mathscr{Y}$. The proof for $\mathscr{Z}$ is similar. We will show that, for $\nu$-almost all $(t,x)$, the following holds: for all $k \in \mathbb{N}$, there exists $\delta_{t, x} > 0$ such that $|\mathscr{Y}(t,x,\xi_{1}) - \mathscr{Y}(t,x,\xi_{2})| < 2^{-k}$ for all $\xi_{1}, \xi_{2} \in L^{2}(\mathcal{F}_{T})$ with $\norm{\xi_{1} - \xi_{2}}_{L^{2}(\mathcal{F}_{T})} < \delta_{t, x}$.

Clearly, $\mathcal{Y}$ is globally Lipschitz, and with the definition of $\mathscr{Y}$, it follows that
\begin{gather}\label{inequality}
    \int_{[0,T] \times \mathbb{R}^{\mathcal{A}}} \vert \mathscr{Y}(t, x, \xi_{1}) - \mathscr{Y}(t, x, \xi_{2}) \vert^{2} \nu(dt, dx) \leq C \norm{\xi_{1} - \xi_{2}}_{L^{2}(\mathcal{F}_{T})}^{2}.
\end{gather}
Combining (\ref{inequality}) and Chebyshev's inequality we are able to show that for any $\delta >0$, $\xi_{1}, \xi_{2} \in L^{2}(\mathcal{F}_{T})$ with $\norm{\xi_{1} - \xi_{2}}_{L^{2}(\mathcal{F}_{T})} \leq \delta$, we have that 
\begin{flalign*}
        \nu(\{ (s,y) \in [0,T] & \times \mathbb{R}^{\mathcal{A}} \colon | \mathscr{Y}(s,y,\xi_{1}) - \mathscr{Y}(s,y,\xi_{2})| \geq 2^{-k} \}) \\ &\leq \frac{1}{2^{-2k}} \int_{[0,T] \times \mathbb{R}^{\mathcal{A}}} |\mathscr{Y}(s,y,\xi_{1}) - \mathscr{Y}(s,y, \xi_{2})|^{2} \nu(ds, dy) \leq C \frac{\delta^{2}}{2^{-2k}}
\end{flalign*}
for any $k \in \mathbb{N}$. Set $\delta \coloneqq C^{-1} 2^{-3k}$, such that 
\begin{flalign*}
        \nu(\{ (s,y) \in [0,T] & \times \mathbb{R}^{\mathcal{A}} \colon | \mathscr{Y}(s,y,\xi_{1}) - \mathscr{Y}(s,y,\xi_{2})| \geq 2^{-k} \}) \leq 2^{-k}.
\end{flalign*}
The subsets 
\begin{flalign*}
        A_{k} \coloneqq \{ (s,y) \in [0,T] & \times \mathbb{R}^{\mathcal{A}} \colon | \mathscr{Y}(s,y,\xi_{1}) - \mathscr{Y}(s,y,\xi_{2})| \geq 2^{-k} \}
\end{flalign*}
verify that $\sum_{k=1}^{\infty} \nu(A_{k}) \leq \sum_{k=1}^{\infty} 2^{-k} = 1 < \infty$. By the Borel-Cantelli lemma, 
\begin{gather*}
    \nu( \{ \text{$(s,y) \in A_{k}$ for infinitely many $k$} \}) = 0, 
\end{gather*}
which means that for $\nu$-almost all $(t,x)$, there exists $L_{t,x} \in \mathbb{N}$ such that for all $k \geq L_{t,x}$
\begin{gather*}
    | \mathscr{Y}(t,x,\xi_{1}) - \mathscr{Y}(t,x,\xi_{2})| < 2^{-k}, \quad \text{ whenever $\norm{\xi_{1} - \xi_{2}}_{L^{2}(\mathcal{F}_{T})} \leq C^{-1} 2^{-3k}$.}
\end{gather*}

\end{proof}

The previous continuity result allows us to prove that the functionals $\mathscr{Y}$ and $\mathscr{Z}$ are jointly measurable. The proof relies on the following lemma, which can be found in \cite[Lemma 4.51]{AliprantisCharalambosD2006IDAA}.

\begin{lemma}\label{lemma cara}
    Let $(S, \Sigma, \mu)$ be a measure space, $X$ a separable metric space, and $Y$ a metric space. A function $f \colon S \times X \to Y$ is a Carathéodory function if
     \begin{enumerate}[label=(\roman*)]
        \item for all $x \in X$, the function $f^{x} \coloneqq f(\cdot, x) \colon S \to Y$ is $(\Sigma, \mathcal{B}(Y))$-measurable;
        \item $s \in S$ $\mu$-a.e., the function $f_{s} \coloneqq f(s, \cdot)\colon X \to Y$ is continuous. 
    \end{enumerate}
    Then every Carathéodory function $f\colon S \times X \to Y$ is jointly measurable. That is, it is measurable with respect to the product $\sigma$-algebra $\Sigma \otimes \mathcal{B}(X)$. 
\end{lemma}

\begin{theorem}
    $\mathscr{Y}, \mathscr{Z} \colon [0,T] \times \mathbb{R}^{\mathcal{A}} \times L^{2}(\mathcal{F}_{T}) \to \mathbb{R} \times \mathbb{R}^{d}$ are jointly measurable with respect to the product $\sigma$-algebra.
\end{theorem}
\begin{proof}
    By Proposition \ref{continuity of xi}, we have that for $\nu$-almost all $(t,x) \in [0,T] \times \mathbb{R}^{\mathcal{A}}$, the maps $\mathscr{Y}(t,x, \cdot)$ and $\mathscr{Z}(t,x, \cdot)$ are continuous. Since we also clearly have that $\mathscr{Y}(\cdot,\cdot, \xi)$ and $\mathscr{Z}(\cdot,\cdot, \xi)$ are measurable for all $\xi \in L^{2}(\mathcal{F}_{T})$, we conclude that they are Carathéodory functions, and by Lemma \ref{lemma cara} they are jointly measurable. 
\end{proof}

\section{Operator Euler scheme for BSDEs}\label{section 4}

In this section we will present a scheme to numerically approximate the solution operators of a BSDE. For this, we combine the backward Euler scheme for BSDEs that we reviewed in Section \ref{subsection b euler} with the infinite-dimensional maps $\mathscr{Y}, \mathscr{Z}$ introduced in the previous section. Recall that our goal is to numerically approximate the solution operators given in Definition \ref{solution operators}, which can be expressed as the solution to the Operator BSDE 
\begin{gather*}
    \mathcal{Y}_{t}(\xi) = \xi + \int_{t}^{T} g(s, \mathcal{Y}_{s}(\xi), \mathcal{Z}_{s}(\xi)) ds - \int_{t}^{T} \mathcal{Z}_{s}(\xi) \cdot dB_{s}.
\end{gather*}
Let $\pi = \{t_{0} = 0 < t_{1} < \dots < t_{n} = T\}$ be a partition of the interval $[0,T]$. We write $\Delta t_{i} = t_{i+1}-t_{i}$ and $\Delta B_{i} \coloneqq B_{t_{i+1}} - B_{t_{i}}$ for $i=0, \dots, n-1$. We also set $| \pi | \coloneqq \max_{i} \Delta t_{i}$.

We then define the operators $\mathcal{Y}_{i}^{\pi} \colon L^{2}(\mathcal{F}_{T}) \to L^{2}(\mathcal{F}_{t_{i}})$ and $\mathcal{Z}_{i}^{\pi} \colon L^{2}(\mathcal{F}_{T}) \to L^{2}(\mathcal{F}_{t_{i}})$ recursively by $\mathcal{Y}_{n}^{\pi}(\xi) \coloneqq \xi$, $\mathcal{Z}_{n}^{\pi}(\xi) \coloneqq 0$, and for $i=n-1, \dots, 0$,
\begin{gather*}
    \mathcal{Z}_{i}^{\pi}(\xi) \coloneqq \frac{1}{\Delta t_{i}} \mathbb{E}_{t_{i}} \Big[ \mathcal{Y}_{i+1}^{\pi}(\xi)\Delta B_{i} \Big], \quad \mathcal{Y}_{i}^{\pi}(\xi) \coloneqq \mathbb{E}_{t_{i}} \big[ \mathcal{Y}_{i+1}^{\pi}(\xi) \big] + \Delta t_{i} g(t_{i}, \mathcal{Y}_{i}^{\pi}(\xi), \mathcal{Z}_{i}^{\pi}(\xi)),
\end{gather*}
which is well-defined as long as $|\pi| < [g]_{L}^{-1}$. This scheme is clearly inspired by the backward Euler scheme for BSDEs, as recalled in Section \ref{subsection b euler}, but with a key difference: the terminal condition is no longer fixed. Instead, at each step, we define operators rather than random variables. We call such scheme the Operator Euler scheme for BSDEs.

\subsection{Convergence of the Operator Euler scheme for BSDEs}

Let us set some notation for the regularity of solutions. For $\pi = \{t_{0} = 0 < t_{1} < \dots < t_{n} = T\}$ a partition of the interval $[0,T]$, we define the functionals $(\mathcal{R}^{Y}, \mathcal{R}^{Z})(\cdot, \pi)  \colon L^{2}(\mathcal{F}_{T}) \to [0, \infty)$ by
\begin{gather*}
    \mathcal{R}^{Y}(\xi, \pi) \coloneqq \max_{i=0, \dots, n-1} \mathbb{E} \Big[ \sup_{t \in [t_{i}, t_{i+1}]} |\mathcal{Y}_{t}(\xi) - \mathcal{Y}_{t_{i}}(\xi)|^{2} \Big]
\end{gather*}
and 
\begin{gather*}
    \mathcal{R}^{Z}(\xi, \pi) \coloneqq \mathbb{E} \Big[ \sum_{i=0}^{n-1} \int_{t_{i}}^{t_{i+1}} |\mathcal{Z}_{t}(\xi) - \overline{\mathcal{Z}}_{i}^{\pi}(\xi)|^{2} dt \Big], \quad \overline{\mathcal{Z}}_{i}^{\pi}(\xi) \coloneqq \frac{1}{\Delta t_{i}} \mathbb{E}_{t_{i}} \Big[ \int_{t_{i}}^{t_{i+1}} \mathcal{Z}_{t}(\xi) dt \Big].
\end{gather*}
Notice that $\overline{\mathcal{Z}}_{i}^{\pi}(\xi)$ is defined through the actual solution of the BSDE, and not through some approximation. Therefore, $\mathcal{R}^{Z}(\xi, \pi)$ indicates a type of regularity of the $Z$ component.

The following result states that $\mathcal{R}^{Y}(\xi, \pi)$ and $\mathcal{R}^{Z}(\xi, \pi)$ converge to zero pointwise in $\xi$ when $|\pi| \to 0$. Moreover, the convergence is uniform on compact subsets whenever we choose a sequence of nested partitions.  

\begin{lemma}
Let $(\pi_{m})_{m \geq 1}$ be a sequence of partitions of $[0,T]$ such that $|\pi_{m}| \to 0$ when $m \to \infty$. Then, for each fixed $\xi \in L^{2}(\mathcal{F}_{T})$, we have that 
    \begin{gather*}
        \mathcal{R}^{Y}(\xi, \pi_{m}), \mathcal{R}^{Z}(\xi, \pi_{m}) \to 0 \quad \text{when $m \to \infty$.}
    \end{gather*}
    If we furthermore have that $\pi_{m+1} \subset \pi_{m}$, then the convergence is uniform on compact subsets of $L^{2}(\mathcal{F}_{T})$. That is, for any compact $K \subset L^{2}(\mathcal{F}_{T})$, we have 
    \begin{gather*}
        \sup_{ \xi \in K} \abs{\mathcal{R}^{Y}(\xi,  \pi_{m})},  \sup_{ \xi \in K} \abs{\mathcal{R}^{Z}(\xi, \pi_{m})} \to 0 \quad \text{when $m \to \infty$.}
    \end{gather*}
\end{lemma}
\begin{proof}
Let us write $\pi \coloneqq \{t_{0} = 0 < t_{1} < \dots < t_{n} = T\}$ for a generic partition of $[0,T]$. We start by proving the pointwise convergence of $\mathcal{R}^{Y}$. First of all, we have that 
    \begin{gather*}
        \mathcal{Y}_{t}(\xi) - \mathcal{Y}_{t_{i}}(\xi) = \int_{t_{i}}^{t} g(s, \mathcal{Y}_{s}(\xi), \mathcal{Z}_{s}(\xi)) ds - \int_{t_{i}}^{t} \mathcal{Z}_{s}(\xi) \cdot dB_{s}. 
    \end{gather*}
    Using Hölder's and Burkholder-Davis-Gundy inequalities, together with the linear growth assumption on $g$, we easily obtain
    \begin{flalign*}
        \mathbb{E}\Big[ \sup_{t \in [t_{i}, t_{i+1}]} |\mathcal{Y}_{t}(\xi) - \mathcal{Y}_{t_{i}}(\xi) |^{2} \Big] \leq & C \Big\{ |\Delta t_{i}|^{2} \Big(1  +  \mathbb{E} \Big[ \sup_{t \in [t_{i}, t_{i+1}]} |\mathcal{Y}_{t}(\xi)|^{2} \Big] \Big) \\ &+ \mathbb{E} \Big[ \int_{t_{i}}^{t_{i+1}} |\mathcal{Z}_{s}(\xi)|^{2} ds \Big] \Big\},
    \end{flalign*}
    where $C$ is a constant independent of $\xi$ and $\pi$. The a priori estimates lead to 
    \begin{gather*}
        \mathcal{R}^{Y}(\xi, \pi) \leq C \Big\{  |\pi|^{2} \big( 1 + \mathbb{E} | \xi|^{2} \big) + \max_{i=0, \dots, n-1} \mathbb{E} \Big[ \int_{t_{i}}^{t_{i+1}} |\mathcal{Z}_{s}(\xi)|^{2} ds \Big] \Big\}.
    \end{gather*}
    Let us write now $\pi_{m} \coloneqq \{t_{0}^{m} = 0 < t_{1}^{m} < \dots < t_{n(m)}^{m} = T\}$ for a sequence of partitions such that $|\pi_{m}| \to 0$. One can then prove that the last term converges to zero by setting $\{k_{m}\}_{m \geq 1}$ to be the sequence of integers such that 
    \begin{gather*}
       \mathbb{E} \Big[ \int_{t_{k_{m}}^{m}}^{t_{k_{m}+1}^{m}} |\mathcal{Z}_{s}(\xi)|^{2} ds \Big] =  \max_{i=0, \dots, n(m)-1} \mathbb{E} \Big[ \int_{t_{i}^{m}}^{t_{i+1}^{m}} |\mathcal{Z}_{s}(\xi)|^{2} ds \Big] ,
    \end{gather*}
    and then apply the dominated convergence theorem to interchange limit and expectation. 
    
    In order to prove that $\mathcal{R}^{Z}$ converges pointwise to zero, one can use the fact that $\overline{\mathcal{Z}}_{i}^{\pi}(\xi)$ is the best approximation for $\mathcal{Z}(\xi)$ on $[t_{i}, t_{i+1}]$ in the following sense:
    \begin{gather*}
        \mathbb{E} \Big[ \int_{t_{i}}^{t_{i+1}} | \mathcal{Z}_{t}(\xi) - \overline{\mathcal{Z}}_{t_{i}}^{\pi}(\xi) |^{2} dt \Big] \leq \mathbb{E} \Big[ \int_{t_{i}}^{t_{i+1}} | \mathcal{Z}_{t}(\xi)- \eta |^{2} dt \Big], \quad \text{for any $\eta \in L^{2}(\mathcal{F}_{t_{i}})$.}
    \end{gather*}
    By setting $\eta = \mathcal{Z}_{t_{i}}(\xi)$ and the well-known fact that such piecewise constant process converges to $\mathcal{Z}(\xi)$ in the $\mathbb{H}_{T}^{2}$ norm, we can conclude.   
    
    For the uniform convergence on compacts we need the hypothesis of $\pi_{m} \subset \pi_{m+1}$, which implies that the maps $\mathcal{R}^{Y}(\cdot, \pi_{m}) \colon K \subset L^{2}(\mathcal{F}_{T}) \to \mathbb{R}$ are monotonically decreasing, meaning that $\mathcal{R}^{Y}(\cdot, \pi_{m+1}) \leq \mathcal{R}^{Y}(\cdot, \pi_{m})$ pointwise. Since these are clearly continuous mappings in $\xi$, we are then able to apply Dini's Theorem, concluding the proof.
\end{proof}

We are now ready to present the main result of this section, which provides an estimate for the approximation error of the Operator Euler scheme. This estimate is expressed in terms of $|\pi|$, $\mathcal{R}^{Y}$ and $\mathcal{R}^{Z}$, and the error is defined (pointwise in $\xi$) as
  \begin{flalign*}
    \varepsilon(\xi, \pi) \coloneqq \max_{0 \leq i \leq n-1} \mathbb{E} \Big[ \sup_{t \in [t_{i}, t_{i+1}]} |\mathcal{Y}_{t}(\xi) - \mathcal{Y}_{i}^{\pi}(\xi)|^{2} \Big] + \mathbb{E} \Big[ \sum_{i=0}^{n-1} \int_{t_{i}}^{t_{i+1}} \abs{\mathcal{Z}_{t}(\xi) - \mathcal{Z}_{i}^{\pi}(\xi) }^{2} dt  \Big] .
\end{flalign*}

\begin{theorem}\label{convergence Euler}
  Assume that $g$ is uniformly $\frac{1}{2}$-Hölder continuous in $t$ with Hölder constant $[g]_{L}$. Then there exists a constant $C > 0$, depending only on $T, [g]_{L}$ and $d$, such that
  \begin{flalign}\label{error estimates classical}
    \varepsilon(\xi, \pi) \leq C \Big( |\pi| +\mathcal{R}^{Y}(\xi, \pi) + \mathcal{R}^{Z}(\xi, \pi) \Big).
\end{flalign}
\end{theorem}

\begin{remark}
    The estimate given in (\ref{error estimates classical}) for each fixed $\xi$ follows using the same arguments as the ones found in \cite[Theorem 3.1]{zhang_bsde_method} or \cite[Theorem 5.3.3]{zhang_bsde_method}, where the terminal condition is taken to be a Lipschitz function (or functional) of a forward diffusion in order to control $\mathcal{R}^{Y}(\xi, \pi)$ and $\mathcal{R}^{Z}(\xi, \pi)$. 

Although this inequality is used without proof in \cite{HuYaozhong2011MCFB}, we have found no reference that either states or proves this result in full generality. For completeness, we therefore include a self‐contained proof below.
\end{remark}

The proof relies on the following inequality.

\begin{lemma}\label{technical lemma}
For $i=0, \dots, n-1$, we have that there exists a constant $C>0$, depending only on $T, [g]_{L}$ and $d$, such that
    \begin{flalign}
     \mathbb{E} \Big[ | \Delta \mathcal{Y}_{i}^{\pi}(\xi) |^{2} + & \frac{\Delta t_{i}}{2} |\Delta \mathcal{Z}_{i}^{\pi}(\xi) |^{2} \Big] \leq (1+C \Delta t_{i}) \mathbb{E} | \Delta \mathcal{Y}_{i+1}^{\pi}(\xi) |^{2} \notag \\ & + C \Big\{ \Delta t_{i} \mathcal{R}^{Y}(\xi, \pi) + |\Delta t_{i}|^{2} +  \mathbb{E} \Big[ \int_{t_{i}}^{t_{i+1}} |\mathcal{Z}_{t}(\xi) - \overline{\mathcal{Z}}_{i}^{\pi}(\xi)|^{2} dt \Big] \Big\}, \label{useful inequality}
\end{flalign} 
where
\begin{gather*}
    \Delta \mathcal{Y}_{i}^{\pi}(\xi) \coloneqq (\mathcal{Y}_{i}^{\pi} - \mathcal{Y}_{t_{i}})(\xi), \quad \Delta \mathcal{Z}_{i}^{\pi}(\xi) \coloneqq (\mathcal{Z}_{i}^{\pi} - \overline{\mathcal{Z}}_{i}^{\pi})(\xi).
\end{gather*}
\end{lemma}

\begin{proof}
    We assume for notation simplicity that $d=1$. By the Martingale Representation Theorem, we have that for each $\xi \in L^{2}(\mathcal{F}_{T})$, there exists a process $(\widehat{\mathcal{Z}}_{t}^{\pi}(\xi))_{t \in [t_{i}, t_{i+1}]}$ such that
\begin{gather}\label{mrt}
    \mathcal{Y}_{i+1}^{\pi}(\xi) = \mathbb{E}_{t_{i}} \big[ \mathcal{Y}_{i+1}^{\pi}(\xi) \big] + \int_{t_{i}}^{t_{i+1}} \widehat{\mathcal{Z}}_{t}^{\pi}(\xi) dB_{t} . 
\end{gather}
Isolating $\mathbb{E}_{t_{i}} [ \mathcal{Y}_{i+1}^{\pi} (\xi) ] $ and plugging it into the recursive formula for $\mathcal{Y}_{i}^{\pi}(\xi)$ yields 
\begin{flalign}
    \mathcal{Y}_{i}^{\pi}(\xi) &= \mathbb{E}_{t_{i}} [ \mathcal{Y}_{i+1}^{\pi}(\xi) ] + \Delta t_{i} g(t_{i}, \mathcal{Y}_{i}^{\pi}(\xi), \mathcal{Z}_{i}^{\pi}(\xi)) \notag \\ & = \mathcal{Y}_{i+1}^{\pi}(\xi) + \Delta t_{i} g(t_{i}, \mathcal{Y}_{i}^{\pi}(\xi), \mathcal{Z}_{i}^{\pi}(\xi)) - \int_{t_{i}}^{t_{i+1}} \widehat{\mathcal{Z}}_{t}^{\pi}(\xi) dB_{t}. \label{formula for Y}
\end{flalign}
Then, using the definition of $\mathcal{Z}_{i}^{\pi}(\xi)$, (\ref{mrt}) and Itô's isometry we arrive at
\begin{flalign}
    \mathcal{Z}_{i}^{\pi}(\xi) &= \frac{1}{\Delta t_{i}} \mathbb{E}_{t_{i}} \big[ \mathcal{Y}_{i+1}^{\pi}(\xi) \Delta B_{i} \big] \notag = \frac{1}{\Delta t_{i}} \mathbb{E}_{t_{i}} \Big[\int_{t_{i}}^{t_{i+1}} \widehat{\mathcal{Z}}_{t}^{\pi}(\xi) dB_{t} \Delta B_{i} \Big] \\ &= \frac{1}{\Delta t_{i} } \mathbb{E}_{t_{i}} \Big[\int_{t_{i}}^{t_{i+1}} \widehat{\mathcal{Z}}_{t}^{\pi}(\xi) dt \Big]. \label{Z_t}
\end{flalign}
Next, using the definition of $\overline{\mathcal{Z}}_{i}^{\pi}(\xi)$ and (\ref{Z_t}), we have 
\begin{flalign}
    \mathbb{E} | \Delta \mathcal{Z}_{i}^{\pi}(\xi) |^{2} &= \mathbb{E} | \mathcal{Z}_{i}^{\pi} (\xi) - \overline{\mathcal{Z}}_{i}^{\pi} (\xi) |^{2} = \mathbb{E} \abs{\mathcal{Z}_{i}^{\pi}(\xi) - \frac{1}{\Delta t_{i}} \mathbb{E}_{t_{i}} \Big[ \int_{t_{i}}^{t_{i+1}} \mathcal{Z}_{t}(\xi) dt \Big] }^{2} \notag \\ &= \notag \frac{1}{|\Delta t_{i}|^{2}} \mathbb{E}  \abs{ \mathbb{E}_{t_{i}} \Big[ \int_{t_{i}}^{t_{i+1}} (\widehat{\mathcal{Z}}_{t}^{\pi}(\xi) - \mathcal{Z}_{t}(\xi)) dt \Big] }^{2} \\ &\leq \frac{1}{\Delta t_{i}} \mathbb{E}  \Big[  \int_{t_{i}}^{t_{i+1}} |\widehat{\mathcal{Z}}_{t}^{\pi}(\xi) - \mathcal{Z}_{t}(\xi)|^{2} dt  \Big] \eqqcolon \frac{1}{\Delta t_{i}} \mathbb{E}  \Big[  \int_{t_{i}}^{t_{i+1}} |\Delta \overline{\mathcal{Z}}_{t}^{\pi}(\xi)|^{2} dt  \Big] ,\label{inequality Z}
\end{flalign}
where we used the conditional Cauchy-Schwarz inequality and the law of iterated conditional expectations. We also set $\Delta \overline{\mathcal{Z}}_{t}^{\pi}(\xi) \coloneqq \widehat{\mathcal{Z}}_{t}^{\pi}(\xi) - \mathcal{Z}_{t}(\xi)$.

Using (\ref{formula for Y}) we have 
\begin{gather}\label{isolation}
    \Delta \mathcal{Y}_{i}^{\pi}(\xi) = \Delta \mathcal{Y}_{i+1}^{\pi}(\xi) + \int_{t_{i}}^{t_{i+1}} I_{t}(\xi) dt - \int_{t_{i}}^{t_{i+1}} \Delta \overline{\mathcal{Z}}_{t}^{\pi}(\xi) dB_{t}, 
\end{gather}
where $I_{t}(\xi) \coloneqq g(t_{i}, \mathcal{Y}_{i}^{\pi}(\xi), \mathcal{Z}_{i}^{\pi}(\xi)) - g(t, \mathcal{Y}_{t}(\xi), \mathcal{Z}_{t}(\xi))$. Squaring and taking the expectation to the previous expression, together with the equality we obtain isolating $\Delta \mathcal{Y}_{i+1}^{\pi}(\xi) + \int_{t_{i}}^{t_{i+1}} I_{t}(\xi) dt$ in (\ref{isolation}) and the law of iterated conditional expectations, one can see that 
\begin{gather}\label{useful equality Y}
    \mathbb{E} | \Delta \mathcal{Y}_{i}^{\pi}(\xi) |^{2} = \mathbb{E} \Big[ \Big( \Delta \mathcal{Y}_{i+1}^{\pi}(\xi)  + \int_{t_{i}}^{t_{i+1}} I_{t}(\xi) dt \Big)^{2} - \Big( \int_{t_{i}}^{t_{i+1}} \Delta \overline{\mathcal{Z}}_{t}^{\pi}(\xi) dB_{t} \Big)^{2} \Big].
\end{gather}
Using (\ref{inequality Z}), (\ref{useful equality Y}) and Young's inequality $(a+b)^{2} \leq (1+ \frac{\Delta t_{i}}{\epsilon}) a^{2} + (1+\frac{\epsilon}{\Delta t_{i}}) b^{2}$, we get
\begin{flalign}
    \mathbb{E} \Big[ | \Delta \mathcal{Y}_{i}^{\pi}(\xi) |^{2} + \Delta t_{i} |\Delta \mathcal{Z}_{i}^{\pi}(\xi) & |^{2} \Big] \leq  \mathbb{E} \Big[ | \Delta \mathcal{Y}_{i}^{\pi}(\xi) |^{2} + \int_{t_{i}}^{t_{i+1}} |\Delta \overline{\mathcal{Z}}_{t}^{\pi}(\xi) |^{2} dt \Big] \notag \\
    & = \mathbb{E} \Big[ \Big( \Delta \mathcal{Y}_{i+1}^{\pi}(\xi)  + \int_{t_{i}}^{t_{i+1}} I_{t}(\xi) dt \Big)^{2} \Big] \notag  \\
    & \leq \big(1+\frac{\Delta t_{i}}{\epsilon} \big) \mathbb{E} | \Delta \mathcal{Y}_{i+1}^{\pi}(\xi) |^{2} + \big(\Delta t_{i}+ \epsilon \big) \mathbb{E} \Big[ \int_{t_{i}}^{t_{i+1}} |I_{t}(\xi)|^{2} dt \Big] .\label{inequality final}
\end{flalign}
By the globally Lipschitz property of $g$ with respect to the spatial variables and the $\frac{1}{2}$-Hölder regularity with respect to the time variable we have that 
\begin{flalign*}
    | I_{t}(\xi) | &\leq   |g(t_{i}, \mathcal{Y}_{i}^{\pi}(\xi), \mathcal{Z}_{i}^{\pi}(\xi) )  -  g(t_{i}, \mathcal{Y}_{t_{i}}(\xi), \overline{\mathcal{Z}}_{i}^{\pi}(\xi) ) \big| \\& \quad + \big| g(t_{i}, \mathcal{Y}_{t_{i}}(\xi), \overline{\mathcal{Z}}_{i}^{\pi}(\xi) ) - g(t, \mathcal{Y}_{t}(\xi), \mathcal{Z}_{t}(\xi) ) \big|  
    \\ & \leq [g]_{L} \Big\{ |\Delta \mathcal{Y}_{i}^{\pi}(\xi) | + |\Delta \mathcal{Z}_{i}^{\pi}(\xi)| + \sqrt{\Delta t_{i}} + |\mathcal{Y}_{t_{i}}(\xi) - \mathcal{Y}_{t}(\xi)| + |\overline{\mathcal{Z}}_{i}^{\pi}(\xi) - \mathcal{Z}_{t}(\xi) |  \Big\}.
\end{flalign*}
Squaring and integrating the previous inequality we arrive at 
\begin{flalign*}
    \mathbb{E} \Big[ \int_{t_{i}}^{t_{i+1}} | I_{t}(\xi) |^{2} dt \Big] \leq  C_{0} \Big\{ \Delta t_{i} \mathbb{E} \Big[ |\Delta \mathcal{Y}_{i}^{\pi}(\xi) |^{2} & + |\Delta \mathcal{Z}_{i}^{\pi}(\xi) |^{2} \Big] + \Delta t_{i} \mathcal{R}^{Y}(\xi, \pi) + |\Delta t_{i}|^{2} \\ & + \mathbb{E} \Big[ \int_{t_{i}}^{t_{i+1}} |\mathcal{Z}_{t}(\xi) - \overline{\mathcal{Z}}_{i}^{\pi}(\xi)|^{2} dt \Big] \Big\}
\end{flalign*}
for some constant $C_{0}$ depending only on $[g]_{L}$. Plugging this into (\ref{inequality final}) yields
\begin{flalign*}
    \mathbb{E} \Big[ & | \Delta \mathcal{Y}_{i}^{\pi}(\xi) |^{2}  + \Delta t_{i} |\Delta \mathcal{Z}_{i}^{\pi}(\xi)  |^{2} \Big]  \leq \big(1+\frac{\Delta t_{i}}{\epsilon} \big) \mathbb{E} |  \Delta \mathcal{Y}_{i+1}^{\pi} (\xi)|^{2} 
    \\ & + C_{0} \big(\Delta t_{i} + \epsilon \big) \Delta t_{i} \mathbb{E} \Big[ |\Delta \mathcal{Y}_{i}^{\pi} (\xi)|^{2} + |\Delta \mathcal{Z}_{i}^{\pi} (\xi)|^{2} \Big] \\ &
    +C_{0} \big( \Delta t_{i} + \epsilon \big) \Big\{  \Delta t_{i}  \mathcal{R}^{Y}(\xi, \pi) + |\Delta t_{i}|^{2} + \mathbb{E} \Big[ \int_{t_{i}}^{t_{i+1}} |\mathcal{Z}_{t}(\xi) - \overline{\mathcal{Z}}_{i}^{\pi}(\xi)|^{2} dt \Big] \Big\}.
\end{flalign*}
If we set $\epsilon \coloneqq 1/4C_{0}$ and assume that $|\pi| \leq \epsilon$, so that $C_{0}(\epsilon + \Delta t_{i}) \leq \frac{1}{2}$, then 
\begin{flalign*}
    \mathbb{E} \Big[  \Big(1 - \frac{\Delta t_{i}}{2} \Big) | \Delta \mathcal{Y}_{i}^{\pi}(\xi) |^{2} & + \frac{\Delta t_{i}}{2}  |\Delta \mathcal{Z}_{i}^{\pi}(\xi)  |^{2} \Big]  \leq \big(1+C\Delta t_{i} \big) \mathbb{E} |  \Delta \mathcal{Y}_{i+1}^{\pi} (\xi)|^{2} 
    \\ &  +C \Big\{  \Delta t_{i}  \mathcal{R}^{Y}(\xi, \pi) + |\Delta t_{i}|^{2} + \mathbb{E} \Big[ \int_{t_{i}}^{t_{i+1}} |\mathcal{Z}_{t}(\xi) - \overline{\mathcal{Z}}_{i}^{\pi}(\xi)|^{2} dt \Big] \Big\}.
\end{flalign*}
Finally, dividing by $(1-\Delta t_{i})$ on both sides and using that $(1+C\Delta t_i)/(1-\frac{\Delta t_i}{2}) \leq 1 + C'\Delta t_i$, one easily arrives at (\ref{useful inequality}).
\end{proof}

\begin{proof}[Proof of Theorem \ref{convergence Euler}]
    We use a discrete version of Gronwall's inequality (see Lemma 5.4 in \cite{zhang_bsde_method}) to the inequality provided in Lemma \ref{technical lemma}, with 
    \begin{flalign*}
        a_{i} = \mathbb{E} | \Delta \mathcal{Y}_{i}^{\pi}(\xi) |^{2}, \quad b_{i} =  \Delta t_{i-1} \mathcal{R}^{Y}(\xi, \pi) + |\Delta t_{i-1}|^{2} +  \mathbb{E} \Big[ \int_{t_{i-1}}^{t_{i}} |\mathcal{Z}_{t}(\xi) - \overline{\mathcal{Z}}_{i-1}^{\pi}(\xi)|^{2} dt \Big].
    \end{flalign*}
    Since $a_{n} = 0$, this leads us to 
\begin{flalign}
    \max_{0 \leq i \leq n-1}\mathbb{E} | \Delta \mathcal{Y}_{i}^{\pi}(\xi) |^{2} &\leq C \Big\{ \mathcal{R}^{Y}(\xi, \pi) + |\pi| + \mathcal{R}^{Z}(\xi, \pi) \Big\}.  \label{inequality convergence Y}
\end{flalign}
By using that 
\begin{gather*}
    \max_{0 \leq i \leq n-1} \mathbb{E} \Big[ \sup_{t \in [t_{i}, t_{i+1}]} | \mathcal{Y}_{t}(\xi) - \mathcal{Y}_{i}^{\pi}(\xi)|^{2} \Big] \leq C \Big\{ \mathcal{R}^{Y}(\xi, \pi) + \max_{0 \leq i \leq n-1} \mathbb{E} | \mathcal{Y}_{t_{i}}(\xi) - \mathcal{Y}_{i}^{\pi}(\xi) |^{2} \Big\},
\end{gather*}
we obtain the first inequality. Now, if we sum over $i = 0, \dots, n-1$ in (\ref{useful inequality}) and take into account that $\Delta \mathcal{Y}_{n}^{\pi}(\xi) = 0$ we obtain
\begin{flalign*}
    \sum_{i=0}^{n-1} \mathbb{E} | \Delta \mathcal{Y}_{i}^{\pi}(\xi) |^{2} + \frac{1}{2}\sum_{i=0}^{n-1} \Delta t_{i} |\Delta \mathcal{Z}_{i}^{\pi}(\xi) |^{2} & \leq  \sum_{i=1}^{n-1} \mathbb{E} | \Delta \mathcal{Y}_{i}^{\pi}(\xi) |^{2}  + C \sum_{i=1}^{n-1} \Delta t_{i-1} \mathbb{E} | \Delta \mathcal{Y}_{i}^{\pi}(\xi) |^{2} \\ & + C \mathcal{R}^{Y}(\xi, \pi) + C|\pi| + C \mathcal{R}^{Z}(\xi, \pi) ,
\end{flalign*}
and thus
\begin{flalign*}
     \sum_{i=0}^{n-1} \Delta t_{i} |\Delta \mathcal{Z}_{i}^{\pi}(\xi) |^{2} \leq C \sum_{i=1}^{n-1} \Delta t_{i-1} \mathbb{E} | \Delta \mathcal{Y}_{i}^{\pi}(\xi) |^{2} + C \mathcal{R}^{Y}(\xi, \pi) + C|\pi| + C \mathcal{R}^{Z}(\xi, \pi) .
\end{flalign*}
By using (\ref{inequality convergence Y}) we arrive at 
\begin{gather*}
    \sum_{i=0}^{n-1} \Delta t_{i} |\Delta \mathcal{Z}_{i}^{\pi}(\xi) |^{2} \leq C \Big\{ \mathcal{R}^{Y}(\xi, \pi) + |\pi| + \mathcal{R}^{Z}(\xi, \pi) \Big\}.
\end{gather*}
Finally, note that 
\begin{flalign*}
    \sum_{i=0}^{n-1} \mathbb{E} \Big[ \int_{t_{i}}^{t_{i+1}} | \mathcal{Z}_{t}(\xi) - \mathcal{Z}_{i}^{\pi}(\xi) |^{2} dt \Big] & \leq C \Big\{ \sum_{i=0}^{n-1} \mathbb{E} \Big[ \int_{t_{i}}^{t_{i+1}} | \mathcal{Z}_{t}(\xi) - \overline{\mathcal{Z}}_{i}^{\pi}(\xi) |^{2} dt \Big] \\ & + \sum_{i=0}^{n-1} \mathbb{E} \Big[ \int_{t_{i}}^{t_{i+1}} | \Delta  \mathcal{Z}_{i}^{\pi} (\xi)|^{2} dt \Big] \Big\} \\
    & = C \Big\{ \mathcal{R}^{Z}(\xi, \pi) + \sum_{i=0}^{n-1} \Delta t_{i} | \Delta  \mathcal{Z}_{i}^{\pi} (\xi) |^{2} \Big\} \\& \leq C \Big\{ \mathcal{R}^{Z}(\xi, \pi) + |\pi| + \mathcal{R}^{Y}(\xi, \pi) \Big\}.
\end{flalign*}
\end{proof}

\begin{corollary}[Convergence of the operator Euler scheme for BSDEs]
    Let $(\pi_{m})_{m \geq 1}$ be a sequence of partitions of $[0,T]$ such that $|\pi_{m}| \to 0$ when $m \to \infty$. Then, under the setting of Theorem \ref{convergence Euler}, the Operator Euler scheme converges to the solution operator, meaning that for every $\xi \in L^{2}(\mathcal{F}_{T})$, 
    \begin{gather*}
        \varepsilon(\xi, \pi_{m}) \to 0 \quad \text{when $m \to \infty$.}
    \end{gather*}
    If we furthermore have that $\pi_{m+1} \subset \pi_{m}$, then the convergence is uniform on compact subsets of $L^{2}(\mathcal{F}_{T})$. That is, for any compact $K \subset L^{2}(\mathcal{F}_{T})$, we have 
    \begin{gather*}
        \sup_{ \xi \in K} \abs{\varepsilon(\xi, \pi_{m})} \to 0 \quad \text{when $m \to \infty$.}
    \end{gather*}
\end{corollary}

\begin{remark}
    As we can see in Theorem \ref{convergence Euler}, the rate of convergence of the Operator Euler scheme is at most $1/2$, and is determined by the maps $\mathcal{R}^{Y}, \mathcal{R}^{Z}$. An interesting question is then under which conditions we can obtain an inequality of the type 
    \begin{gather*}
        \max\{\mathcal{R}^{Y}(\xi, \pi), \mathcal{R}^{Z}(\xi, \pi) \} \leq C |\pi|,
    \end{gather*}
    with $C>0$ possibly depending on $\xi$ but independent of $\pi$. The most well-known conditions for this are given when $\xi$ is of the form $f(X_{\cdot})$, where $f$ is a functional satisfying a Lipschitz condition with respect to the $L^{\infty}$ norm, and $X$ is the solution to a forward diffusion, see \cite{zhang_bsde_method}.

    The same result is proved in \cite{HuYaozhong2011MCFB}, under the assumptions that $\xi$ is twice differentiable in the sense of Malliavin, and the first and second derivatives satisfy some integrability conditions, as well as some additional assumptions on the regularity of the generator $g$.  
\end{remark}

\subsection{Markovian maps in the Operator Euler scheme}\label{Section: Markovian maps in the Operator Euler scheme}

Just like in the backward Euler scheme for BSDEs, one can find measurable maps that relate the operators $\mathcal{Y}_{i}^{\pi}$ and $\mathcal{Z}_{i}^{\pi}$ to $X_{t_{i}}$. 

\begin{theorem}\label{theorem markovian maps discrete case}
    For each $i \in \{0, \dots, n\}$ there exist measurable maps $\mathscr{Y}_{i}^{\pi}\colon \mathbb{R}^{\mathcal{A}} \times L^{2}(\mathcal{F}_{T}) \to \mathbb{R}$ and $\mathscr{Z}_{i}^{\pi}\colon \mathbb{R}^{\mathcal{A}} \times L^{2}(\mathcal{F}_{T}) \to \mathbb{R}^{d}$ such that 
\begin{gather*}
    \mathcal{Y}_{i}^{\pi}(\xi) = \mathscr{Y}_{i}^{\pi}(X_{t_{i}}, \xi), \quad \mathcal{Z}_{i}^{\pi}(\xi) = \mathscr{Z}_{i}^{\pi}(X_{t_{i}}, \xi), \quad \text{$\mathbb{P}$-a.s.}
\end{gather*}
Moreover, if $\xi$ belongs to $\Pi_{p,M}(L^{2}(\mathcal{F}_{T}))$, then the maps $(\mathscr{Y}_{i}^{\pi}, \mathscr{Z}_{i}^{\pi})_{0 \leq i \leq n}$ admit a finite-dimensional representation, which only depends on $X^{\leq p,M}$ and $\xi^{\leq p,M}$.
\end{theorem}
\begin{proof}
    We proceed by backward induction. For $i=n$, we have that $\mathscr{Y}_{n}^{\pi}(x, \xi) \coloneqq \mathscr{Y}(t_{n}, x, \xi)$ and $\mathscr{Z}_{n}^{\pi}(x, \xi) \coloneqq 0$, which are clearly measurable. 
    
    Assume now that this holds for $i=k+1$. That is, $\mathcal{Y}_{k+1}^{\pi}(\xi) = \mathscr{Y}_{k+1}^{\pi}(X_{t_{k+1}}, \xi)$ and $\mathcal{Z}_{k+1}^{\pi}(\xi) = \mathscr{Z}_{k+1}^{\pi}(X_{t_{k+1}}, \xi)$ $\mathbb{P}$-a.s., with both maps measurable. Then 
\begin{flalign*}
    \mathcal{Z}_{k}^{\pi}(\xi) = \frac{1}{\Delta t_{k}} \mathbb{E}_{t_{k}} \Big[\mathscr{Y}_{k+1}^{\pi}(X_{t_{k+1}}, \xi) \Delta B_{k} \Big].
\end{flalign*}
By Lemma \ref{markov property X infinite dimensional} there is a measurable map $\varphi$ such that $X_{t_{k+1}} = \varphi(X_{t_{k}}, \{ B_{u} - B_{t_{k}}\}_{u \in [t_{k}, t_{k+1}] })$ $\mathbb{P}$-a.s. We therefore have that 
\begin{gather*}
    \mathcal{Z}_{k}^{\pi}(\xi) = \frac{1}{\Delta t_{k}} \mathbb{E}_{t_{k}} \Big[ \mathscr{Y}_{k+1}^{\pi} \big(\varphi(X_{t_{k}}, \{ B_{u} - B_{t_{k}}\}_{u \in [t_{k}, t_{k+1}] }), \xi \big) \Delta B_{k} \Big].
\end{gather*}
 Then, by the freezing lemma \cite[Lemma 4.1]{BaldiPaolo2017SCAI}, there exists a measurable function $\mathscr{Z}_{k}^{\pi}\colon \mathbb{R}^{\mathcal{A}} \times L^{2}(\mathcal{F}_{T}) \to \mathbb{R}^{d}$ such that $\mathcal{Z}_{k}^{\pi}(\xi) = \mathscr{Z}_{k}^{\pi}(X_{t_{k}}, \xi)$ $\mathbb{P}$-a.s. Similarly, 
\begin{gather*}
    \mathcal{Y}_{k}^{\pi}(\xi) \coloneqq \mathbb{E}_{t_{k}} \Big[ \mathscr{Y}_{k+1}^{\pi}(X_{t_{k+1}}, \xi)\Big]+ g(t_{k}, \mathcal{Y}_{k}^{\pi}(\xi), \mathscr{Z}_{k}^{\pi}(X_{t_{k}}, \xi)) \Delta t_{k} .
\end{gather*}
Now, by using Picard iterations to solve the equation, induction and the freezing lemma we conclude that $\mathcal{Y}_{k}^{\pi}(\xi) = \mathscr{Y}_{k}^{\pi}(X_{t_{k}}, \xi)$ $\mathbb{P}$-a.s. for some measurable $\mathscr{Y}_{k}^{\pi}$. 

The last statement of the theorem is proved in a similar way. 
\end{proof}

Theorem \ref{theorem markovian maps discrete case} shows that instead of directly computing the operators $\mathcal{Y}_{i}^{\pi}$ and $\mathcal{Z}_{i}^{\pi}$, we can equivalently find the functionals $\mathscr{Y}_{i}^{\pi}$ and $\mathscr{Z}_{i}^{\pi}$. By the $L^{2}$-projection rule, this can be done by setting $\mathscr{Y}_{n}^{\pi}(X_{t_{n}}, \xi) = \xi$, and defining recursively, for $i=n-1, \dots, 0$,
\begin{flalign*}
    \mathscr{Z}_{i}^{\pi}(\cdot, \xi) &\coloneqq \argmin_{V_{i}(\cdot, \xi)} \mathbb{E} \Big| V_{i}(X_{t_{i}}, \xi) - \frac{1}{\Delta t_{i}} \mathscr{Y}_{i+1}^{\pi}(X_{t_{i+1}}, \xi) \Delta B_{t_{i}} \Big|^{2}  \\
   \mathscr{Y}_{i}^{\pi}(\cdot, \xi) &\coloneqq \argmin_{U_{i}(\cdot, \xi)} \mathbb{E}  \Big| U_{i}(X_{t_{i}}, \xi) -\mathscr{Y}_{i+1}^{\pi}(X_{t_{i+1}}, \xi) - \Delta t_{i} g(t_{i}, U_{i}(X_{t_{i}}, \xi), \mathscr{Z}_{i}^{\pi}(X_{t_{i}}, \xi)) \Big|^{2} .
\end{flalign*}
Given the structure of the previous optimization problems, we can solve them by classical linear least squares regression methods. 
Instead, if we are interested in using nonlinear techniques, like deep learning based methods, then it is more efficient to solve 
\begin{flalign*}
    (\mathscr{Y}_{i}^{\pi}, \mathscr{Z}_{i}^{\pi})(\cdot, \xi) \coloneqq \argmin_{(V_{i}, U_{i})(\cdot, \xi)} & \mathbb{E} \Big| U_{i}(X_{t_{i}}, \xi) -\mathscr{Y}_{i+1}^{\pi}(X_{t_{i+1}}, \xi) \\ & - \Delta t_{i} g(t_{i}, U_{i}(X_{t_{i}}, \xi), V_{i}(X_{t_{i}}, \xi)) + V_{i}(X_{t_{i}}, \xi) \cdot \Delta B_{i} \Big|^{2}  ,
\end{flalign*}
which achieves it's minimum at $0$ (see DBDP1 in \cite{HureCome2020Dbsf}). We remark that the solution to both optimization problems coincide. 

\subsubsection{Finite-dimensional approximation}\label{section finite-dimensional approximation}

The previous theorem provides a way of approximating the solution operators at a discrete set of time points by solving some optimization problems within a space of functionals. However, whenever we intend to actually implement the method, we need to approximate these functionals by some finite-dimensional maps. 

Let us define $(\mathscr{Y}_{i}^{\pi, p, M}, \mathscr{Z}_{i}^{\pi, p, M}) \colon \mathbb{R}^{\mathcal{A}} \times L^{2}(\mathcal{F}_{T}) \to \mathbb{R} \times \mathbb{R}^{d}$ by
\begin{gather}\label{finite-dimensional maps approximation}
    \mathscr{Y}_{i}^{\pi, p, M}(x, \xi) \coloneqq \mathscr{Y}_{i}^{\pi}( x, \Pi_{p,M}(\xi)), \quad \mathscr{Z}_{i}^{\pi, p, M}(x, \xi) \coloneqq \mathscr{Z}_{i}^{\pi}( x, \Pi_{p,M}(\xi)).
\end{gather}

\begin{remark}\label{remark finite dimensional domain}
    By the second part of Theorem \ref{theorem markovian maps discrete case}, we have that the maps (\ref{finite-dimensional maps approximation}) have a finite-dimensional representation. More specifically, they are equivalent to some maps which only depend (possibly in a nonlinear way) on $X^{\leq p,M}$ and $\xi^{\leq p,M}$.
\end{remark}

The following theorem bounds the error of the above finite-dimensional maps in terms of the truncation error of the chaos expansion. 

\begin{theorem}\label{theorem error bound finite-dim appr}
    Let $p, M \in \mathbb{N}$. Then, for all $\xi \in L^{2}(\mathcal{F}_T)$, we have that
\begin{align*}
    &\max_{0 \leq i \leq n} \mathbb{E} \big|\mathscr{Y}_{i}^{\pi}(X_{t_{i}}, \xi) - \mathscr{Y}_{i}^{\pi, p, M}(X_{t_{i}}, \xi)\big|^{2} \\
    &\quad + \sum_{i=0}^{n-1} \mathbb{E} \Bigg[ \int_{t_{i}}^{t_{i+1}} \big|\mathscr{Z}_{i}^{\pi}(X_{t_{i}}, \xi) - \mathscr{Z}_{i}^{\pi, p, M}(X_{t_{i}}, \xi)\big|^{2} \, dt \Bigg] \leq C \mathbb{E} \abs{\xi - \Pi_{p, M}(\xi)}^{2},
\end{align*}
where $C>0$ is a constant depending only on $[g]_L$ and $T$.
\end{theorem}

\begin{proof}
The previous statement is equivalent to 
\begin{flalign}\notag
    \max_{0 \leq i \leq n} \mathbb{E} & |\mathcal{Y}_{i}^{\pi}(\xi) - \mathcal{Y}_{i}^{\pi, p, M}(\xi)|^{2} \\  & +  \sum_{i=0}^{n-1}  \mathbb{E} \Big[ \int_{t_{i}}^{t_{i+1}} \abs{\mathcal{Z}_{i}^{\pi}(\xi) - \mathcal{Z}_{i}^{\pi, p, M}(\xi) }^{2} dt  \Big]   \leq C \mathbb{E} \abs{\xi - \Pi_{p, M}(\xi)}^{2}, \label{main ineq theorem equiv}
\end{flalign}
where we set 
\begin{gather*}
    \mathcal{Y}_{i}^{\pi, p, M} \coloneqq \mathcal{Y}_{i}^{\pi} \circ \Pi_{p,M}, \quad \mathcal{Z}_{i}^{\pi, p, M} \coloneqq \mathcal{Z}_{i}^{\pi} \circ \Pi_{p,M}.
\end{gather*}
For ease of notation, let us fix $\xi \in K$, and let 
\begin{gather*}
    \Delta \mathcal{Y}_{i}^{\pi, p, M} \coloneqq \mathcal{Y}_{i}^{\pi}(\xi) - \mathcal{Y}_{i}^{\pi, p, M}(\xi) \quad \text{and} \quad \Delta \mathcal{Z}_{i}^{\pi, p, M} \coloneqq \mathcal{Z}_{i}^{\pi}(\xi) - \mathcal{Z}_{i}^{\pi, p, M}(\xi).
\end{gather*}
Let us first prove that there exists a constant $C'>0$, depending only on $[g]_{L}$ and $T$, such that for any $i = 0, \dots, n-1$ we have
\begin{gather}\label{inequality yz}
    \mathbb{E} \big[ |\Delta \mathcal{Y}_{i}^{\pi, p, M}|^{2} + \frac{\Delta t_{i}}{2} \mathbb{E} |\Delta \mathcal{Z}_{i}^{\pi, p, M}|^{2} \big] \leq (1+C' \Delta t_{i}) \mathbb{E} |\Delta \mathcal{Y}_{i+1}^{\pi, p, M}|^{2} .
\end{gather}
The proof follows similar arguments to those of Theorem \ref{convergence Euler}, so we will omit some details. First of all, by the Martingale Representation Theorem, there exist processes $(\widehat{\mathcal{Z}}_{t}^{\pi}), (\widehat{\mathcal{Z}}_{t}^{\pi, p, M})$ such that
\begin{gather}\label{mrt 2}
    \mathcal{Y}_{i+1}^{\pi} = \mathbb{E}_{t_{i}} \big[ \mathcal{Y}_{i+1}^{\pi} \big] + \int_{t_{i}}^{t_{i+1}} \widehat{\mathcal{Z}}_{t}^{\pi} dB_{t} , \quad \mathcal{Y}_{i+1}^{\pi, p, M} = \mathbb{E}_{t_{i}} \big[ \mathcal{Y}_{i+1}^{\pi, p, M} \big] + \int_{t_{i}}^{t_{i+1}} \widehat{\mathcal{Z}}_{t}^{\pi, p, M} dB_{t},
\end{gather}
which easily yields the recursive formulas
\begin{flalign}
    \mathcal{Y}_{i}^{\pi} &= \mathcal{Y}_{i+1}^{\pi} + \Delta t_{i} g(t_{i}, \mathcal{Y}_{i}^{\pi}, \mathcal{Z}_{i}^{\pi})  - \int_{t_{i}}^{t_{i+1}} \widehat{\mathcal{Z}}_{t}^{\pi} dB_{t}, \label{formula for Y 2 1} \\
    \mathcal{Y}_{i}^{\pi, p, M} &= \mathcal{Y}_{i+1}^{\pi, p, M} + \Delta t_{i} g(t_{i}, \mathcal{Y}_{i}^{\pi, p, M}, \mathcal{Z}_{i}^{\pi, p, M})  - \int_{t_{i}}^{t_{i+1}} \widehat{\mathcal{Z}}_{t}^{\pi, p, M} dB_{t}. \label{formula for Y 2 2}
\end{flalign}
Then, using the definition of $\mathcal{Z}_{i}^{\pi}$ and $\mathcal{Z}_{i}^{\pi, p, M}$, (\ref{mrt 2}) and Itô's isometry, we arrive at
\begin{gather}\label{Z_t_2}
    \mathcal{Z}_{i}^{\pi} = \frac{1}{\Delta t_{i} } \mathbb{E}_{t_{i}} \Big[\int_{t_{i}}^{t_{i+1}} \widehat{\mathcal{Z}}_{t}^{\pi} dt \Big], \quad \mathcal{Z}_{i}^{\pi, p, M} = \frac{1}{\Delta t_{i} } \mathbb{E}_{t_{i}} \Big[\int_{t_{i}}^{t_{i+1}} \widehat{\mathcal{Z}}_{t}^{\pi, p, M} dt \Big].
\end{gather}
Next, by (\ref{Z_t_2}), one can easily see that 
\begin{flalign}
    \mathbb{E} | \Delta \mathcal{Z}_{i}^{\pi, p, M} |^{2} \leq  \frac{1}{\Delta t_{i}} \mathbb{E}  \Big[  \int_{t_{i}}^{t_{i+1}} |\Delta \overline{\mathcal{Z}}_{t}^{\pi, p, M}|^{2} dt  \Big], \quad \Delta \overline{\mathcal{Z}}_{t}^{\pi, p, M} \coloneqq \widehat{\mathcal{Z}}_{t}^{\pi} - \widehat{\mathcal{Z}}_{t}^{\pi, p, M}. \label{inequality Z 2}
\end{flalign}
Using (\ref{formula for Y 2 1}) and (\ref{formula for Y 2 2}), we have 
\begin{gather*}
    \Delta \mathcal{Y}_{i}^{\pi, p, M} = \Delta \mathcal{Y}_{i+1}^{\pi, p, M} + \Delta t_{i} I_{i}^{p,M} - \int_{t_{i}}^{t_{i+1}} \Delta \overline{\mathcal{Z}}_{t}^{\pi, p, M} dB_{t},  
\end{gather*}
where $I_{i}^{p,M} \coloneqq g(t_{i}, \mathcal{Y}_{i}^{\pi}, \mathcal{Z}_{i}^{\pi})  - g(t_{i}, \mathcal{Y}_{i}^{\pi, p, M}, \mathcal{Z}_{i}^{\pi, p, M}) $. The previous expression yields
\begin{gather}\label{useful equality Y 2}
    \mathbb{E} | \Delta \mathcal{Y}_{i}^{\pi, p, M} |^{2} = \mathbb{E} \Big[ \Big( \Delta \mathcal{Y}_{i+1}^{\pi, p, M}  + \Delta t_{i} I_{i}^{p,M} \Big)^{2} - \int_{t_{i}}^{t_{i+1}} | \Delta \overline{\mathcal{Z}}_{t}^{\pi, p, M}|^{2} dt \Big].
\end{gather}
The inequality (\ref{inequality Z 2}), equation (\ref{useful equality Y 2}), and Young's inequality $(a+b)^{2} \leq (1+ \frac{\Delta t_{i}}{\gamma}) a^{2} + (1+\frac{\gamma}{\Delta t_{i}}) b^{2}$, provide
\begin{flalign}
    \mathbb{E} \Big[ | \Delta \mathcal{Y}_{i}^{\pi, p, M} |^{2} + \Delta t_{i} |\Delta \mathcal{Z}_{i}^{\pi, p, M} |^{2} \Big] \leq \big(1+\frac{\Delta t_{i}}{\gamma} \big) \mathbb{E} | \Delta \mathcal{Y}_{i+1}^{\pi, p, M} |^{2} + \big(\Delta t_{i}+ \gamma \big) \Delta t_{i} \mathbb{E} | I_{i}^{p,M} |^{2} .\label{inequality final 2}
\end{flalign}
Observe that, by the globally Lipschitz property of $g$ with respect to the spatial variables, we have that 
\begin{flalign*}
    E| I_{i}^{p,M} |^{2} \leq C_{0} \Big\{ \mathbb{E} | \Delta \mathcal{Y}_{i}^{\pi, p, M} |^{2} + \mathbb{E} | \Delta \mathcal{Z}_{i}^{\pi, p, M} |^{2} \Big\}, 
\end{flalign*}
for some $C_{0}$ depending only on $[g]_{L}$. Plugging this into (\ref{inequality final 2}), choosing $\gamma = 1/4C_{0}$ and setting $|\pi| \leq \gamma$, so that $C_{0}(\gamma + \Delta t_{i}) \leq \frac{1}{2}$, we obtain the desired inequality (\ref{inequality yz}). 

We now proceed to prove (\ref{main ineq theorem equiv}). By the discrete version of Gronwall´s inequality \cite[Lemma 5.4]{zhang_bsde_method}, we get 
\begin{gather}\label{result fa Y}
    \max_{0 \leq i \leq n} \mathbb{E} |\Delta \mathcal{Y}_{i}^{\pi,p,M} |^{2} \leq C \mathbb{E} \abs{\xi - \Pi_{p, M}(\xi)}^{2}.
\end{gather}
Now, if we sum over $i=0, \dots, n-1$ in (\ref{inequality yz}), we arrive at 
\begin{gather*}
    \sum_{i=0}^{n-1} \frac{\Delta t_{i}}{2} |\Delta \mathcal{Z}_{i}^{\pi,p,M} |^{2} \leq \mathbb{E} \abs{\xi - \Pi_{p, M}(\xi)}^{2} + C \sum_{i=1}^{n} \Delta t_{i-1} \mathbb{E} |\Delta \mathcal{Y}_{i}^{\pi,p,M} |^{2} , 
\end{gather*}
which together with (\ref{result fa Y}) implies 
\begin{gather*}
     \sum_{i=0}^{n-1} \Delta t_{i} |\Delta \mathcal{Z}_{i}^{\pi} |^{2} \leq C \mathbb{E} \abs{\xi - \Pi_{p, M}(\xi)}^{2}.
\end{gather*}
\end{proof}

An immediate consequence of the previous theorem is that the above finite-dimensional maps converge uniformly on compact subsets of $L^{2}(\mathcal{F}_{T})$.

\begin{corollary}
    Let $K \subset L^{2}(\mathcal{F}_{T})$ be a compact subset. Then, for all $\epsilon > 0$ there exist $p^{\epsilon}, M^{\epsilon} \in \mathbb{N}$ such that for all $p \geq p^{\epsilon}$ and $M \geq M^{\epsilon}$,
\begin{align*}
    &\max_{0 \leq i \leq n} \mathbb{E} \big|\mathscr{Y}_{i}^{\pi}(X_{t_{i}}, \xi) - \mathscr{Y}_{i}^{\pi, p, M}(X_{t_{i}}, \xi)\big|^{2} \\
    &\quad + \sum_{i=0}^{n-1} \mathbb{E} \Bigg[ \int_{t_{i}}^{t_{i+1}} \big|\mathscr{Z}_{i}^{\pi}(X_{t_{i}}, \xi) - \mathscr{Z}_{i}^{\pi, p, M}(X_{t_{i}}, \xi)\big|^{2} \, dt \Bigg] \leq \epsilon
\end{align*}
for all $\xi \in K$.
\end{corollary}

\begin{proof}
Since $L^{2}(\mathcal{F}_{T})$ is a Hilbert space, the chaos decomposition provides a Schauder basis. This means that, for any $\epsilon > 0$, we can choose $p^\epsilon$ and $M^\epsilon$ such that, for all $p \geq p^{\epsilon}$ and $M \geq M^{\epsilon}$,
\begin{gather*}
   \mathbb{E} \abs{\mathcal{Y}_{n}^{\pi}(\xi) - \mathcal{Y}_{n}^{\pi, p, M}(\xi)}^{2} = \mathbb{E} \abs{\xi - \Pi_{p, M}(\xi)}^{2} \leq \epsilon \quad \forall \xi \in K.
\end{gather*}
One then concludes by using Theorem \ref{theorem error bound finite-dim appr}.
\end{proof}

\begin{remark}
    For an arbitrary terminal condition $\xi \in L^{2}(\mathcal{F}_T)$, one generally cannot obtain a quantitative convergence rate for the truncation error of the Wiener–chaos expansion. However, such rates are available under additional Malliavin-smoothness assumptions (see \cite{BriandPhilippe2014SOBB}). To keep the framework as broad as possible, we do not impose these assumptions here.   
\end{remark}

\section{Description of the algorithm}\label{section 5}

In this section we describe an implementation of the algorithm that we use to numerically approximate the solution operators $\mathcal{Y}$ and $\mathcal{Z}$, which is based on the Operator Euler scheme introduced and analyzed in the previous section.  

\subsection{Choice for the truncated basis of $L^{2}([0,T]; \mathbb{R}^{d})$}

Let $\overline{\pi} = \{\overline{t}_{0} = 0 < \dots < \overline{t}_{M} = T\}$ be a time partition of $[0,T]$. Set $\Delta \overline{t}_{i} \coloneqq \overline{t}_{i+1}-\overline{t}_{i}$. Following \cite{BriandPhilippe2014SOBB}, we consider
\begin{gather*}
    h_{i}^{j}(t) = \frac{1}{\sqrt{\Delta \overline{t}_{i}}} \mathbf{1}_{(\overline{t}_{i-1}, \overline{t}_{i}]}(t) e_{j}, \quad i=1, \dots, M, \quad j=1, \dots, d,
\end{gather*}
where $\{e_{j}\}_{1 \leq j \leq d}$ is the canonical basis of $\mathbb{R}^{d}$. This can be seen as a truncated orthonormal basis of $L^{2}([0,T]; \mathbb{R}^{d})$, since we can complete it for example with the Haar basis in each sub-interval. This choice makes the simulation of the integrals appearing in (\ref{infinite stochastic process hermite}) very simple, since we have that 
\begin{gather*}
    \int_{0}^{t} h_{i}^{j}(s) \cdot dB_{s} = \mathbf{1}_{\{ \overline{t}_{i-1} \leq t \}} \frac{1}{\sqrt{\Delta \overline{t}_{i}}} (B_{t}^{j} - B_{\overline{t}_{i-1}}^{j})
\end{gather*}
for $j=1,\dots,d$.

\subsection{Implementation}

The notation that we use in the algorithm is the following:
\begin{itemize}
    \item $d$: dimension of the Brownian motion;
    \item $\pi = \{t_{0} = 0 < \dots < t_{n} = T\}$: partition of $[0,T]$ used for the Operator Euler scheme; 
    \item $p$: maximum order of the chaos decomposition;
    \item $\overline{\pi} = \{\overline{t}_{0} = 0 < \dots < \overline{t}_{M} = T\}$: partition of $[0,T]$ used for the truncation of the basis of $L^{2}([0,T]; \allowbreak\mathbb{R}^{d})$ in the chaos decomposition;
    \item $Md$ will be the number of elements in the truncated basis of $L^{2}([0,T];\mathbb{R}^{d})$;
    \item $\mathcal{A}_{p, M, d}\coloneqq \{ a \in \mathcal{A} \ : \ |a| \leq p, a_{k} = 0 \ \forall k > M d \}$. This is the index set considered when truncating the chaos decomposition to level $p$ and taking $M \times d$ elements of the basis of $L^{2}([0,T];\mathbb{R}^{d})$. Its cardinality $\vert\mathcal{A}_{p, M, d}\vert$ indicates the dimension of the space $\Pi_{p,M}(L^{2}(\mathcal{F}_{T}))$, introduced in (\ref{truncation of the chaos decomposition}). 
\end{itemize}

As seen at the end of Section \ref{Section: Markovian maps in the Operator Euler scheme}, for each $\xi \in L^{2}(\mathcal{F}_{T})$, one can find the truncated Markovian maps by solving 
\begin{flalign*}
    (\mathscr{Y}_{i}^{\pi, p, M}, \mathscr{Z}_{i}^{\pi, p, M})(\cdot, \xi) \coloneqq \argmin_{(V_{i}, U_{i})(\cdot, \xi)} & \mathbb{E} \Big| U_{i}(X_{t_{i}}, \xi) -\mathscr{Y}_{i+1}^{\pi, p, M}(X_{t_{i+1}}, \xi) \\ & - \Delta t_{i} g(t_{i}, U_{i}(X_{t_{i}}, \xi), V_{i}(X_{t_{i}}, \xi)) + V_{i}(X_{t_{i}}, \xi) \cdot \Delta B_{i} \Big|^{2}  ,
\end{flalign*}
for $i=n-1, \dots, 0$. Since we want to solve this for a wide range of $\xi$, but do not want to solve it for each one of them separately, we proceed as follows. For each $a \in  \mathcal{A}_{p,M, d}$, choose $ d_{a}^{1}, d_{a}^{2} \in \mathbb{R}$ such that $d_{a}^{1} \leq d_{a}^{2}$. Consider then the compact subspace $K \subset L^{2}(\mathcal{F}_{T})$ given by 
\begin{gather}\label{compact subset}
    K \coloneqq \Big\{ \xi \in L^{2}(\mathcal{F}_{T}) \colon \xi = \sum_{k=0}^{p} \sum_{|a|=k}  d_{a} X_{T}^{a}, \ d_{a} \in [d_{a}^{1}, d_{a}^{2}], \quad a = (a_{1}, \dots, a_{M d}) \Big\}.
\end{gather}
Let $\mu$ denote the uniform measure on $K$. We then solve the optimization problem
\begin{flalign}
     (\mathscr{Y}_{i}^{\pi, p, M} &, \mathscr{Z}_{i}^{\pi, p, M})(\cdot, \cdot) \coloneqq \min_{(U_{i}, V_{i})(\cdot, \cdot)} \int_{\Omega \times K} \Big| U_{i}(X_{t_{i}}(\omega), \xi) -\mathscr{Y}_{i+1}^{\pi, p, M}(X_{t_{i+1}}(\omega), \xi) \notag \\ & - \Delta t_{i} g(t_{i}, U_{i}(X_{t_{i}}(\omega), \xi), V_{i}(X_{t_{i}}(\omega), \xi)) + V_{i}(X_{t_{i}}(\omega), \xi) \cdot \Delta B_{i} \Big|^{2} \mathbb{P}(d\omega) \otimes \mu(d\xi). \label{second algorithm euler}
\end{flalign}
By the first part of Remark \ref{remark finite dimensional domain}, the maps over which we are optimizing have a representation of the form $U_{i}, V_{i} \colon \mathbb{R}^{\mathcal{A}_{p, M, d}} \times \Pi_{p,Md}(L^{2}(\mathcal{F}_{T})) \to \mathbb{R} \times \mathbb{R}^{d}$. 

Since the dimension of the domain is proportional to $|\mathcal{A}_{p, M,d} \vert$, which by nature is quite high, we choose to work with neural networks. These have been already used successfully in finding the Markovian maps in FBSDEs problems where the forward equation is of high dimensionality, see e.g. \cite{HanJiequn2018Shpd}, \cite{HanJiequn2020Cotd}, \cite{HureCome2020Dbsf}. We refer to this method as the Deep Operator BSDE.

\subsection{Computational limitations of the method}

A key limitation of the proposed approach is the \emph{curse of dimensionality}. Recall that the truncation of the Wiener chaos decomposition requires choosing a maximum order $p$ and a finite number $M d$ of basis functions in $L^{2}([0,T];\mathbb{R}^{d})$. 
The cardinality of the resulting index set $\mathcal{A}_{p,M,d}$ can be obtained by using Stars and Bars and Pascal's identity (see e.g. \cite{FellerWilliam1968Aitp})
\begin{gather}\label{combinatorics cardinality}
    \vert\mathcal{A}_{p, M, d}\vert = \sum_{k=0}^{p} \frac{(M d+k-1)!}{(M d -1)!k!} = \frac{(M d + p)!}{(M d)! p!}.
\end{gather}
This quantity grows rapidly both in the chaos order $p$ and in the product $Md$. For instance, when $p$ is fixed, we have $\big|\mathcal{A}_{p,M,d}\big| = \mathcal{O}((M d)^{p})$, while for fixed $M d$ the growth is $\mathcal{O}(p^{M d})$. 

This impacts the dimensions of the optimization problems we need to solve. At each Euler step $t_{i}$, the approximation of the maps $(\mathscr{Y}_{i}^{\pi,p,M}, \mathscr{Z}_{i}^{\pi,p,M})$ involves neural networks whose input dimension is proportional to $\vert\mathcal{A}_{p,M,d}\vert$. Training these networks often involves millions of parameters, which makes the optimization challenging: convergence can be slow, the results may be sensitive to the choice of hyperparameters, and the computational cost is large. 

In addition, the computational burden also depends on the compact set $K$ introduced in \eqref{compact subset}. The wider the ranges $[d_a^1,d_a^2]$ for the coefficients, the larger the variability in the inputs $\xi$, and hence the more complex the learning problem becomes. Larger sets $K$ typically require richer network architectures, more training samples, and longer optimization times. 

Modern stochastic optimization techniques and hardware accelerators help to manage moderately large problems, but in very high-dimensional settings these issues remain a major limitation of the method.

\subsection{Comparison with existing parametric methods}

As stated in Theorem \ref{theorem markovianity}, when the terminal condition depends only on the terminal value of a forward SDE, the solution admits a simple Markovian representation. If the objective is to learn the solution operator for a \emph{parametric family} of such terminal conditions, see e.g. \cite[Section 2.4]{ElKarouiN.1997BSDE}, the Wiener chaos expansion is not strictly required: one can instead rely directly on the corresponding forward SDE and approximate the Markovian solution operator by including the parameters as additional inputs. This is aligned, in spirit, with the method introduced in \cite{NEURIPS2020_c1714160}.

The main motivation of the present work, however, is to address more general situations. One such case is when the objective is to approximate the solution operator across \emph{several families} of terminal conditions of a forward diffusion, where \emph{several families} refers to functionally different forms of terminal conditions (e.g., call and put options, or both path-dependent and terminal-state dependent payoffs). To our knowledge, this is the first method capable of handling such settings within a unified framework.

\section{Numerical examples}\label{section 6}

All experiments were conducted on a machine running Red Hat Enterprise Linux 8.10, equipped with an Intel(R) Xeon(R) Gold 5120 CPU and an NVIDIA A10 GPU with 24 GB of memory. The algorithm was implemented in Python 3.9.19 using PyTorch 2.3.0 as the main deep learning library. The code is available at the GitHub repository \url{https://github.com/pere98diaz/Deep-Operator-BSDE}.

In order to test the performance of the Deep Operator BSDE, we proceed as follows. We first fix the generator $g$ and a collection of families of terminal conditions. For each of these families, we approximate—by Monte Carlo—the coefficients of the truncated chaos decomposition for several representative terminal conditions. From these computations we extract, for each $a \in \mathcal{A}_{p,M}$, the maximum and minimum values, which then serve as the bounds $d_{a}^{1}$ and $d_{a}^{2}$ in (\ref{compact subset}). In this way we define the compact subset of $L^{2}(\mathcal{F}_{T})$ where the solution operators are approximated. We then evaluate the performance of the Deep Operator BSDE on the original families of terminal conditions.  

These families can be viewed as low-dimensional projections of the high-dimensional set $K$. We emphasize that, although the reported results concern these projections, the Deep Operator BSDE is in fact trained on the much larger set $K$. The method does not see the individual families during training, which underlines that the learned solution operator is genuinely defined on a high-dimensional domain.

As baseline methods, we use either the approach proposed in \cite{BriandPhilippe2014SOBB} or, in the case of linear BSDEs, a direct Monte Carlo approximation. It is worth stressing that these baselines evaluate the solution operator only at a specific terminal condition, whereas our proposed methodology learns the entire solution operator at once. It is therefore natural to expect that the Deep Operator BSDE may not always match perfectly the baseline across the whole domain.  

The time partition $\pi$ used in the Operator Euler scheme is $\pi = \{ T \times \tfrac{i}{10} : 0 \leq i \leq 10 \}$. The choice of $p$ and $\overline{\pi}$ were made empirically by evaluating the ratio
\begin{gather*}
    \frac{\mathbb{E}\abs{\xi - \Pi_{p, Md}(\xi)}^2}{\mathbb{E}\abs{\xi}^2}
\end{gather*}
across the different families of terminal conditions, and selecting the configuration that provided the best balance between accuracy and tractability.  

Following \cite{NEURIPS2020_c1714160}, we assess the discrepancy between the Deep Operator BSDE and the baseline using the scaled relative error
\begin{gather*}
    \text{err}_Y(\xi) = \frac{\abs{\mathcal{Y}_{0}^{\text{DO}}(\xi) - \mathcal{Y}_0^{\text{Base}}(\xi)}}{1+\abs{\mathcal{Y}_0^{\text{Base}}(\xi)}},
\end{gather*}
with an analogous definition for $Z$. Here \textit{DO} stands for the approximation given by the Deep Operator BSDE, and \textit{Base} for the one given by the Baseline. This metric behaves like an absolute error when the baseline is small, and like a relative error when the baseline is large.  

A detailed description of the training hyperparameters and the network architectures used in the Deep Operator BSDE at each time step is provided in the Supplementary Material.

\subsection{Pricing and hedging -- linear generator}

We consider the case of pricing and hedging options under the physical measure in a Black-Scholes setting, which corresponds to the linear generator $g(t,y,z) = -ry - z (\mu-r)/\sigma$, with $r$ being the interest rate, $\mu$ the drift and $\sigma$ the volatility, see e.g. \cite[Section 4.5.1]{ZhangJianfeng2017BSDE}.

The true solution at time $t=0$ is given by 
\begin{gather*}
    Y_{0} = e^{-rT}\mathbb{E}_{\mathbb{Q}} [ \xi ], \quad Z_{0} = \frac{\partial Y_{0}}{\partial S_{0}}\Big\vert_{S_{0} = s_{0}} \times \sigma s_{0},
\end{gather*}
where $\mathbb{Q}$ represents the risk neutral probability measure. We get a benchmark for $Y_{0}$ and $Z_{0}$ by Monte Carlo. The chosen parameters are $r=0.01$, $\sigma = 0.2$, $T=1$ and $s_{0} = 1$.  

For the truncation of the chaos expansion, we set $p=3$ and use the partition $\overline{\pi} = \{ T \times \tfrac{i}{5} : 0 \leq i \leq 5 \}$. Since $d=1$, this means that the number of chaos coefficients that are considered is $|\mathcal{A}_{3,5,1}| = 56$.

We evaluate the Deep Operator BSDE approximation across 22 different families of options. We train the Deep Operator BSDE 8 times. For each family of terminal conditions, we first compute the mean error within that family, and then report in Table \ref{results for example 1} the average and standard deviation of these mean errors across the 8 runs.

\begin{table}[ht]
\centering
\begin{tabular}{|c|c|c|}
\hline
\textbf{Option} & \textbf{$\text{err}_Y \;(\times 10^{-3})$} & \textbf{$\text{err}_Z \;(\times 10^{-2})$} \\
\hline
Call  & 3.85 $\pm$ 1.05 & 1.27 $\pm$ 0.70 \\
Put  & 4.23 $\pm$ 1.98 & 1.08 $\pm$ 0.56 \\
Down-and-Out Call  & 3.41 $\pm$ 1.25 & 1.48 $\pm$ 0.61 \\
Up-and-Out Call  & 2.96 $\pm$ 1.20 & 1.11 $\pm$ 0.82 \\
Down-and-Out Put  & 3.41 $\pm$ 1.26 & 1.75 $\pm$ 1.05 \\
Up-and-Out Put  & 2.71 $\pm$ 1.79 & 2.69 $\pm$ 0.86 \\
Down-and-In Call  & 2.90 $\pm$ 0.65 & 4.65 $\pm$ 0.74 \\
Up-and-In Call  & 3.56 $\pm$ 1.10 & 1.65 $\pm$ 0.96 \\
Down-and-In Put  & 2.78 $\pm$ 1.77 & 1.66 $\pm$ 0.91 \\
Up-and-In Put  & 2.98 $\pm$ 1.29 & 2.31 $\pm$ 1.02 \\
Double-Knock-Out Call & 3.04 $\pm$ 1.31 & 1.35 $\pm$ 0.88 \\
Double-Knock-Out Put  & 3.36 $\pm$ 1.24 & 1.54 $\pm$ 0.98 \\
Double-Knock-In Call  & 3.56 $\pm$ 1.14 & 2.10 $\pm$ 1.00 \\
Double-Knock-In Put   & 2.85 $\pm$ 1.69 & 1.59 $\pm$ 0.88 \\
Power Asian Call Fixed Strike     & 3.68 $\pm$ 1.96 & 2.81 $\pm$ 0.80 \\
Power Asian Put Fixed Strike      & 3.36 $\pm$ 1.48 & 5.70 $\pm$ 0.82 \\
Power Asian Call Floating Strike  & 2.51 $\pm$ 0.79 & 3.93 $\pm$ 0.97 \\
Power Asian Put Floating Strike   & 3.25 $\pm$ 1.99 & 5.65 $\pm$ 1.01 \\
Lookback Fixed Strike Call        & 4.19 $\pm$ 1.49 & 2.88 $\pm$ 0.48 \\
Lookback Fixed Strike Put         & 4.55 $\pm$ 2.04 & 4.98 $\pm$ 0.92 \\
Power Lookback Floating Strike Call & 4.10 $\pm$ 1.17 & 16.90 $\pm$ 1.16 \\
Power Lookback Floating Strike Put  & 6.20 $\pm$ 3.64 & 18.59 $\pm$ 0.84 \\
\hline
\end{tabular}
\caption{Summary of the 22 option families and their errors (mean $\pm$ std). Values scaled by $10^{-3}$ and $10^{-2}$, respectively.}
\label{results for example 1}
\end{table}

We illustrate the results with a few representative plots. A more extensive collection is available in the Supplementary Material and in the accompanying GitHub repository.  

We start with the one–dimensional families corresponding to Put options and Power Asian Put options with floating strike, defined respectively by
\begin{gather*}
    \xi(K) = (K - S_T)_+, \quad \xi(p) = \big( S_T - (\tfrac{1}{10}\sum_{i=1}^{10} S_{t_i})^p \big)_+.
\end{gather*}
The corresponding plots are shown in Figures~\ref{Figure 1(a)}–\ref{Figure 1(b)}. For both one-dimensional families, the $\mathcal{Y}_0$ component is approximated with high accuracy. For $\mathcal{Z}_0$, the Put option also aligns closely with the baseline, while for the Power Asian Put (floating strike) the approximation is slightly less precise. This behavior is to be expected, as approximating the $\mathcal{Z}$ component is typically more challenging, both in classical numerical schemes and in the Euler scheme.

\begin{figure}[h]
    \centering
    \begin{subfigure}{7.0cm}
        \centering
        \includegraphics[width=7.0cm]{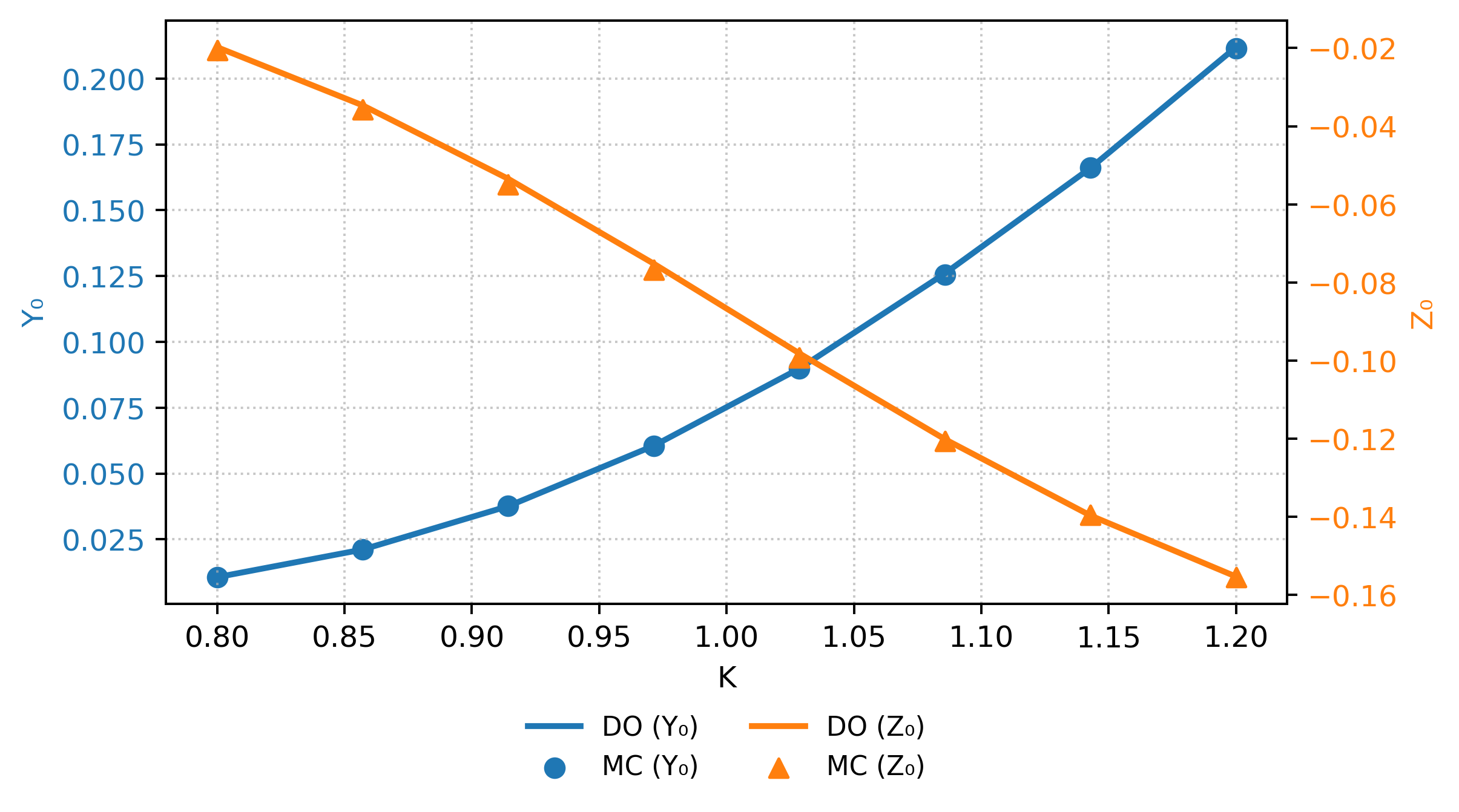}
        \caption{Deep Operator BSDE vs.\ Monte Carlo approximation of the solution operators $\mathcal{Y}_0$ and $\mathcal{Z}_0$ for Put options.}
        \label{Figure 1(a)}
    \end{subfigure}
    \hspace{0.06cm}
    \begin{subfigure}{7.0cm}
        \centering
        \includegraphics[width=7.0cm]{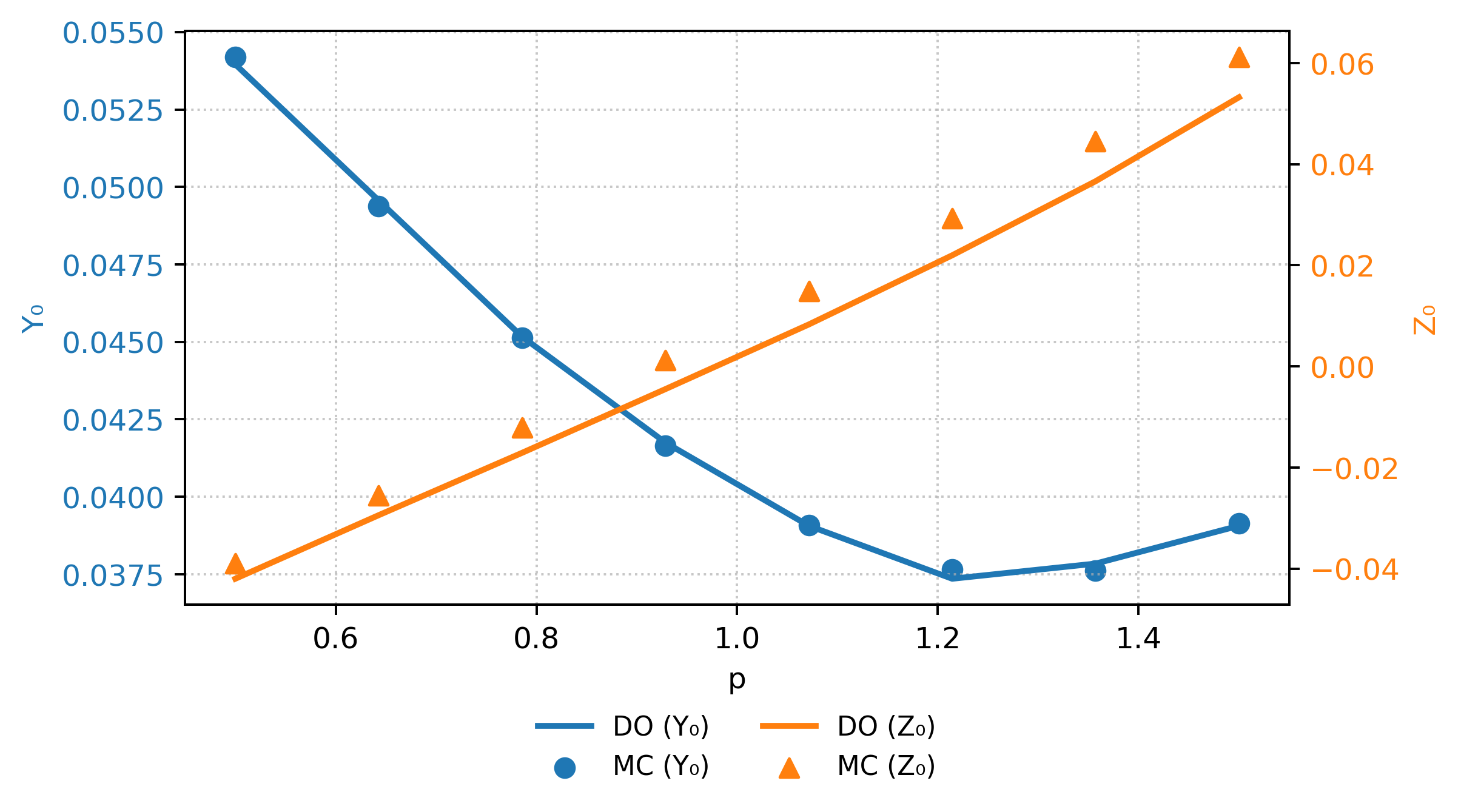}
        \caption{Deep Operator BSDE vs.\ Monte Carlo approximation of the solution operators $\mathcal{Y}_0$ and $\mathcal{Z}_0$ for Power Asian Put options with Floating Strike.}
        \label{Figure 1(b)}
    \end{subfigure}
    \caption{}
    \label{Figure 1}
\end{figure}

We also include some plots for the two-dimensional family of Down-and-Out Call options, defined as
\begin{gather*}
    \xi(K,L) \coloneqq (S_{T}-K)_{+} \mathbf{1}_{\{ S_{t_{i}} \geq L \ \forall  i \in \{0, \dots, 10\} \}}.
\end{gather*}

Figures~\ref{Figure 2(a)}–\ref{Figure 2(b)} display the surfaces of $\mathcal{Y}_{0}$ produced by the Deep Operator BSDE and the Monte Carlo baseline for different values of $K$ and $L$. The two surfaces are closely aligned. Figures~\ref{Figure 3(a)}–\ref{Figure 3(b)} report the corresponding results for the $Z$ component. Finally, Figures~\ref{Figure 4(a)}–\ref{Figure 4(b)} report the errors for $\mathcal{Y}_{0}$ and $\mathcal{Z}_{0}$. While the errors for $\mathcal{Y}_{0}$ remain small and relatively stable across the domain, the errors for $\mathcal{Z}_{0}$ increase with larger values of $K$.

\begin{figure}[h]
    \centering
    \begin{subfigure}{6.5cm}
        \centering
        \includegraphics[width=6.5cm]{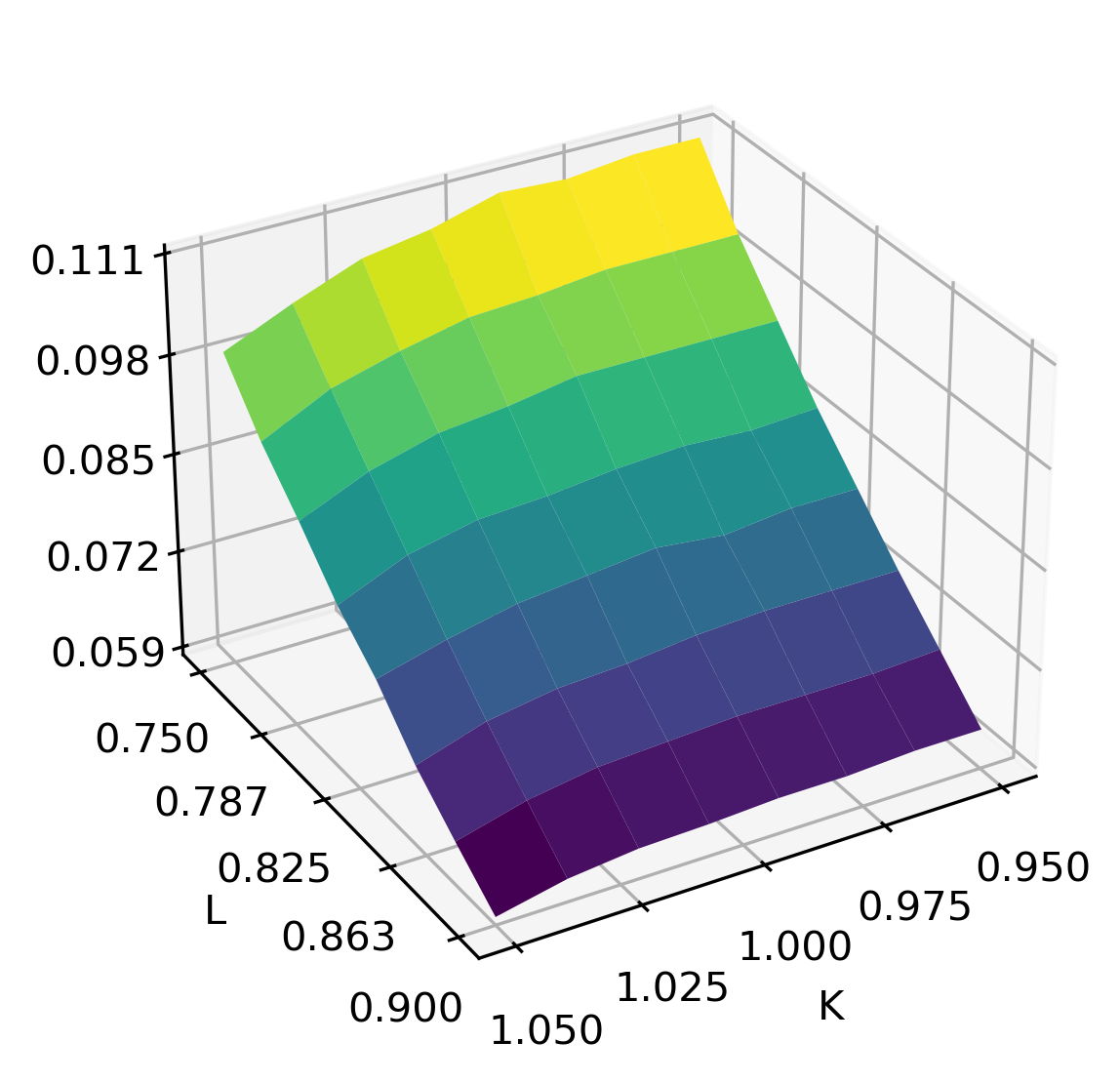}
        \caption{Deep Operator BSDE approximation of the solution operator $\mathcal{Y}_0$ for Down-and-Out Call options.}
        \label{Figure 2(a)}
    \end{subfigure}
    \hspace{0.06cm}
    \begin{subfigure}{6.5cm}
        \centering
        \includegraphics[width=6.5cm]{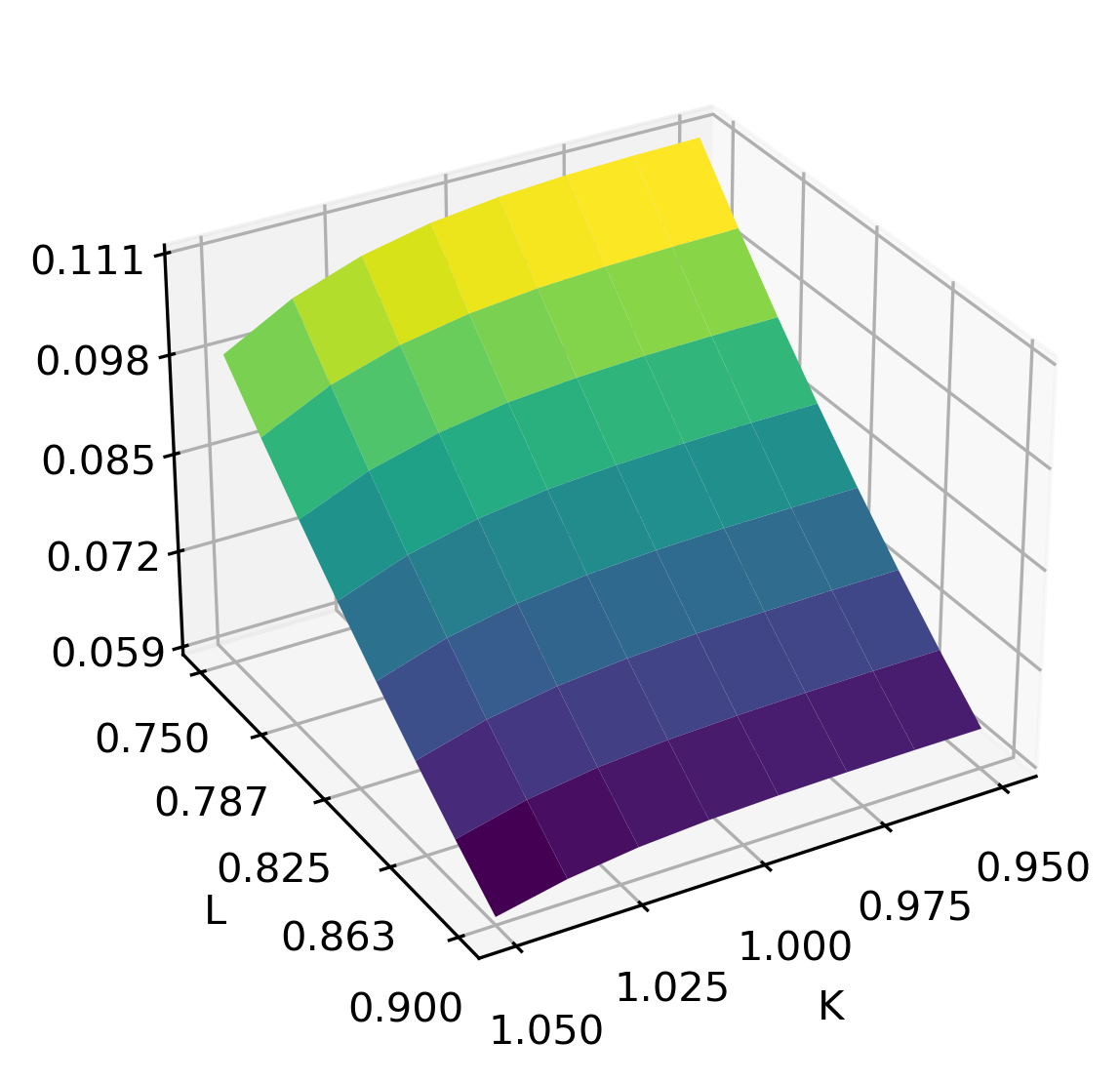}
        \caption{Monte Carlo approximation of the solution operator $\mathcal{Y}_0$ for Down-and-Out Call options.}
        \label{Figure 2(b)}
    \end{subfigure}
    \caption{}
    \label{Figure 2}
\end{figure}
\begin{figure}[h]
    \centering
    \begin{subfigure}{6.5cm}
        \centering
        \includegraphics[width=6.5cm]{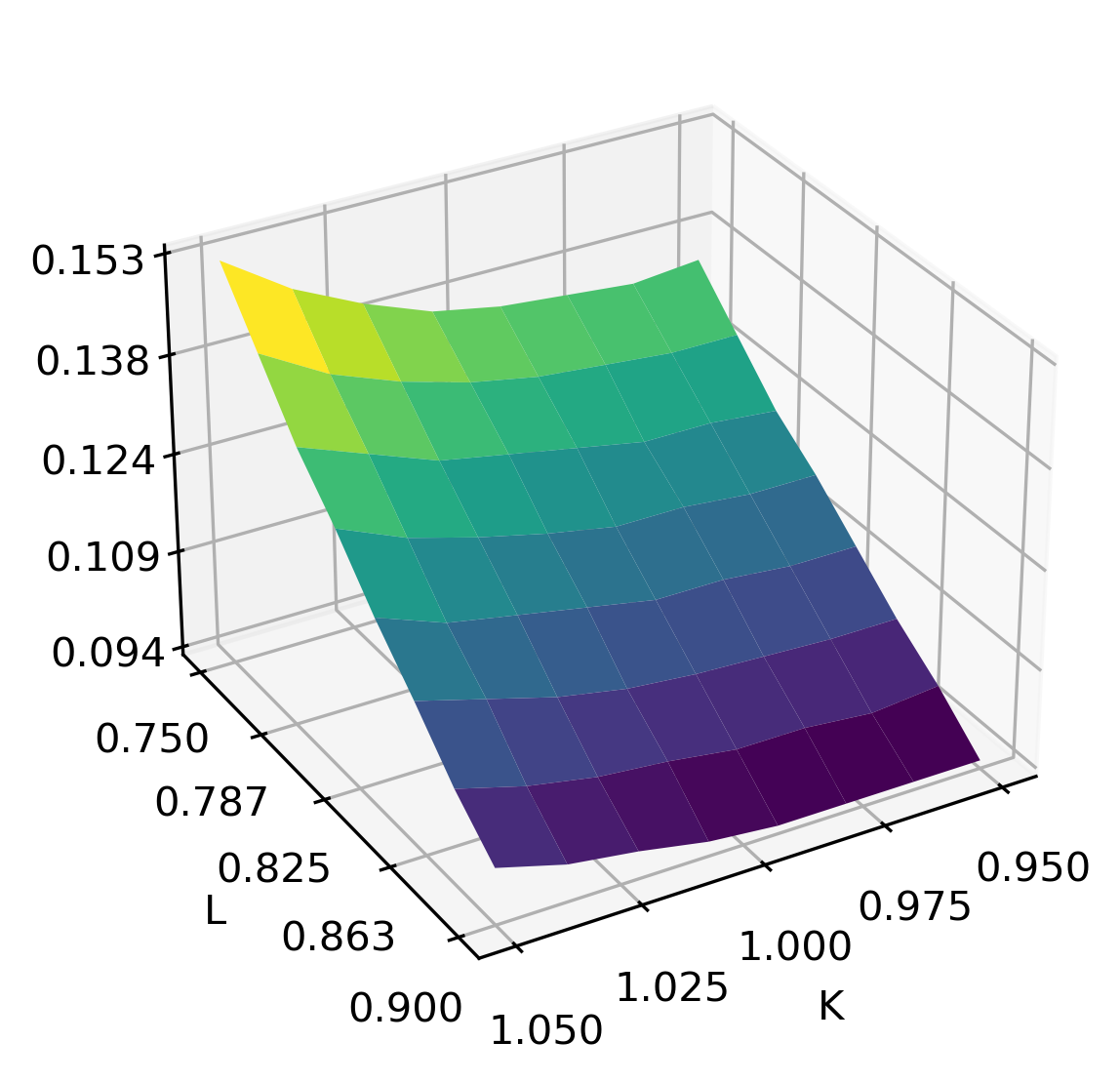}
        \caption{Deep Operator BSDE approximation of the solution operator $\mathcal{Z}_0$ for Down-and-Out Call options.}
        \label{Figure 3(a)}
    \end{subfigure}
    \hspace{0.06cm}
    \begin{subfigure}{6.5cm}
        \centering
        \includegraphics[width=6.5cm]{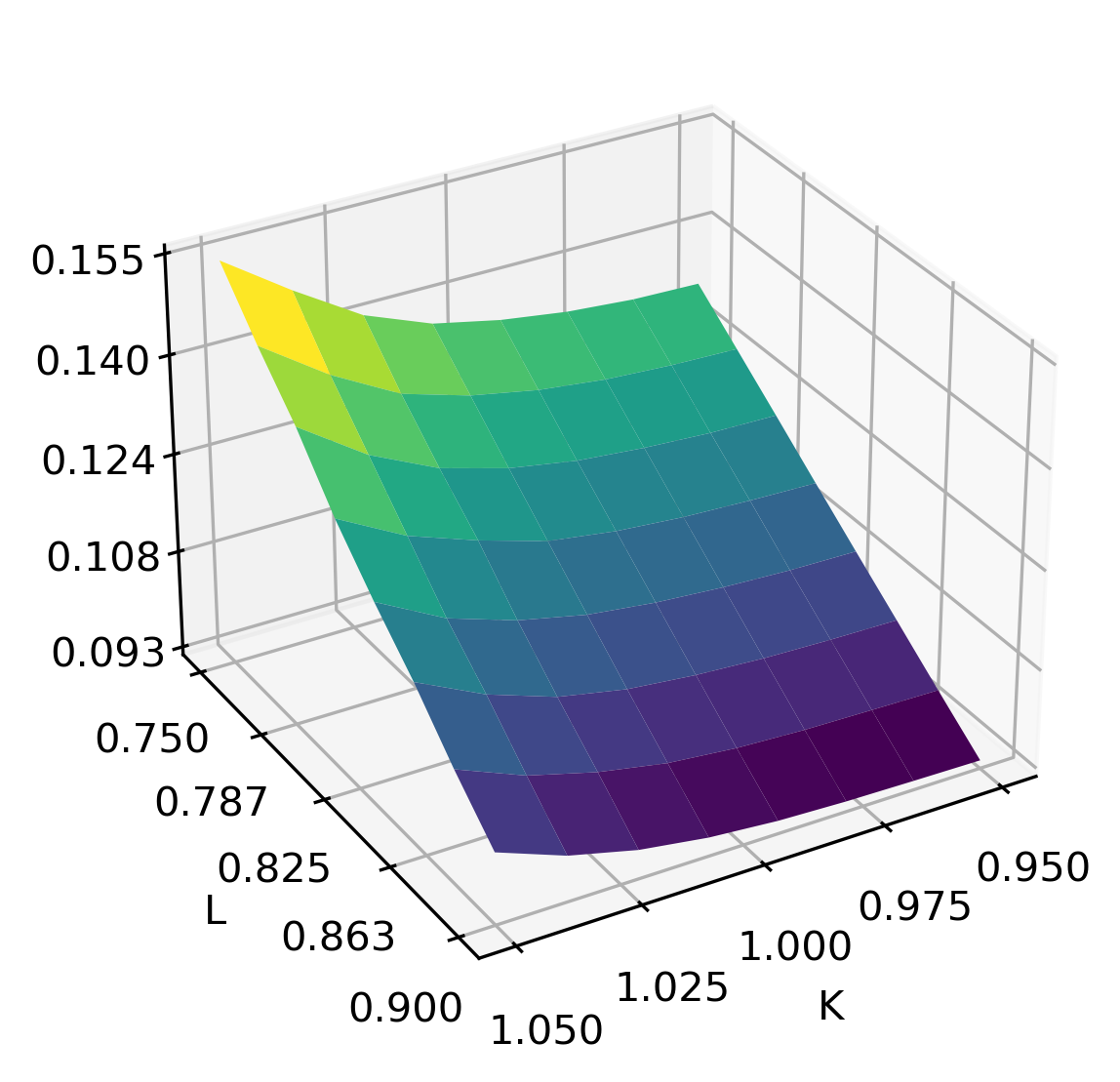}
        \caption{Monte Carlo approximation of the solution operator $\mathcal{Z}_0$ for Down-and-Out Call options.}
        \label{Figure 3(b)}
    \end{subfigure}
    \caption{}
    \label{Figure 3}
\end{figure}

\begin{figure}[h]
    \centering
    \begin{subfigure}{6.5cm}
        \centering
        \includegraphics[width=6.5cm]{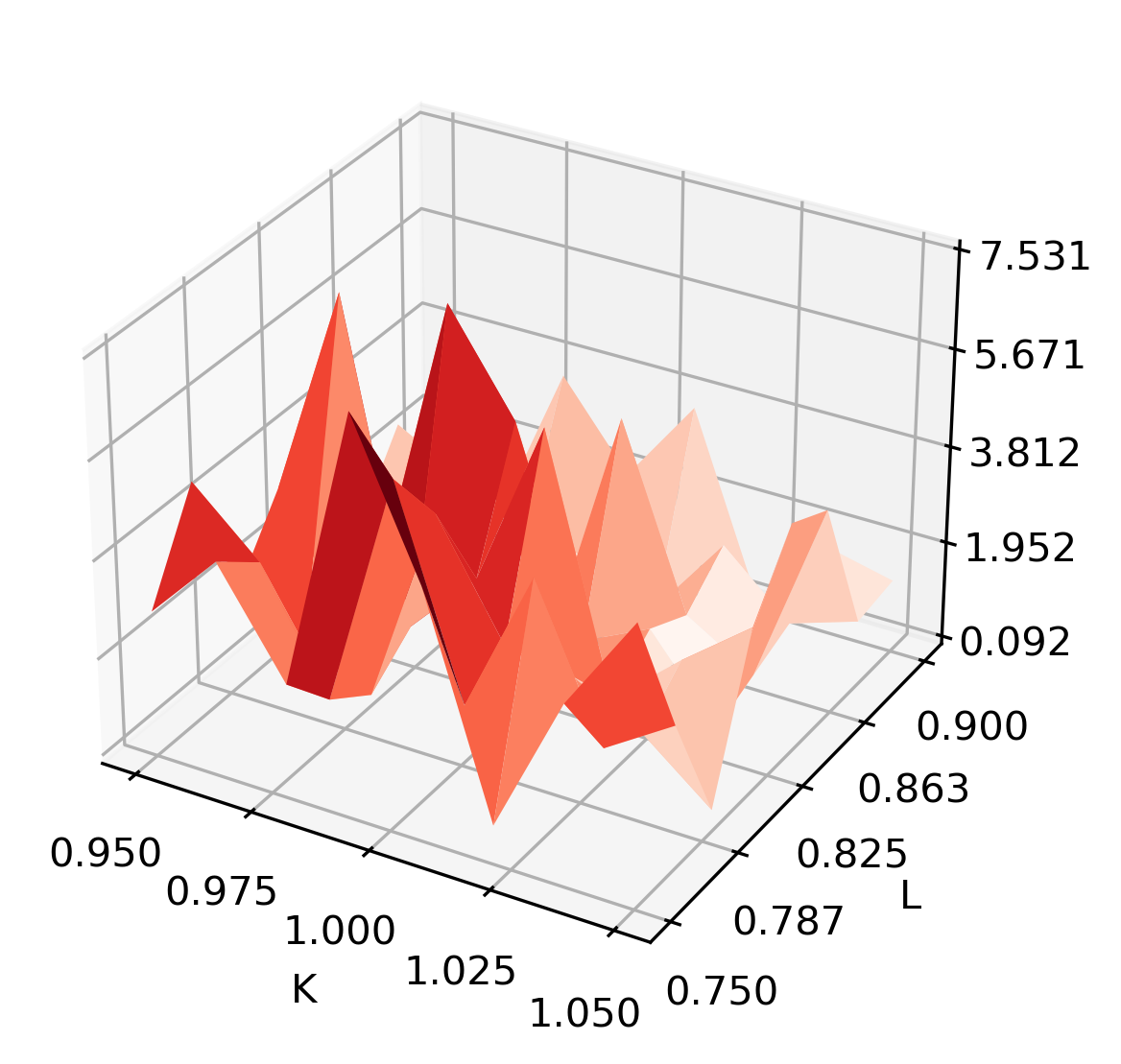}
        \caption{Error between the two approximations of the solution operator $\mathcal{Y}_0$ for Down-and-Out Call options (values shown $\times 10^{-4}$).}
        \label{Figure 4(a)}
    \end{subfigure}
    \hspace{0.06cm}
    \begin{subfigure}{6.5cm}
        \centering
        \includegraphics[width=6.5cm]{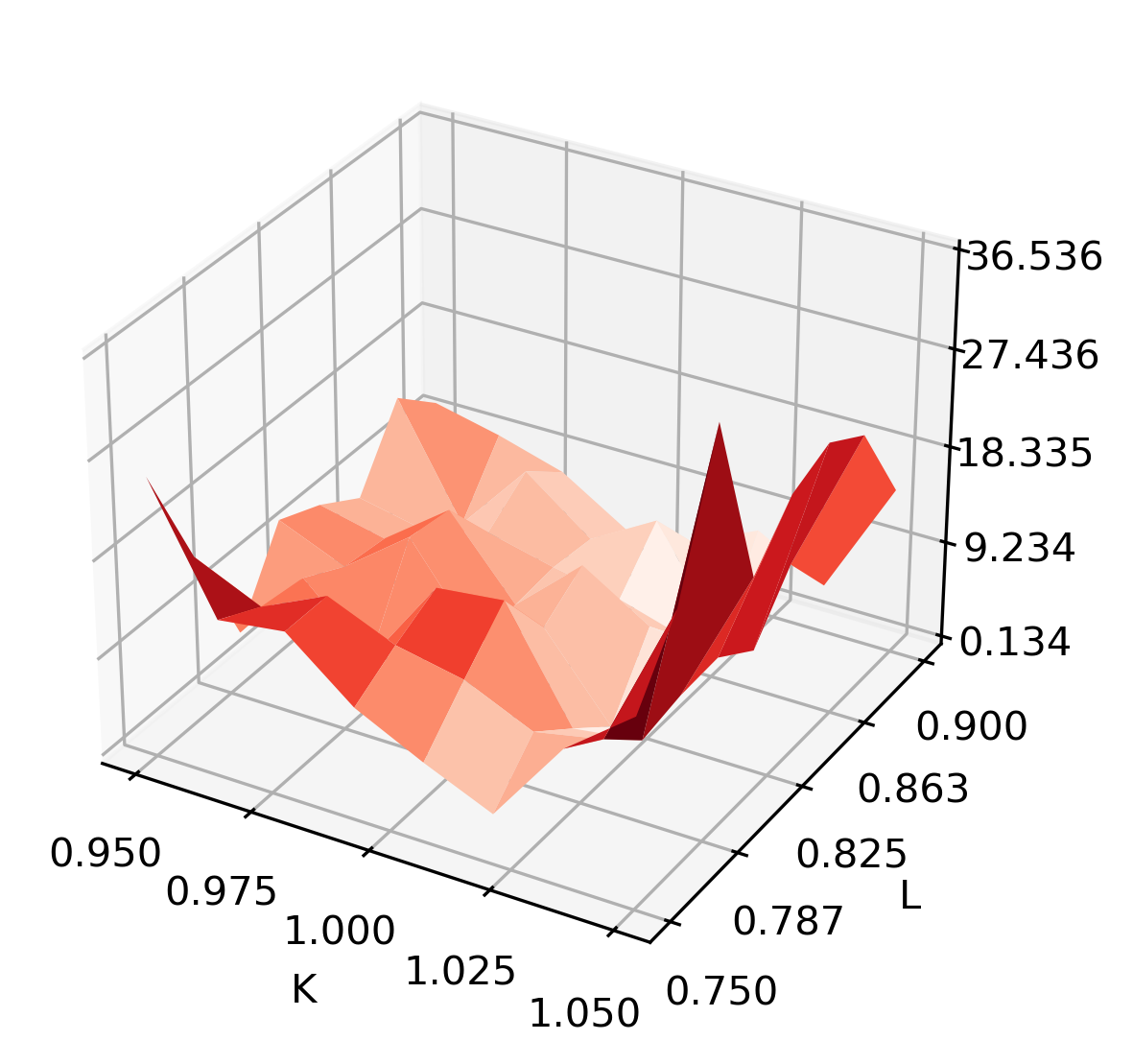}
        \caption{Error between the two approximations of the solution operator $\mathcal{Z}_0$ for Down-and-Out Call options (values shown $\times 10^{-4}$).}
        \label{Figure 4(b)}
    \end{subfigure}
    \caption{}
    \label{Figure 4}
\end{figure}

\subsection{Pricing and hedging -- nonlinear generator in dimension 2}

We next consider the case of pricing and hedging of options in a Black-Scholes setting of dimension 2. The dynamics are given by  
\begin{gather*}
    S_{t}^{j} = s_{0}^{j} e^{(\mu^{j}-(\sigma^{j})^{2}/2)t + \sigma^{j} B_{t}^{j}}, \quad \forall t \in [0,T], \quad j = \{1,2\},
\end{gather*}
$\mu^{j}$ and $\sigma^{j}$ represent the trend and the volatility of the $j$th asset, and $B=(B^{1}, B^{2})$ is a $2$-dimensional Brownian motion such that $\scalar{B^{1}}{B^{2}}_{t} = \rho t$. If one assumes that the borrowing rate $R$ is higher than the lending rate $r$, pricing and hedging an option $\xi$ is equivalent to solving a BSDE with terminal condition $\xi$ and nonlinear generator $g$ defined by 
\begin{gather*}
    g(t,y,z) = -ry - \theta \cdot z + (R-r)\Big(y- (\Sigma^{-1} z)_{1} - (\Sigma^{-1} z)_{2} \Big)_{-},
\end{gather*}
where $\theta \coloneqq \Sigma^{-1} (\mu - r\mathbf{1})$ ($\mathbf{1}$ being the vector with every component equal to one), and $\Sigma$ is the matrix defined by $\Sigma_{ij}=\sigma^{i}L_{ij}$, where $L$ denotes the lower triangular matrix in the Cholesky decomposition $C = L L^{T}$, $C = (\rho \mathbf{1}_{i \neq j} + \mathbf{1}_{i=j})_{1 \leq i, j \leq 2}$. We refer to \cite[Example 1.1]{ElKarouiN.1997BSDE} for the details.

The selected parameters are $r=0.02$, $R=0.1$, $T=1$, $\rho = 0.1$, $s_{0}^{j}=1$, $\mu^{j}=0.02$ and $\sigma^{j}=0.2$ for $j=1,2$. We evaluate the Deep Operator BSDE approximation across 28 different families of options. Since no closed-form expression for the solution of these equations exists, we use the algorithm proposed in \cite{BriandPhilippe2014SOBB} with $P=3$ and $M = 10$ as a baseline. 

For the truncation of the chaos expansion, we again set $p=3$ and use the partition $\overline{\pi} = \{ T \times \tfrac{i}{5} : 0 \leq i \leq 5 \}$. Since $d=2$, this means that the number of chaos coefficients that are considered is $|\mathcal{A}_{3,5,2}| = 286$.

We train the Deep Operator BSDE 8 times. For each family of terminal conditions, we first compute the mean error within that family, and then report in Table \ref{results for example 2} the average and standard deviation of these mean errors across the 8 runs.

\begin{table}[ht]
\centering
\begin{tabular}{|c|c|c|}
\hline
\textbf{Option} & \textbf{$\text{err}_Y \;(\times 10^{-3})$} & \textbf{$\text{err}_Z \;(\times 10^{-2})$} \\
\hline
Call -- Single                & 2.21 $\pm$ 0.65 & 0.41 $\pm$ 0.06 \\
Call -- Basket Weighted       & 2.58 $\pm$ 0.74 & 0.41 $\pm$ 0.07 \\
Call -- Spread                & 7.54 $\pm$ 1.08 & 0.25 $\pm$ 0.04 \\
Call -- Max                   & 3.66 $\pm$ 0.42 & 0.58 $\pm$ 0.07 \\
Call -- Min                   & 1.75 $\pm$ 0.98 & 0.32 $\pm$ 0.03 \\
Call -- Geometric             & 2.56 $\pm$ 0.73 & 0.41 $\pm$ 0.06 \\
Call -- Ratio                 & 6.28 $\pm$ 1.02 & 0.25 $\pm$ 0.04 \\
Put -- Single                 & 3.60 $\pm$ 0.16 & 0.24 $\pm$ 0.09 \\
Put -- Basket Weighted        & 3.19 $\pm$ 0.28 & 0.26 $\pm$ 0.09 \\
Put -- Spread                 & 7.53 $\pm$ 1.09 & 0.27 $\pm$ 0.04 \\
Put -- Max                    & 1.43 $\pm$ 0.38 & 0.18 $\pm$ 0.04 \\
Put -- Min                    & 1.38 $\pm$ 0.49 & 0.38 $\pm$ 0.10 \\
Put -- Geometric              & 3.08 $\pm$ 0.27 & 0.26 $\pm$ 0.09 \\
Put -- Ratio                  & 6.52 $\pm$ 1.02 & 0.35 $\pm$ 0.06 \\
Asian Call -- Single          & 11.00 $\pm$ 0.76 & 1.33 $\pm$ 0.07 \\
Asian Put -- Single           & 7.07 $\pm$ 1.39 & 0.20 $\pm$ 0.03 \\
Asian Call -- Basket Weighted & 10.87 $\pm$ 0.80 & 1.21 $\pm$ 0.10 \\
Asian Put -- Basket Weighted  & 8.01 $\pm$ 1.35 & 0.18 $\pm$ 0.04 \\
Asian Call -- Spread          & 2.13 $\pm$ 0.66 & 1.01 $\pm$ 0.08 \\
Asian Put -- Spread           & 2.17 $\pm$ 0.64 & 1.00 $\pm$ 0.08 \\
Asian Call -- Max             & 13.06 $\pm$ 0.83 & 1.58 $\pm$ 0.08 \\
Asian Put -- Max              & 9.53 $\pm$ 1.16 & 0.30 $\pm$ 0.06 \\
Asian Call -- Min             & 3.22 $\pm$ 1.11 & 0.98 $\pm$ 0.09 \\
Asian Put -- Min              & 3.64 $\pm$ 1.14 & 0.28 $\pm$ 0.13 \\
Asian Call -- Geometric       & 11.03 $\pm$ 0.82 & 1.17 $\pm$ 0.10 \\
Asian Put -- Geometric        & 7.96 $\pm$ 1.32 & 0.17 $\pm$ 0.02 \\
Asian Call -- Ratio           & 2.00 $\pm$ 0.52 & 1.14 $\pm$ 0.07 \\
Asian Put -- Ratio            & 4.17 $\pm$ 1.02 & 0.57 $\pm$ 0.08 \\
\hline
\end{tabular}
\caption{Summary of the 28 option families and their errors (mean $\pm$ std). Values scaled by $10^{-3}$ and $10^{-2}$, respectively.}
\label{results for example 2}
\end{table}

We illustrate the results with two representative plots. A more extensive collection is available in the Supplementary Material and in the accompanying GitHub repository.  

We focus on the one–dimensional families corresponding to Put -- Max options and Asian Call -- Max options, defined respectively by
\begin{gather*}
    \xi(K) = (K - \max(S_T^1, S_T^2))_{+}, \quad \xi(K) = \Big(\tfrac{1}{10}\sum_{i=1}^{10} \max(S_{t_i}^1, S_{t_i}^2) - K \Big)_{+}.
\end{gather*}

The plots are shown in Figures~\ref{Figure 5(a)}–\ref{Figure 5(b)}. For Put–Max options, the Deep Operator BSDE provides results that are very close to the baseline for both $\mathcal{Y}_0$ and the first component of $\mathcal{Z}_0$. In contrast, for Asian Call–Max options, the agreement is less accurate, even though the shapes and orders of magnitude remain consistent.

\begin{figure}[h]
    \centering
    \begin{subfigure}{7.0cm}
        \centering
        \includegraphics[width=7.0cm]{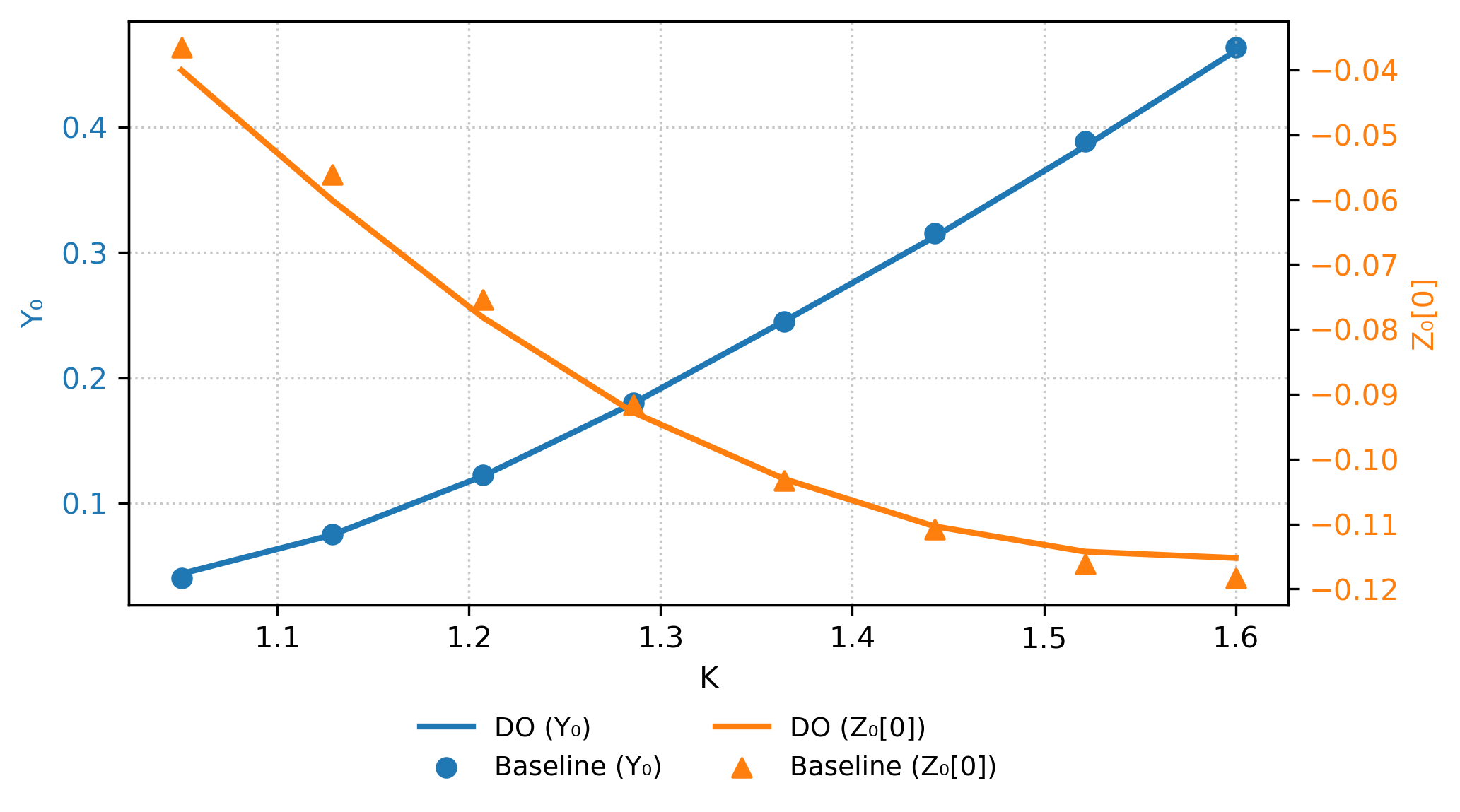}
        \caption{Deep Operator BSDE vs.\ Baseline approximation of the solution operators $\mathcal{Y}_0$ and the first component of $\mathcal{Z}_0$ for Put -- Max options.}
        \label{Figure 5(a)}
    \end{subfigure}
    \hspace{0.06cm}
    \begin{subfigure}{7.0cm}
        \centering
        \includegraphics[width=7.0cm]{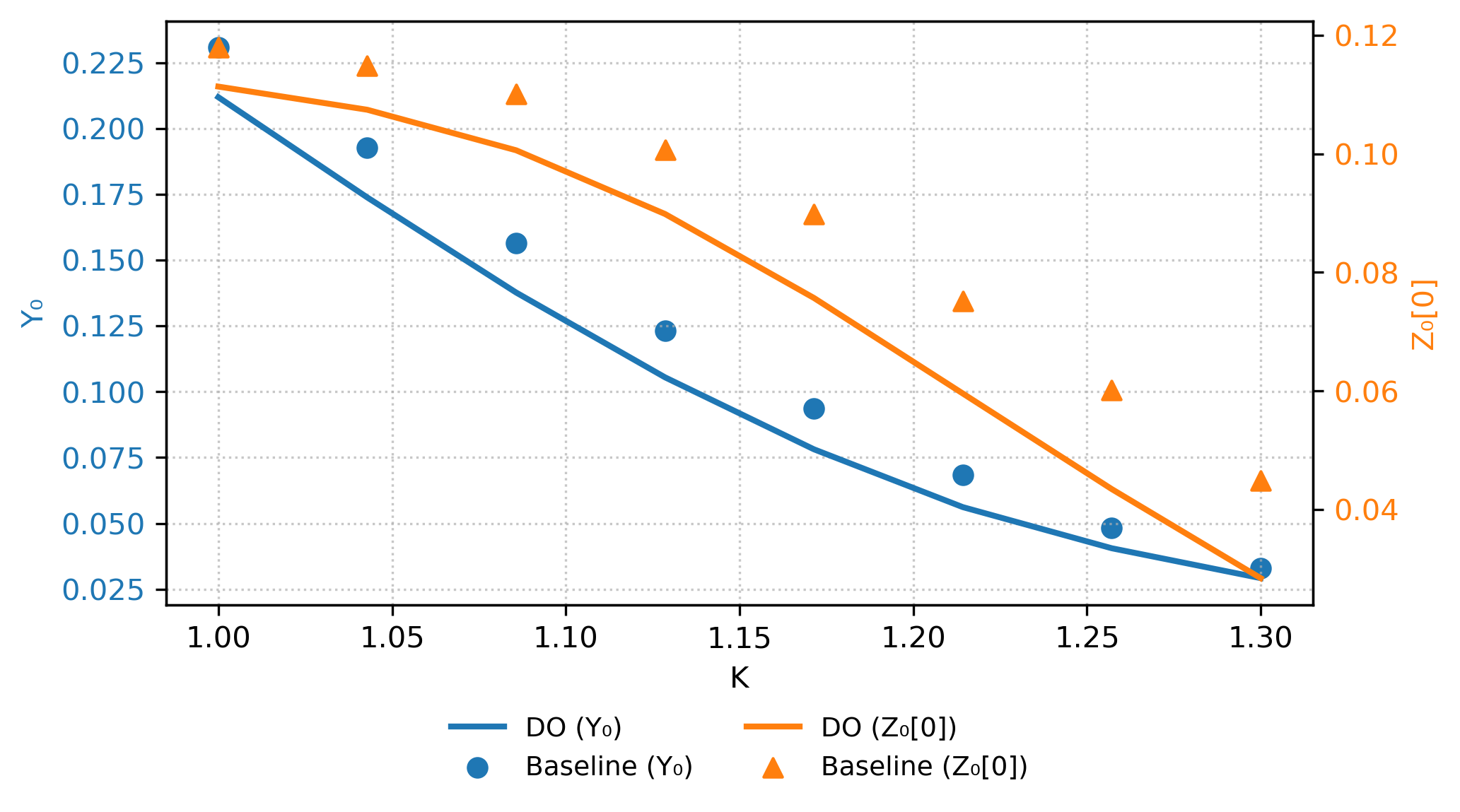}
        \caption{Deep Operator BSDE vs.\ Baseline approximation of the solution operators $\mathcal{Y}_0$ and the first component of $\mathcal{Z}_0$ for Asian Call -- Max options.}
        \label{Figure 5(b)}
    \end{subfigure}
    \caption{}
    \label{Figure 5}
\end{figure}

\section{Conclusion}

We have introduced a numerical method for approximating the solution operators of BSDEs on a discrete time grid, combining the classical Euler scheme with the Wiener chaos decomposition. The method is shown to converge under minimal assumptions, and we have detailed how it can be implemented in practice. Numerical experiments confirm that the proposed method achieves competitive performance compared to more traditional schemes, which in essence only provide evaluations of the operator at individual terminal conditions. For instance, our operator-based perspective allows us to approximate the solution across families of structurally different terminal conditions within a unified framework, a task that cannot be addressed by existing parametric BSDE methods.

%
%

\begin{acks}[Acknowledgments]
The authors thank the anonymous referees for their valuable comments, which have contributed to improving the numerical examples presented in this paper.

Giulia Di Nunno is also affiliated with the Department of Business and Management Science, NHH Norwegian School of Economics, Bergen, Norway.

Pere Diaz Lozano is the corresponding author. 
\end{acks}
\begin{funding}
The present research is carried out within the frame and support of the ToppForsk project nr. 274410 of the Research Council of Norway with the title STORM: Stochastics for Time-Space Risk Models.
\end{funding}


\appendix

\section{Architecture of the neural networks and training hyperparameters}\label{sec:architecture and numerics}

We choose to model $(\mathscr{Y}_{i}^{\pi}$, $\mathscr{Z}_{i}^{\pi})$ with the Multilevel Neural Networks presented in \cite[Section A.2]{NEURIPS2020_c1714160}. 

Recall that $(\mathscr{Y}_{i}^{\pi}$, $\mathscr{Z}_{i}^{\pi})$ take as input $(X_{t_i}^{\leq p, M}, (d_a)_{a \in \mathcal{A}_{p, M,d}})$. Notice that, as $i \in \{0, \dots, m\}$ gets smaller, there are more coordinates in $X_{t_i}^{\leq p, M}$ that become a.s. constant, and hence they do not need to be considered as input when modeling $(\mathscr{Y}_{i}^{\pi}, \mathscr{Z}_{i}^{\pi})$. At $i=0$, all coordinates in $X_{t_0}^{\leq p, M}$ become constant a.s. In practice, this means that the dimensionality of $\mathscr{Y}_{i}^{\pi}$, $\mathscr{Z}_{i}^{\pi}$ shrinks with smaller $i$, reducing the complexity of the optimization problems that we need to solve.

We now proceed to detail the architectures and hyperparameters in each example.

\subsection{Example 1 -- Linear BSDE}

We recall that the number of chaos coefficients in this case is $|\mathcal{A}_{3,5,1}| = 56$.

\paragraph*{Network architecture} The Multilevel network for both $\mathscr{Y}_{i}^{\pi}$, $\mathscr{Z}_{i}^{\pi}$ has the same structure. They take the whole concatenation of $(X_{t_i}^{\leq p, M}, (d_a)_{a \in \mathcal{A}_{p, M,d}})$ as input in the first layer. Following the multilevel network notation, for each step $i$ we use 5 as an amplifying factor, 3 levels, and Tanh as the activation function. The largest networks, i.e. $\mathscr{Y}_{9}^{\pi}$ and $\mathscr{Z}_{9}^{\pi}$, both have $1,418,583$ parameters. The smallest networks, i.e. $\mathscr{Y}_{0}^{\pi}$ and $\mathscr{Z}_{0}^{\pi}$, both have $361,483$ parameters.

\paragraph*{Training hyperparameters}

To optimize the networks we use the Adam optimizer with decoupled weight decay regularization (AdamW), with the weight decay set to $10^{-6}$. Regarding the training hyperparameters, we set a batch size of $5 \times 10^{4}$ samples, a learning rate of $5 \times 10^{-4}$, and train for $10^{5}$ gradient steps. Evaluation is performed every $1000$ steps on a test batch of size $10^{5}$.

Learning rates are scheduled with a multiplicative decay factor $\gamma = 0.90$ applied every $500$ steps, with a minimum learning rate of $10^{-6}$. Early stopping is used with a patience of 5 evaluations. Finally, training is performed in mixed precision with bfloat16 arithmetic to accelerate computations.

\subsection{Example 2 -- Nonlinear BSDE}

We recall that the number of chaos coefficients in this case is $|\mathcal{A}_{3,5,2}| = 286$.

The Multilevel networks for both $\mathscr{Y}_{i}^{\pi}$ and $\mathscr{Z}_{i}^{\pi}$ share the same overall structure. Due to the high-dimensionality of the problem, unlike in Example 1, the concatenation $(X_{t_i}^{\leq p, M}, (d_a)_{a \in \mathcal{A}_{p, M,d}})$ is not fed directly into the first layer. Instead, we adopt a late-fusion scheme.

\paragraph*{Network architecture} At each time step $t_i$, the inputs are split into two blocks:
\[
(d, x) \in \mathbb{R}^{\texttt{chaos\_dim}} \times \mathbb{R}^{\texttt{x\_dim}},
\]
where $d$ collects the chaos coefficients and $x$ the state variables $X_{t_i}^{\leq p,M}$.  
Each block is standardized separately with known mean and std:
\[
\tilde d = \frac{d - \mu_d}{\sigma_d}, 
\qquad
\tilde x = \frac{x - \mu_x}{\sigma_x}.
\]

\emph{Chaos trunk.}  
The chaos block $\tilde d$ is processed through a multilevel stack of $L$ levels:
\[
z = \text{Trunk}(\tilde d) \in \mathbb{R}^{m},
\]
where $m = \lfloor \texttt{factor} \cdot \texttt{chaos\_dim}\rfloor$ is the hidden width, and $L=\texttt{levels}$.  
In our configuration: $\texttt{factor}=2.0$, $\texttt{levels}=2$, and the activation is ReLU.

\emph{FiLM fusion.}  
The state block $\tilde x$ modulates the chaos features $z$ via Feature-wise Linear Modulation (FiLM).  
Two linear maps are learned,
\[
\gamma = \Gamma \tilde x, 
\qquad 
\beta = B \tilde x, 
\quad \gamma,\beta \in \mathbb{R}^{m},
\]
and the fused representation is given by
\[
u = \gamma \odot z + \beta,
\]
where $\odot$ denotes elementwise multiplication.  

\emph{Nonlinear head.}  
The fused features $u$ are passed to a head MLP with $H=\texttt{head\_levels}$ hidden layers, each of width
\[
w = \max\!\Big(1,\;\lfloor \texttt{head\_factor} \cdot m \rfloor\Big).
\]
In our configuration: $\texttt{head\_levels}=2$, $\texttt{head\_factor}=1.5$, activation ReLU.  
This produces the nonlinear output $y_{\mathrm{mlp}} \in \mathbb{R}^{\texttt{out\_size}}$.

\emph{Affine skip.}  
In parallel, an affine map depending only on $x$ is added:
\[
y_{\mathrm{aff}} = A \tilde x + a.
\]
The final network output is the sum of the two pathways:
\[
y = y_{\mathrm{mlp}} + y_{\mathrm{aff}}.
\]

This design ensures that the high-dimensional chaos coefficients $d$ are handled by the multilevel trunk, while the state variables $x$ act as modulators through FiLM, yielding an almost linear but flexible structure suited for large input dimensions.

The largest networks, i.e. $\mathscr{Y}_{9}^{\pi}$ and $\mathscr{Z}_{9}^{\pi}$, have $2,866,865$ and $2,868,010$ parameters, respectively. The smallest networks, i.e. $\mathscr{Y}_{0}^{\pi}$ and $\mathscr{Z}_{0}^{\pi}$, have  $2,539,395$ and $2,540,254$ parameters, respectively.

\paragraph*{Training hyperparameters}

To optimize the networks we use the Adam optimizer with decoupled weight decay regularization (AdamW), with the weight decay set to $10^{-6}$. As for the training hyperparameters, we set a batch size of $10^{5}$ samples, a learning rate of $5 \times 10^{-4}$, and train for $10^{5}$ gradient steps. Evaluation is performed every $1000$ steps on a test batch of size $10^{5}$.

Learning rates are scheduled with a multiplicative decay factor $\gamma = 0.90$ applied every $500$ steps, with a minimum learning rate of $10^{-6}$. Early stopping is used with a patience of 5 evaluations. Training is performed in mixed precision with bfloat16 arithmetic to accelerate computations.

\section{Additional plots}\label{sec: additional plots}

We now present plots for several families of terminal conditions. Additional plots for the rest of the families are available in the GitHub repository \url{https://github.com/pere98diaz/Deep-Operator-BSDE}.

\begin{figure}[H]
    \centering
    \begin{subfigure}{7.0cm}
        \centering
        \includegraphics[width=7.0cm]{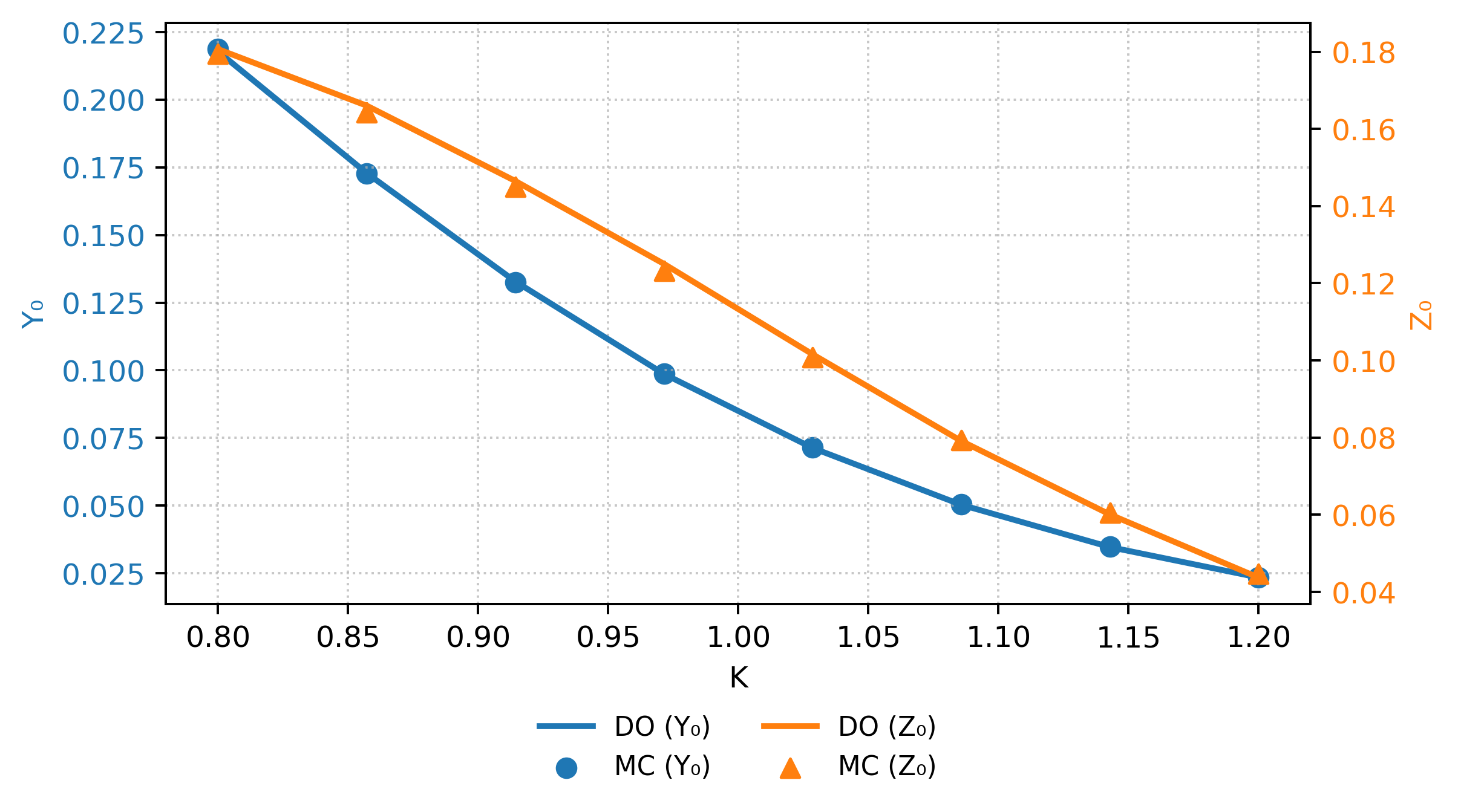}
        \caption{Example 1 -- Call.}
    \end{subfigure}
    \hspace{0.06cm}
    \begin{subfigure}{7.0cm}
        \centering
        \includegraphics[width=7.0cm]{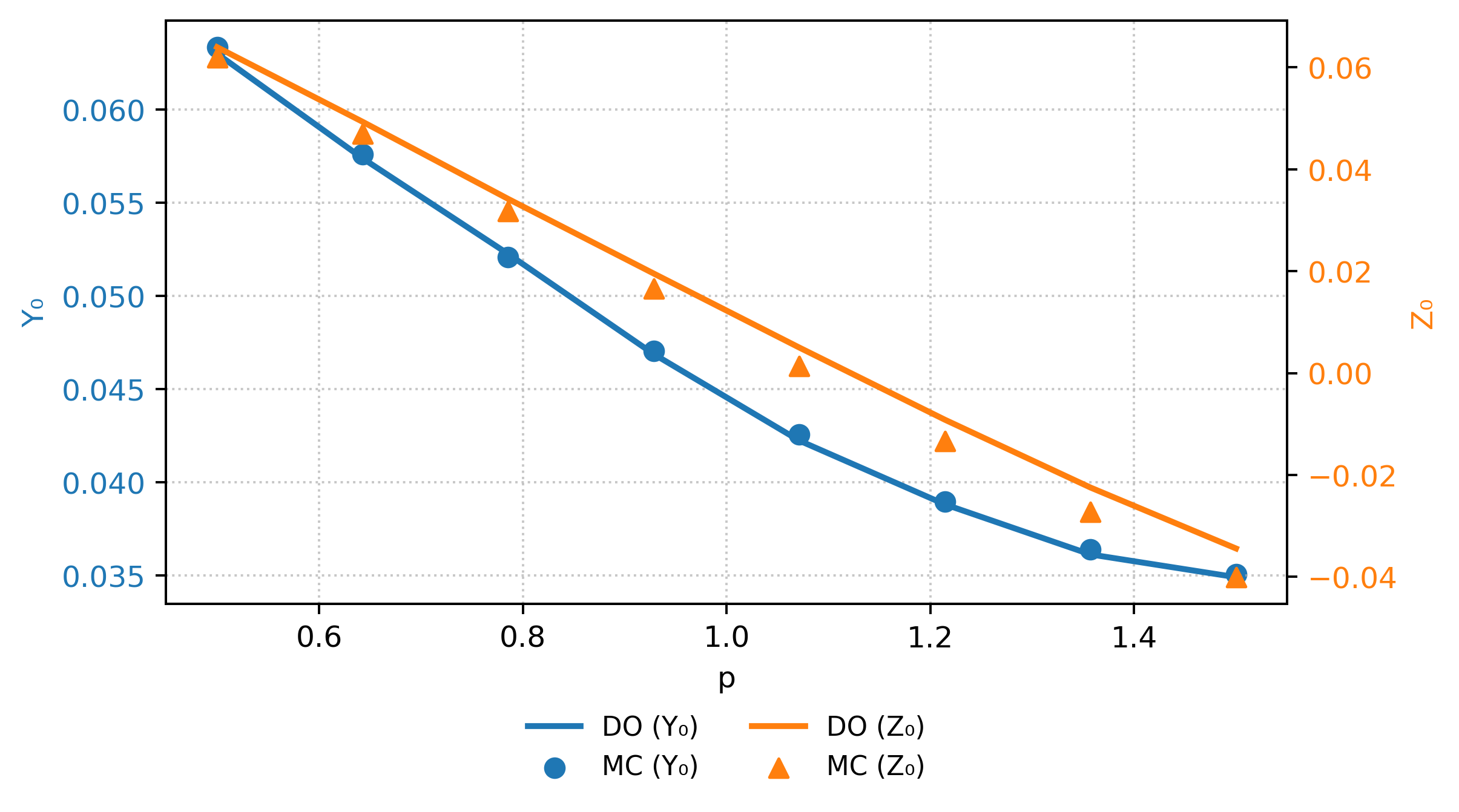}
        \caption{Example 1 -- Power Asian Call with Floating Strike.}
    \end{subfigure}
    \caption{}
\end{figure}

\begin{figure}[H]
    \centering
    \begin{subfigure}{7.0cm}
        \centering
        \includegraphics[width=7.0cm]{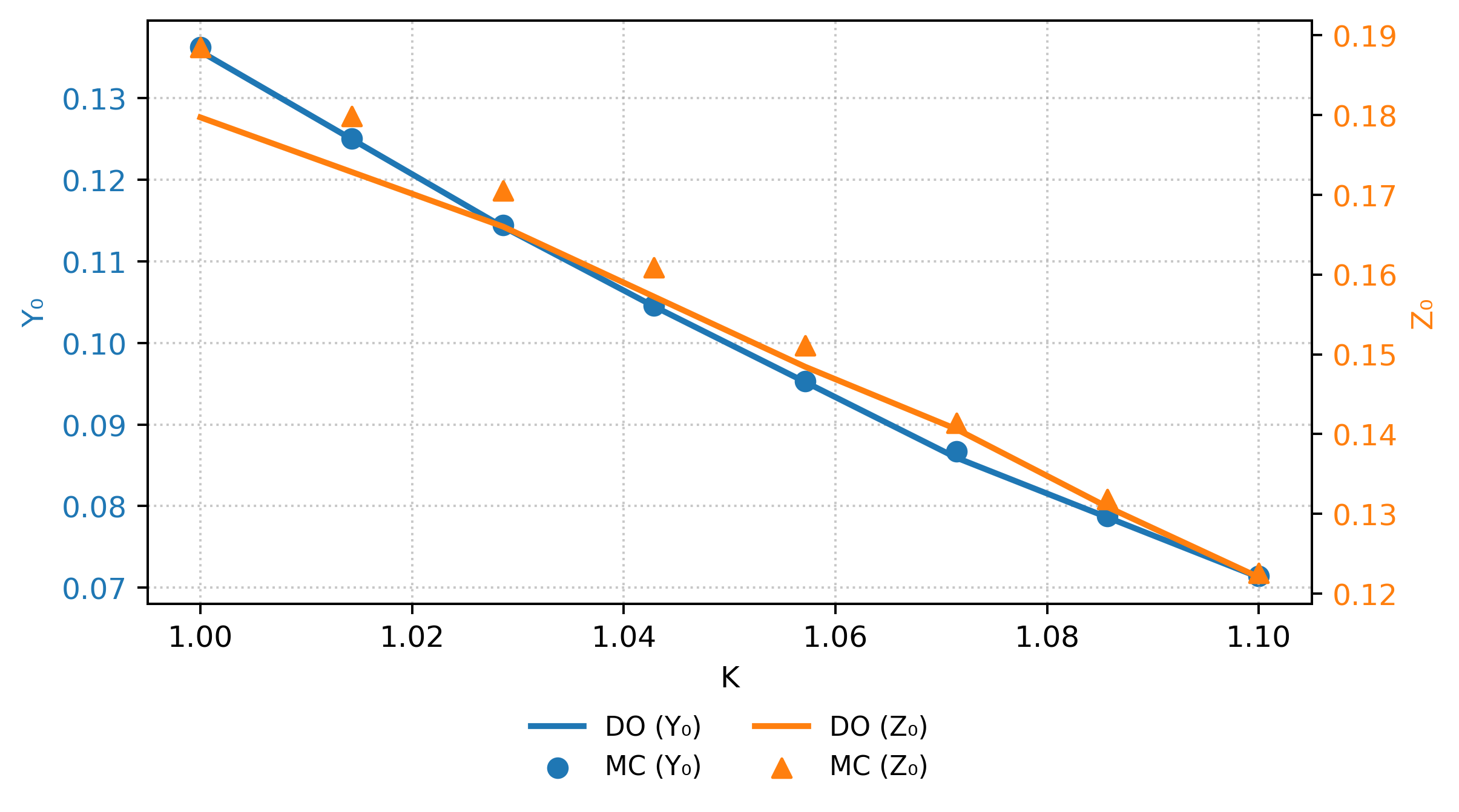}
        \caption{Example 1 -- Lookback Fixed Strike Call.}
    \end{subfigure}
    \hspace{0.06cm}
    \begin{subfigure}{7.0cm}
        \centering
        \includegraphics[width=7.0cm]{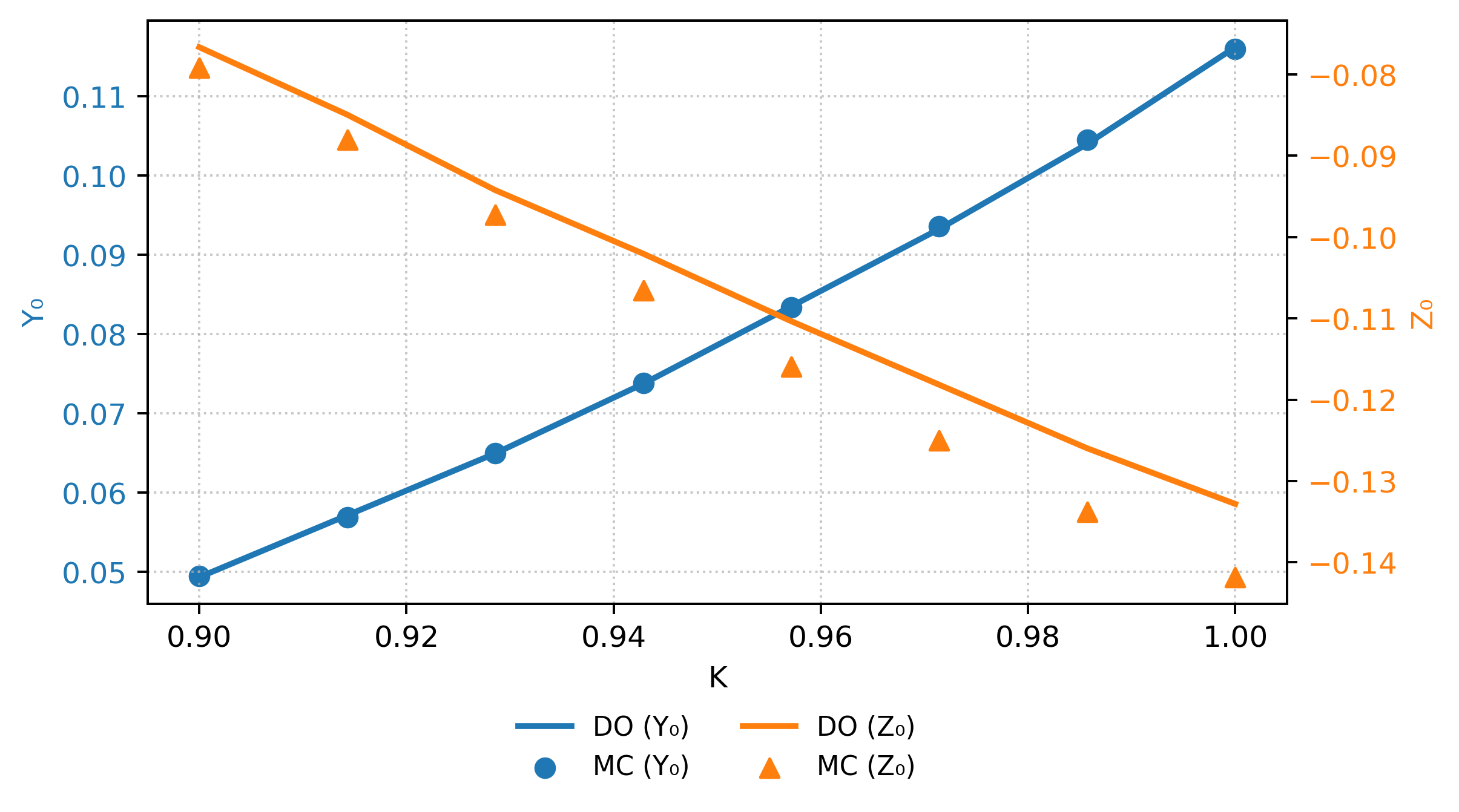}
        \caption{Example 1 -- Lookback Fixed Strike Put.}
    \end{subfigure}
    \caption{}
\end{figure}

\begin{figure}[H]
    \centering
    \begin{subfigure}{7.0cm}
        \centering
        \includegraphics[width=7.0cm]{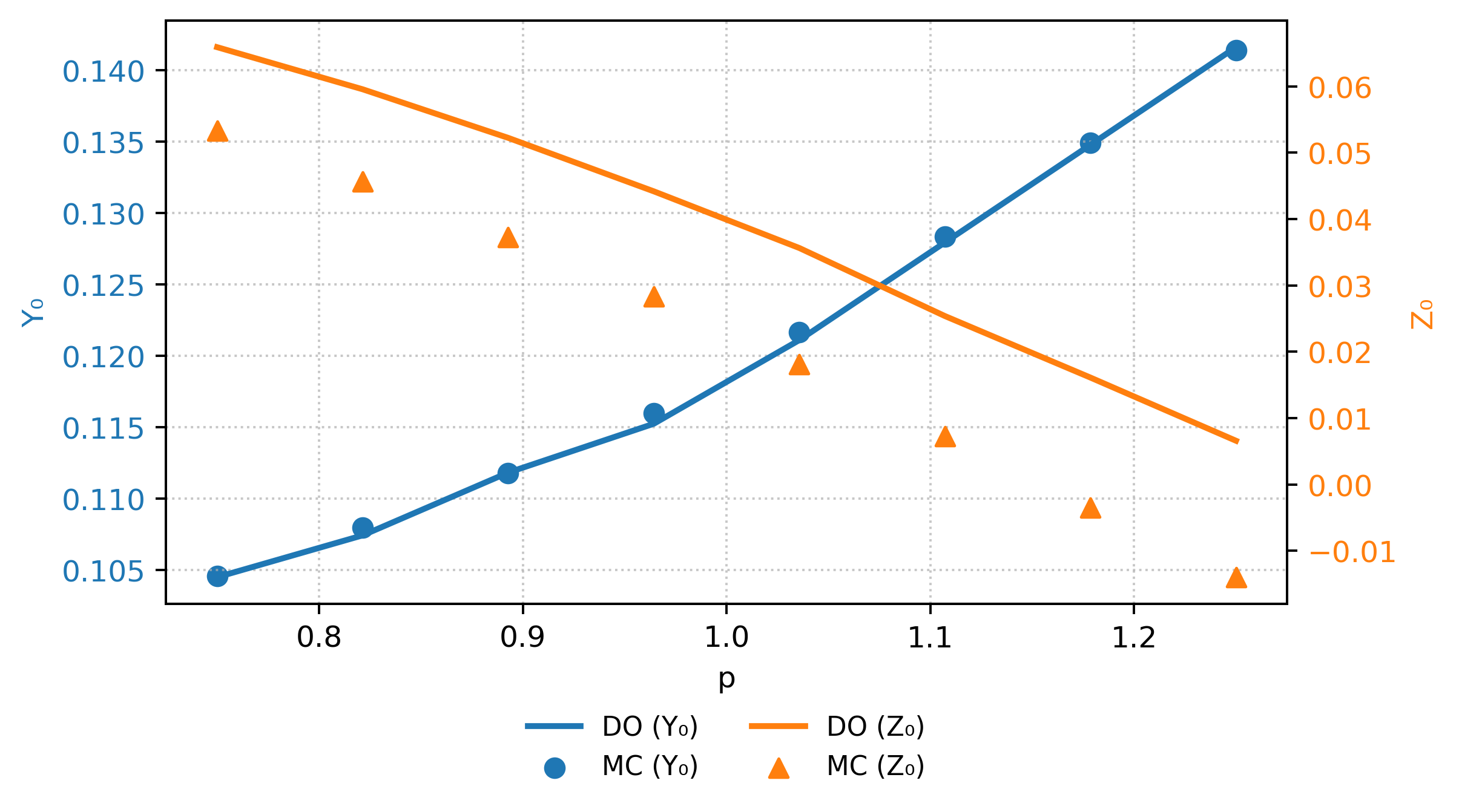}
        \caption{Example 1 -- Power Lookback Floating Strike Call.}
    \end{subfigure}
    \hspace{0.06cm}
    \begin{subfigure}{7.0cm}
        \centering
        \includegraphics[width=7.0cm]{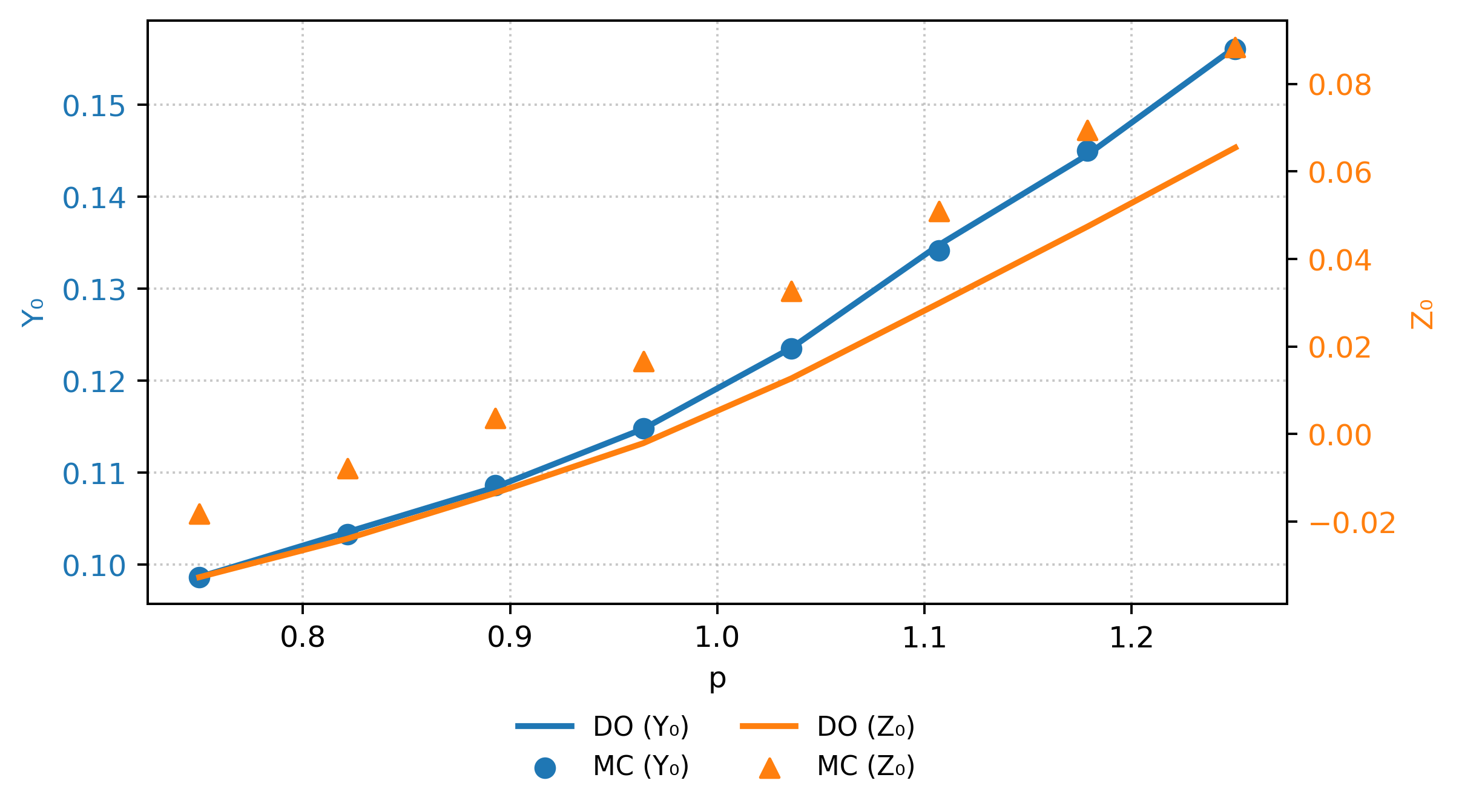}
        \caption{Example 1 -- Power Lookback Floating Strike Put.}
    \end{subfigure}
    \caption{}
\end{figure}

\begin{figure}[H]
    \centering
    \begin{subfigure}{7.0cm}
        \centering
        \includegraphics[width=7.0cm]{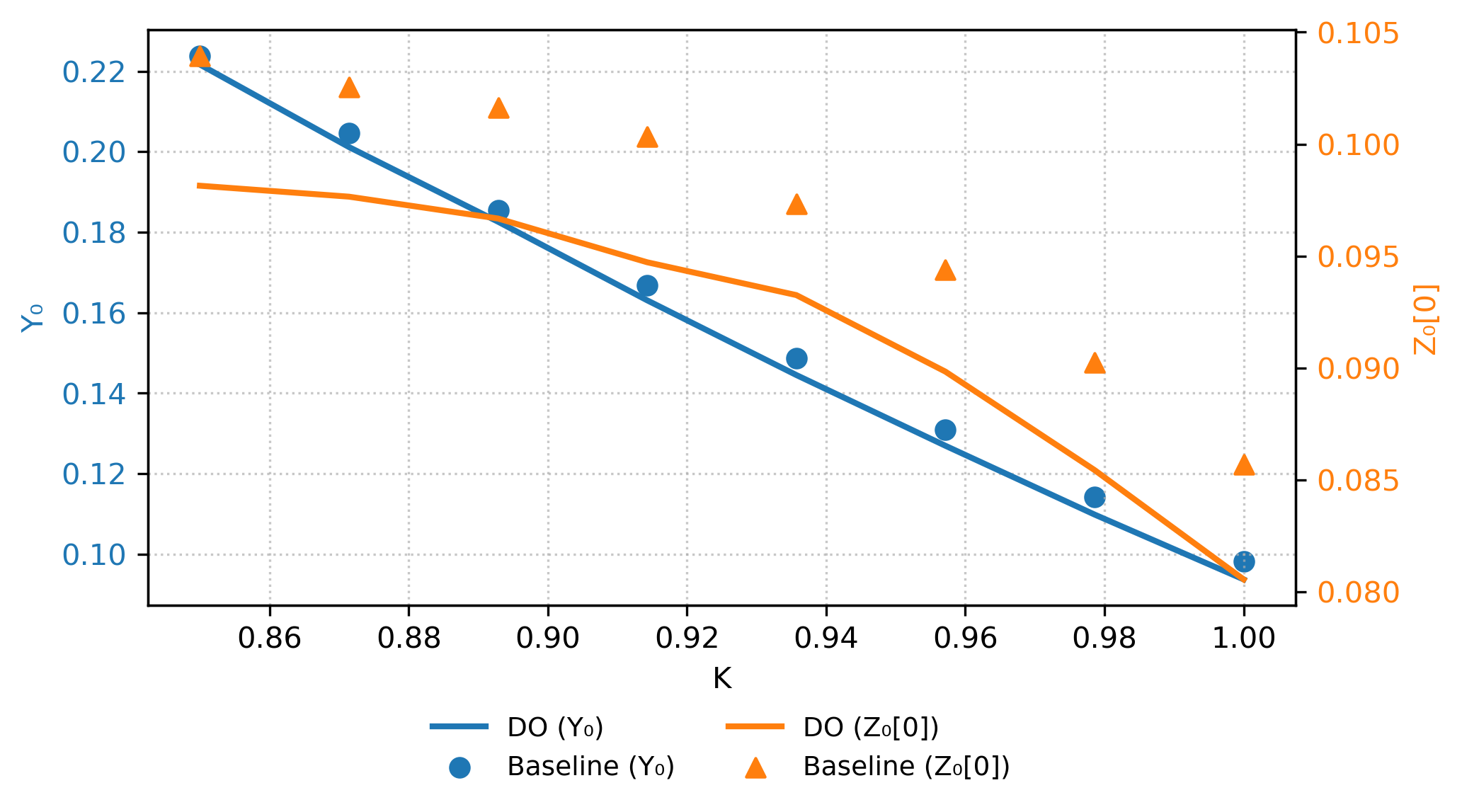}
        \caption{Example 2 -- Asian Call -- Min.}
    \end{subfigure}
    \hspace{0.06cm}
    \begin{subfigure}{7.0cm}
        \centering
        \includegraphics[width=7.0cm]{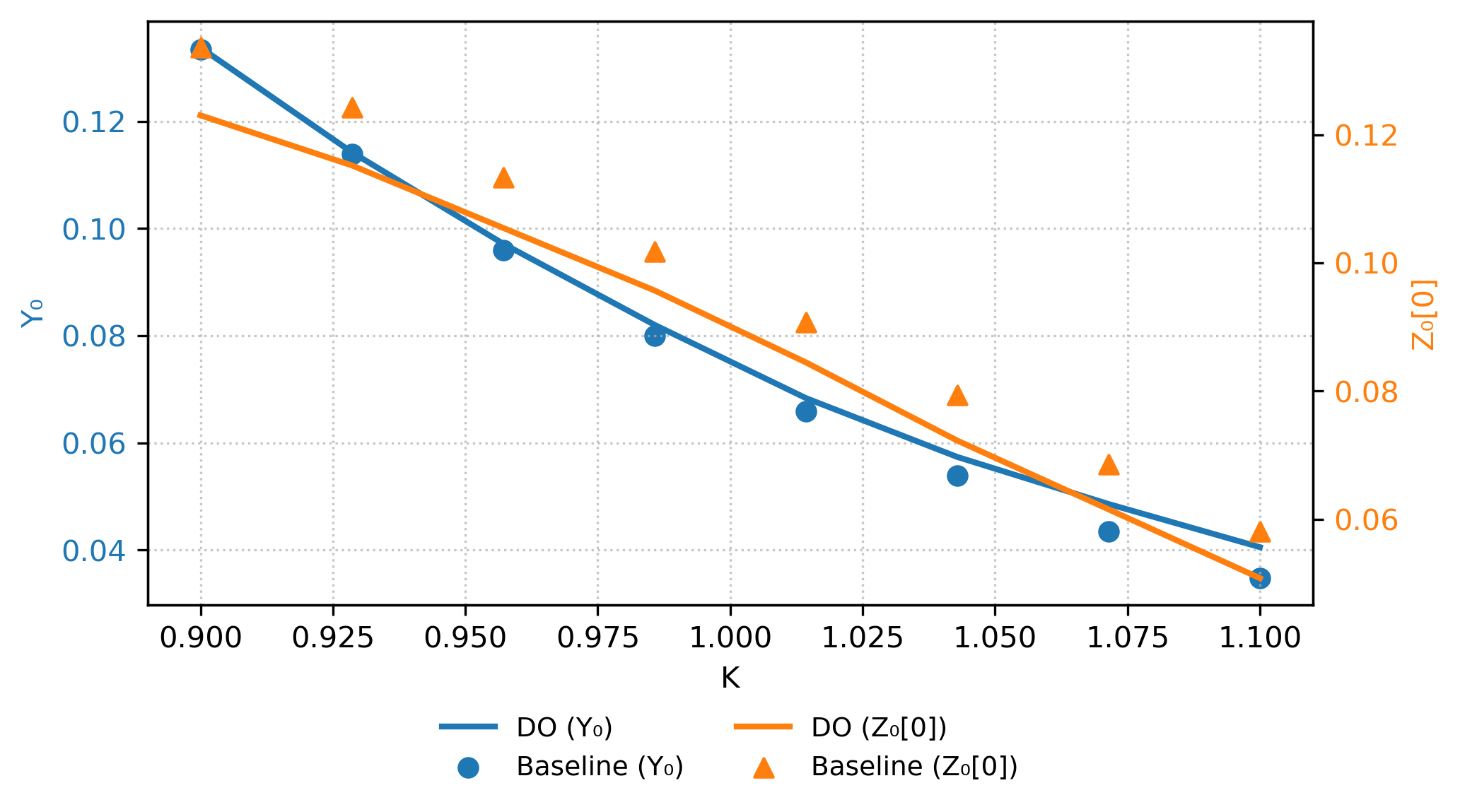}
        \caption{Example 2 -- Asian Call -- Ratio, $i = 1$.}
    \end{subfigure}
    \caption{}
\end{figure}

\begin{figure}[H]
    \centering
    \begin{subfigure}{7.0cm}
        \centering
        \includegraphics[width=7.0cm]{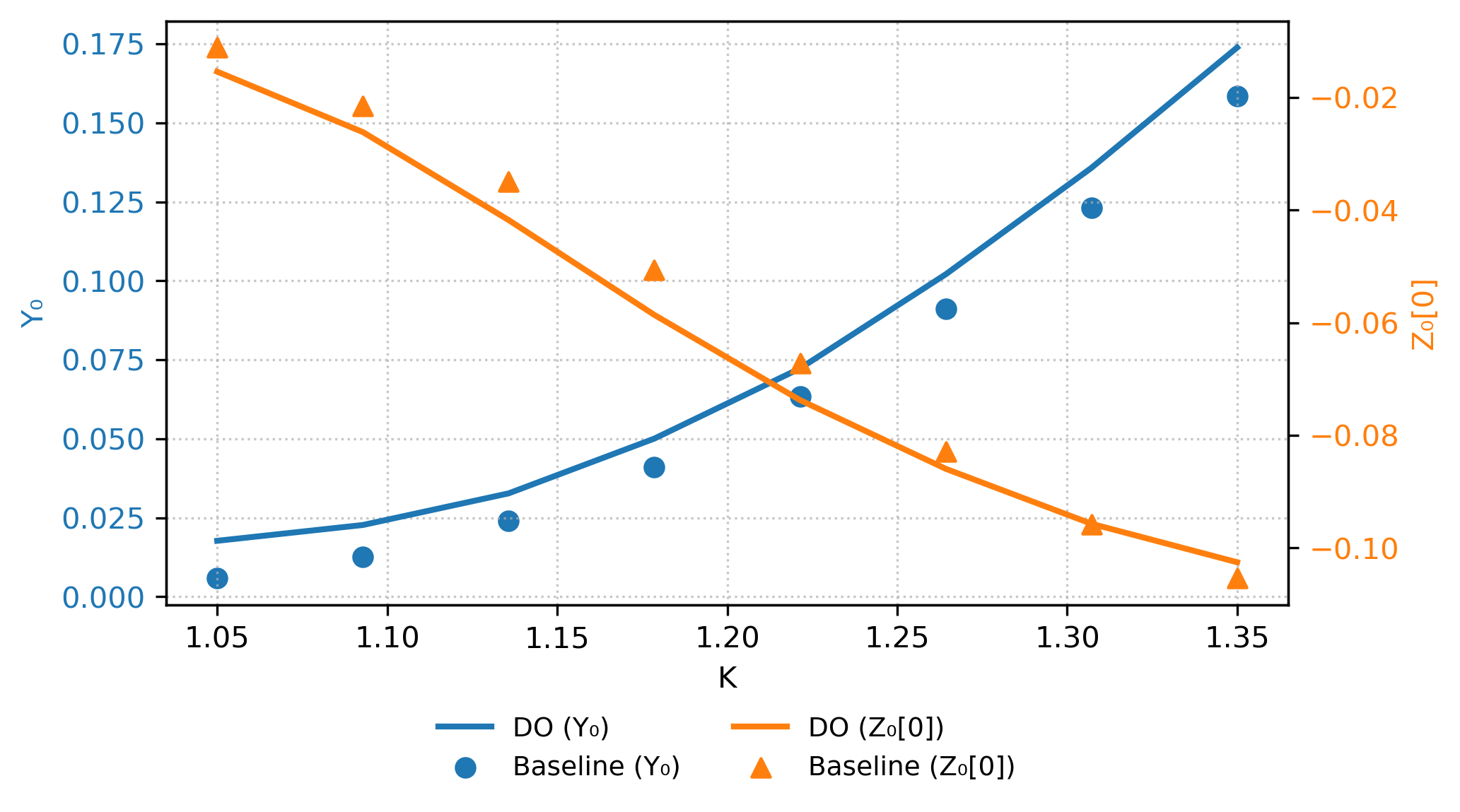}
        \caption{Example 2 -- Asian Put -- Max.}
    \end{subfigure}
    \hspace{0.06cm}
    \begin{subfigure}{7.0cm}
        \centering
        \includegraphics[width=7.0cm]{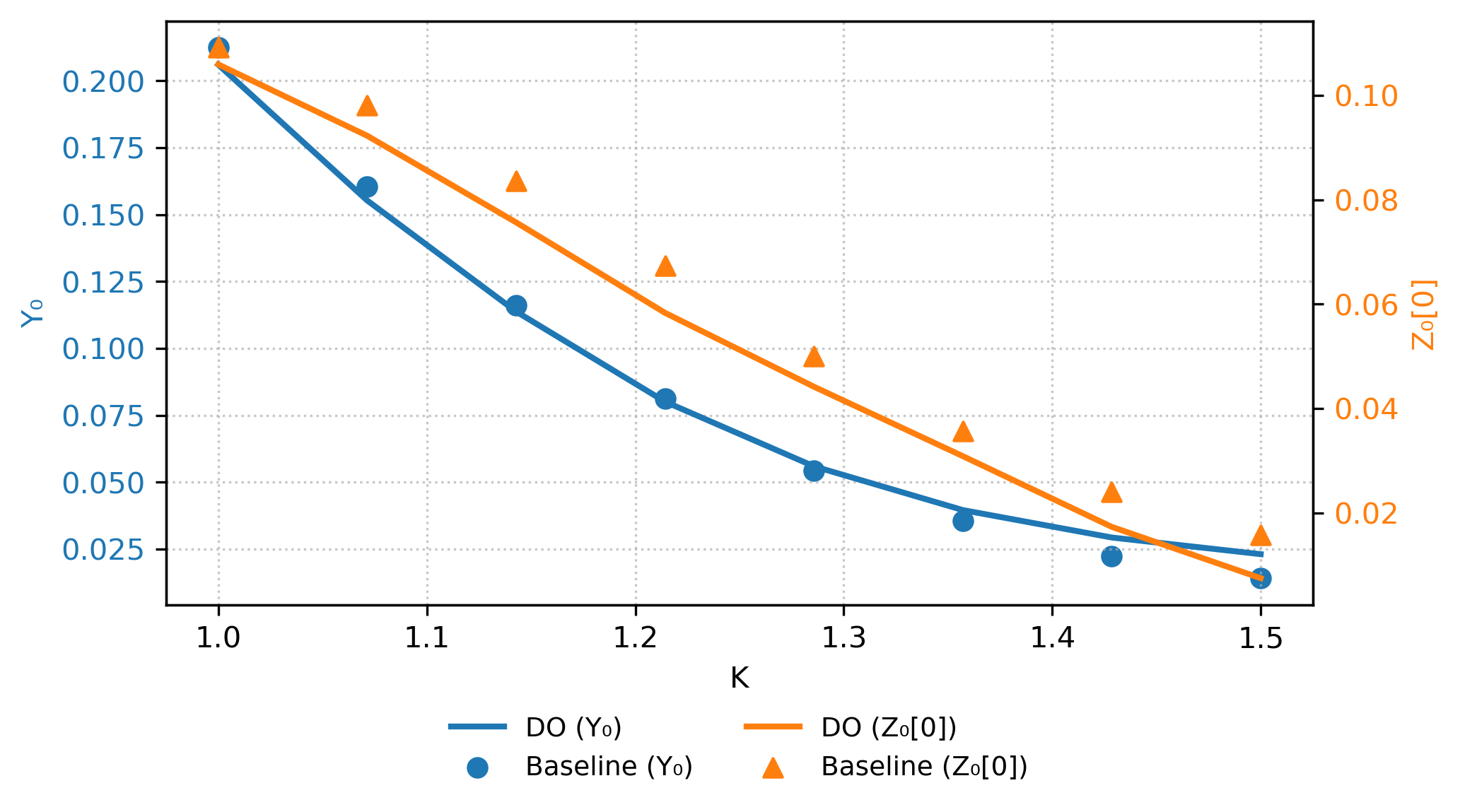}
        \caption{Example 2 -- Call -- Max.}
    \end{subfigure}
    \caption{}
\end{figure}

\begin{figure}[H]
    \centering
    \begin{subfigure}{7.0cm}
        \centering
        \includegraphics[width=7.0cm]{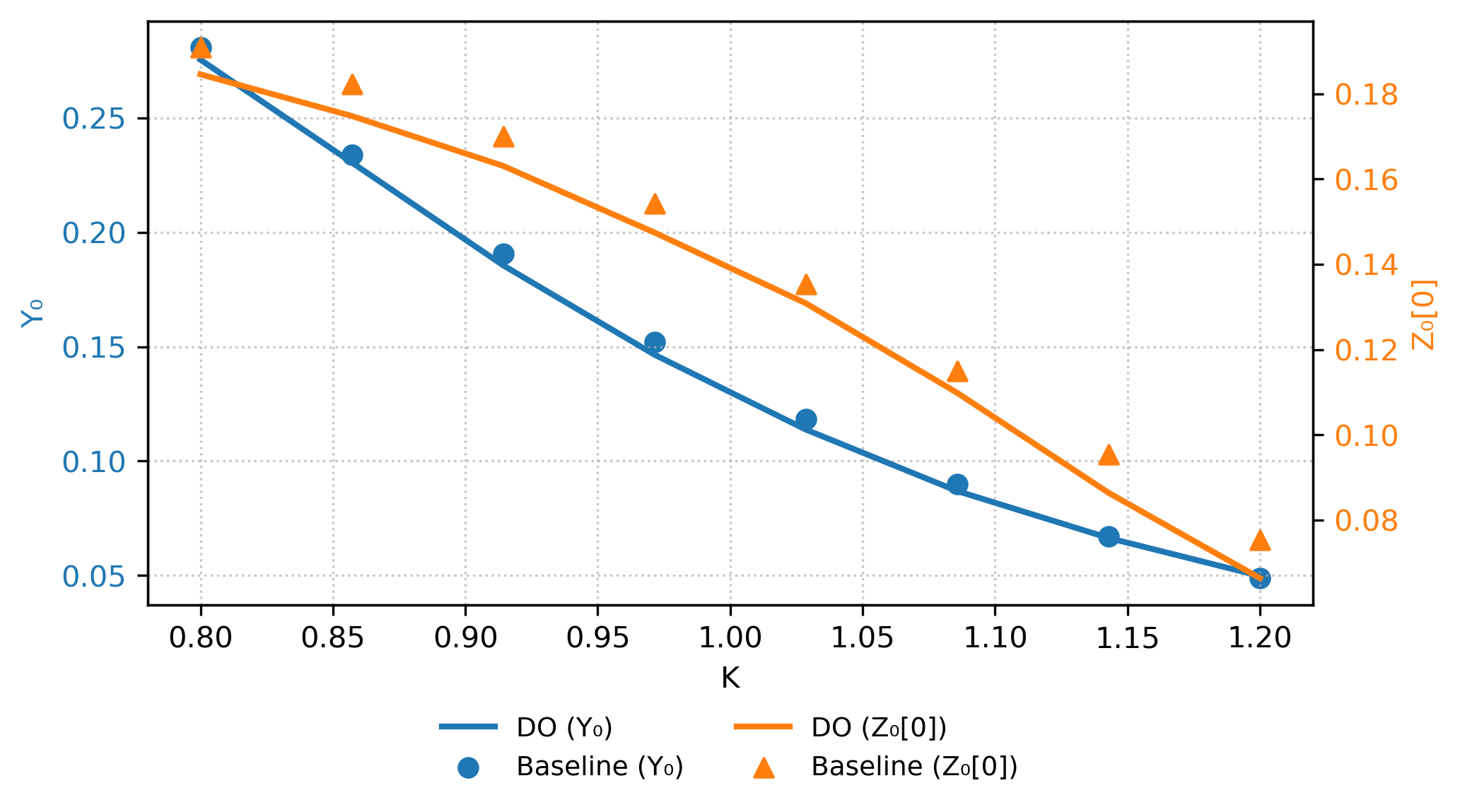}
        \caption{Example 2 -- Call -- Single, $i = 1$.}
    \end{subfigure}
    \hspace{0.06cm}
    \begin{subfigure}{7.0cm}
        \centering
        \includegraphics[width=7.0cm]{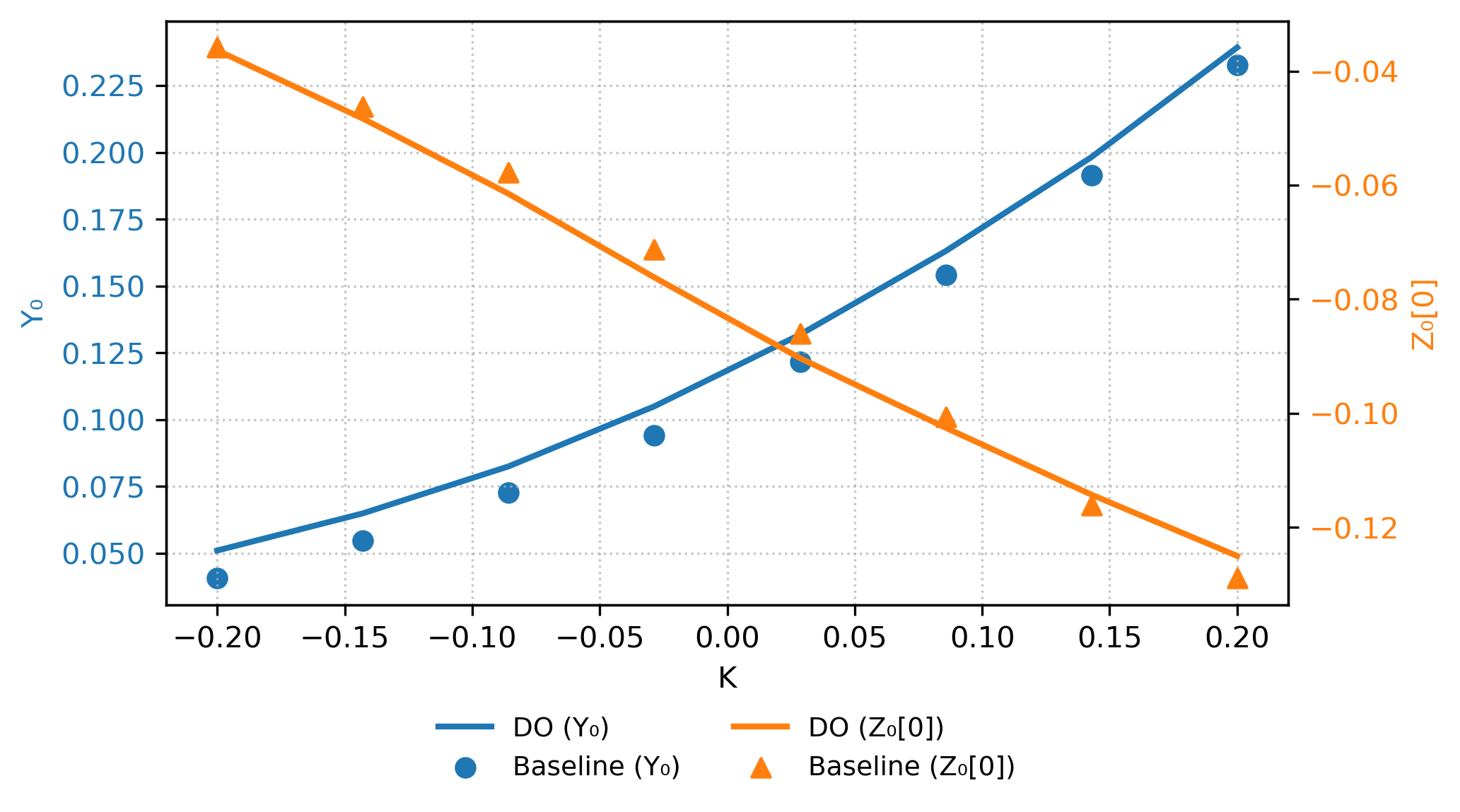}
        \caption{Example 2 -- Put -- Spread, $i=1$.}
    \end{subfigure}
    \caption{}
\end{figure}

\begin{figure}[H]
    \centering
    \begin{subfigure}{7.0cm}
        \centering
        \includegraphics[width=7.0cm]{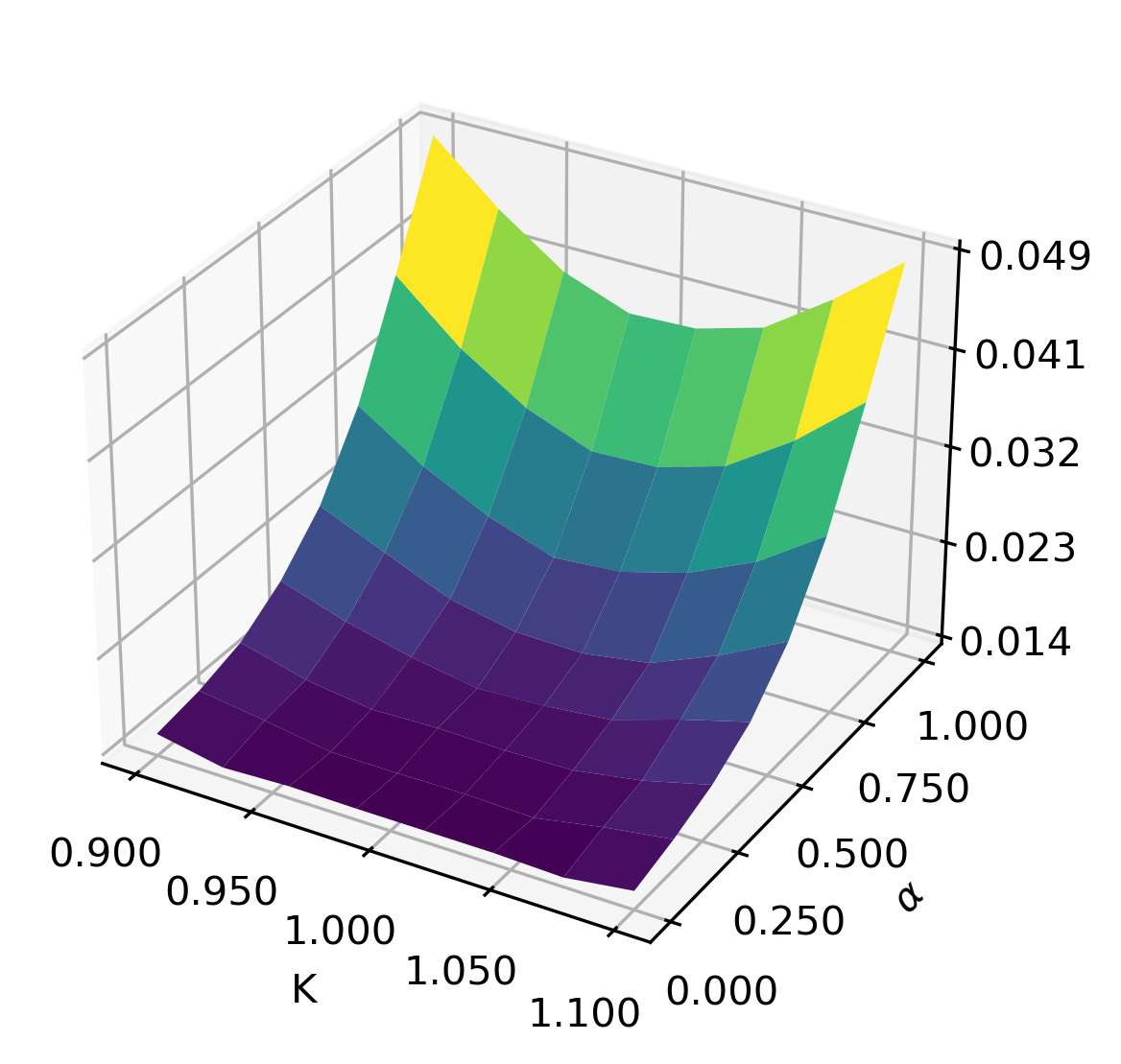}
        \caption{Example 2 -- Asian Put -- Basket weighted, Deep Operator BSDE approximation of $\mathcal{Y}_0$.}
    \end{subfigure}
    \hspace{0.06cm}
    \begin{subfigure}{7.0cm}
        \centering
        \includegraphics[width=7.0cm]{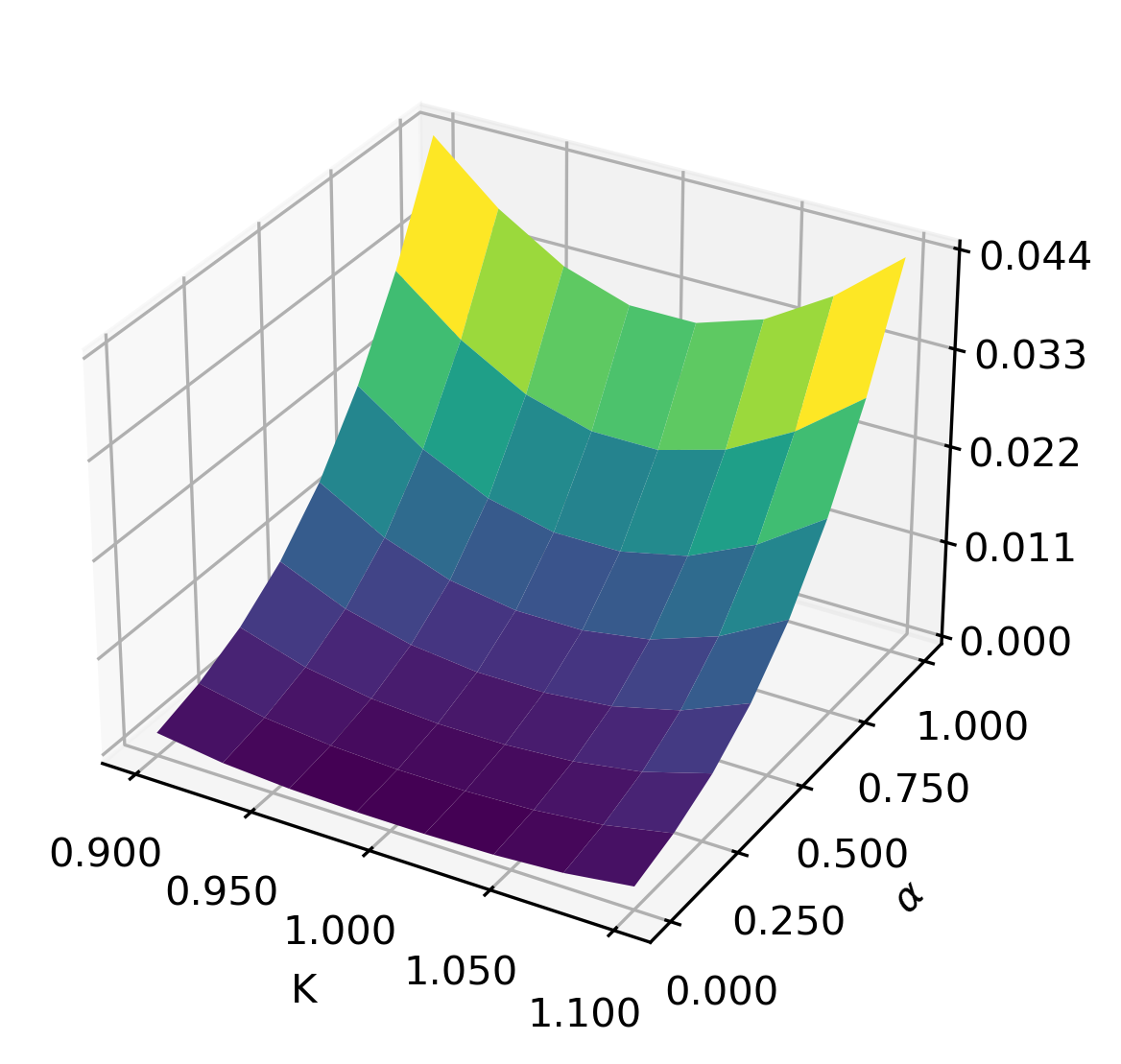}
        \caption{Example 2 -- Asian Put -- Basket weighted, baseline approximation of $\mathcal{Y}_0$.}
    \end{subfigure}
    \caption{}
\end{figure}

\begin{figure}[H]
    \centering
    \begin{subfigure}{7.0cm}
        \centering
        \includegraphics[width=7.0cm]{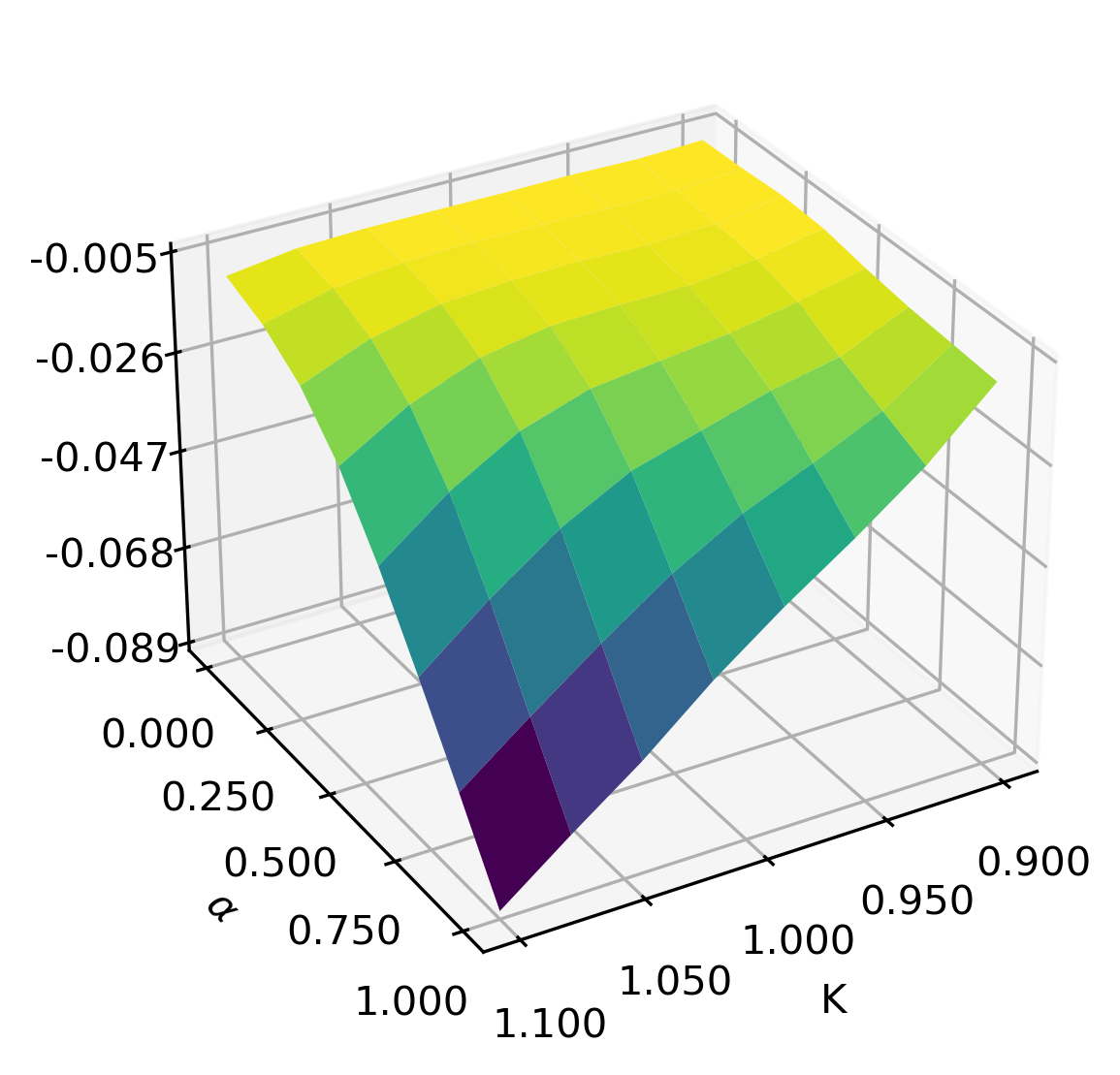}
        \caption{Example 2 -- Asian Put -- Basket weighted, Deep Operator BSDE approximation of the first coordinate of $\mathcal{Z}_0$.}
    \end{subfigure}
    \hspace{0.06cm}
    \begin{subfigure}{7.0cm}
        \centering
        \includegraphics[width=7.0cm]{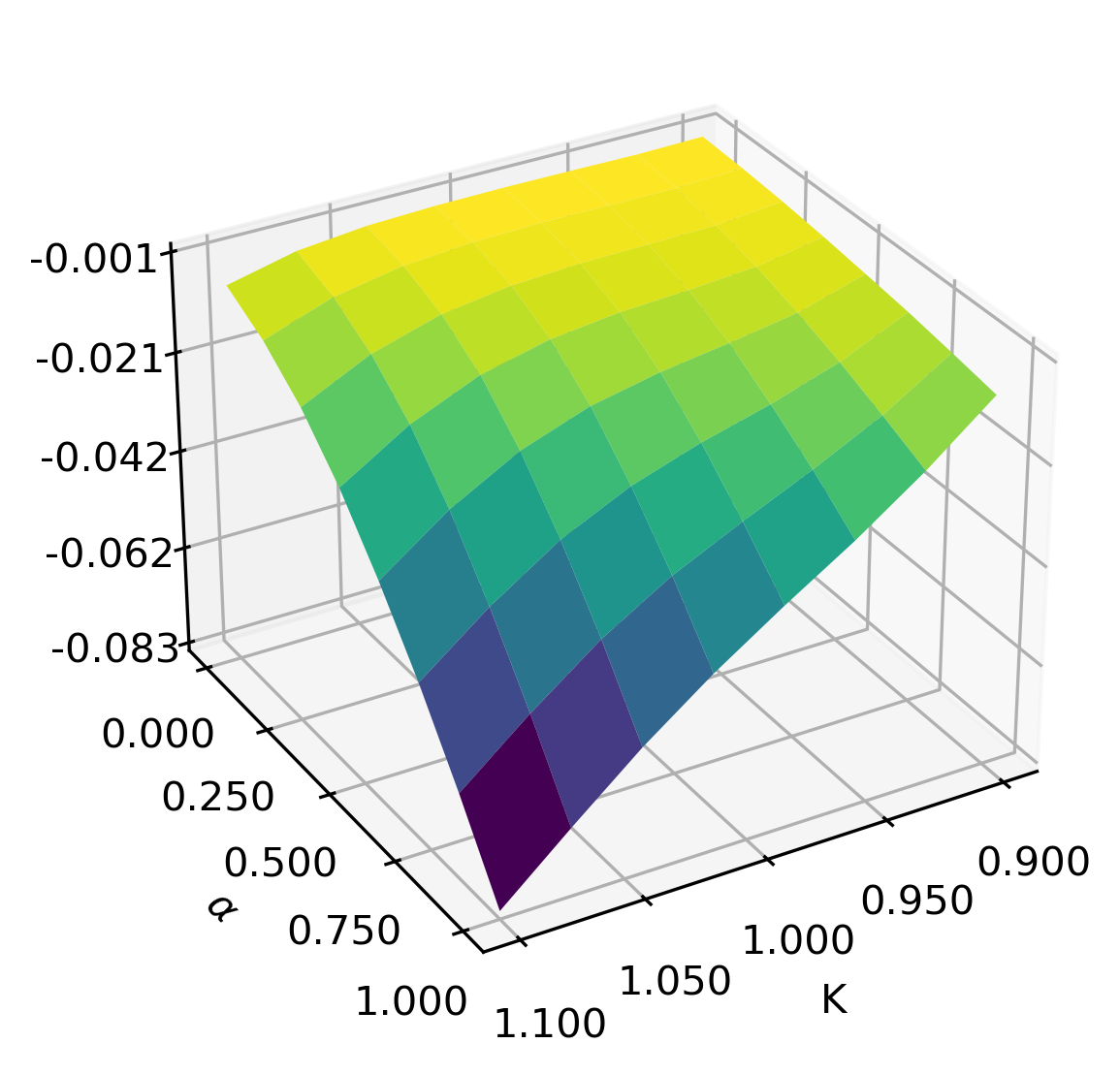}
        \caption{Example 2 -- Asian Put -- Basket weighted, baseline approximation of the first coordinate of $\mathcal{Z}_0$.}
    \end{subfigure}
    \caption{}
\end{figure}


\bibliographystyle{imsart-nameyear} 
\bibliography{bibliography}       

\end{document}